\documentclass[reqno]{amsart}
\usepackage[T1]{fontenc}
\usepackage{lmodern}
\usepackage{amsmath,amssymb,mathtools,mathrsfs}
\usepackage{booktabs,longtable,array}
\usepackage{enumitem}
\usepackage{xcolor}
\usepackage{microtype}
\usepackage[unicode,pdfusetitle,hidelinks]{hyperref}
\usepackage{listings}
\allowdisplaybreaks[2]
\newtheorem{theorem}{Theorem}[section]
\newtheorem{proposition}[theorem]{Proposition}
\newtheorem{lemma}[theorem]{Lemma}
\newtheorem{corollary}[theorem]{Corollary}
\theoremstyle{definition}
\newtheorem{definition}[theorem]{Definition}
\theoremstyle{remark}
\newtheorem{remark}[theorem]{Remark}
\newtheorem{example}[theorem]{Example}
\newcommand{\C}{\mathbb C}
\newcommand{\R}{\mathbb R}
\newcommand{\N}{\mathbb N}
\newcommand{\D}{\mathbb D}
\newcommand{\HH}{\mathbb H}
\DeclareMathOperator{\Tr}{Tr}
\DeclareMathOperator{\Id}{Id}

\DeclareMathOperator{\diag}{diag}

\title[Stability generators, spectral gaps, and computation]{Generators of stability-preserving semigroups,\ spectral gaps, and classical ground-state computation}
\author{Dongsheng Wei}
\address{Independent Researcher}
\email{dongshengwei2025@icloud.com}
\date{}
\hypersetup{pdftitle={Generators of stability-preserving semigroups, spectral gaps, and classical ground-state computation},pdfauthor={Dongsheng Wei}}
\begin{document}
\begin{abstract}
We classify the generators of stability-preserving semigroups on
polynomial spaces with bounded coordinate degrees. The generators have
differential order at most two, with principal coefficients characterized
by low-degree nonnegativity conditions. In disk coordinates, positive
degree damping gives strict zero-freeness and a sharp spectral gap,
extending the Hermitian gap of Bravyi, Gosset, Liu, and Wong to complex
generators. The resulting analytic domain yields deterministic classical
algorithms for ground energies and stable product-state queries for
bounded-degree Suzuki--Fisher Hamiltonians with bounded local strength
and fixed positive fields.
Field perturbation gives zero-field energy-value approximation schemes
for bounded-degree unweighted EPR and bipartite Quantum MaxCut.
We also prove that, up to scalars, the Hermitian qubit Hamiltonians whose
full Gibbs tensors are Lee--Yang at every temperature in a fixed basis
are precisely the edgewise phase-rotated Suzuki--Fisher family.
With strictly positive longitudinal fields, these tensors are zero-free
on a polydisk of radius greater than one at each positive inverse
temperature. These two statements answer questions of Wong, Bravyi,
Gosset, and Liu.
Degree-preserving, coefficient-positive stability semigroups have concave
sector growth rates, yielding token-graph concavity. We also give a
counterexample to a proposed explicit ground-state radius.
\end{abstract}
\maketitle
\vspace{-\baselineskip}
\tableofcontents
\section{Introduction}
\label{intro:section}

The classification of linear operators preserving polynomial zeros goes
back to P\'olya and Schur's theorem on multiplier sequences
\cite{PolyaSchur14}. Its multivariate form concerns \emph{stable}
polynomials: a nonzero polynomial is stable if it does not vanish when
all its variables lie in the open upper half-plane
\(\HH=\{z:\Im z>0\}\). Borcea and Br\"and\'en characterized stability
preservers by their symbols, both on spaces with prescribed coordinate
degrees and in the Weyl algebra \cite{BB09,BB10,BB10II}; here we
study the corresponding continuous evolutions, classifying the
operators \(A\) for which \(e^{tA}\) preserves stability for every
\(t\ge0\) and deriving spectral and computational consequences.
In probability, Borcea, Br\"and\'en, and Liggett \cite{BBL09}
proved that symmetric exclusion preserves strongly Rayleigh measures,
whose probability generating polynomials are real stable.

Requiring preservation for every positive time is especially restrictive
at polynomials with colliding roots: any admissible generator has
canonical differential order at most two, irrespective of the number
of variables or their degrees. For real operators, the classification
reduces to nonpositivity of the second-order coefficients on real lines
and planes, together with the algebraic identities imposed by the degree
bounds. Complex generators have the same real second-order part, while
their imaginary part is a sum of one-variable quadratic vector fields
and a scalar.

\subsection{The classification}

Fix \(\kappa\in\N_{>0}^d\), and write
\[
 V_\kappa^{\mathbb F}
 =\{f\in\mathbb F[z_1,\ldots,z_d]:\deg_{z_i}f\le\kappa_i\},
 \qquad \mathbb F\in\{\R,\C\}.
\]
A real polynomial is real stable if it is stable. An operator preserves
stability if it sends each stable polynomial to a stable polynomial or
to zero. We call preservation \emph{complete} when it also holds after
tensoring the operator with the identity on any finite polynomial space
in additional variables. Real operators are extended complex linearly
when acting on \(V_\kappa^{\C}\).

Every endomorphism of \(V_\kappa^{\mathbb F}\) has a unique representation
\[
 A=\sum_{\alpha\le\kappa}P_\alpha(z)\partial^\alpha;
\]
we call this its canonical representation. The coefficient polynomials
may have degrees greater than \(\kappa\). For a univariate polynomial
\(q\) of degree at most two, put
\begin{equation}\label{gen:eq-J}
 J_{i,\kappa_i}(q)
 =q(z_i)\partial_i-\frac{\kappa_i}{2}q'(z_i).
\end{equation}
These operators preserve \(V_\kappa\). Our main classification is as
follows.

\begin{theorem}[Real generators]
\label{gen:main-real}
Let \(A\in\operatorname{End}_{\R}(V_\kappa^\R)\).  The following conditions
are equivalent.
\begin{enumerate}
\item For every \(t\geq0\), \(e^{tA}\) preserves real stability.
\item For every \(t\geq0\), its complexification preserves complex stability,
also after adjoining arbitrary inert polynomial variables.
\item The canonical representation of \(A\) has order at most two, and its
second-order coefficients satisfy
\begin{align}
 P_{2e_i}(s)&\leq0
 &&\text{for all }s\in\R,\quad \kappa_i\geq2,
 \label{gen:eq-real-unary}\\
 P_{e_i+e_j}(s,t)&\leq0
 &&\text{for all }(s,t)\in\R^2,\quad i<j.
 \label{gen:eq-real-pair}
\end{align}
\end{enumerate}
In the third condition, box preservation forces \(P_{2e_i}\) to depend only
on \(z_i\), with degree at most four, and \(P_{e_i+e_j}\) to depend only on
\(z_i,z_j\), with bidegree at most \((2,2)\).  There are no further
inequalities on the zeroth- and first-order coefficients beyond the
identities already forced by \(A(V_\kappa^\R)\subseteq V_\kappa^\R\).
\end{theorem}

The coefficient \(P_{e_i+e_j}\) in this statement is the full coefficient
of \(\partial_i\partial_j\), with \(i<j\).  Thus a convention writing the
second-order part as \(\sum_{i,j}Q_{ij}\partial_i\partial_j\), with \(Q\)
symmetric, has \(P_{e_i+e_j}=2Q_{ij}\).

\begin{theorem}[Complex generators]
\label{gen:main-complex}
Let \(A\in\operatorname{End}_{\C}(V_\kappa^\C)\).  The following conditions
are equivalent.
\begin{enumerate}
\item For every \(t\geq0\), \(e^{tA}\) preserves complex stability.
\item For every \(t\geq0\), \(e^{tA}\) preserves complex stability after
adjoining arbitrary inert polynomial variables.
\item There are a real generator \(R\) satisfying
Theorem~\ref{gen:main-real}, a real scalar \(c\), and real polynomials
\(q_i\) of degree at most two, nonnegative on \(\R\), such that
\begin{equation}
 A=R+i\left(c\Id+\sum_{i=1}^dJ_{i,\kappa_i}(q_i)\right).
 \label{gen:eq-complex-decomposition}
\end{equation}
\end{enumerate}
Equivalently, the canonical representation has order at most two, all its
second-order coefficients are real and satisfy
\eqref{gen:eq-real-unary}--\eqref{gen:eq-real-pair}, and every
\(\Im P_{e_i}\) is nonnegative on \(\R^d\).  In this formulation box
preservation forces
\(\Im P_{e_i}(z)=q_i(z_i)\) with \(\deg q_i\leq2\), and forces the
corresponding imaginary zeroth-order term in
\eqref{gen:eq-complex-decomposition}.
\end{theorem}

Theorems~\ref{gen:main-real} and~\ref{gen:main-complex} are proved in
Section~\ref{gen:section}. At a multiple root, the first two moments
of a root cluster control the first variation of a stable polynomial;
applying this observation to product polynomials gives the order bound
and the signs of the principal coefficients. This contact principle is
related to the polynomial-abscissa results of Burke and Overton
\cite[Theorem~1.2 and Lemma~1.4]{BurkeOverton01}, and we give a direct
moment proof adapted to the operator problem. For sufficiency, a
sum-of-squares decomposition of each negative principal coefficient
reduces the construction to rank-one generators on two multiaffine
variables, which polarization then realizes on \(V_\kappa\).

The one-variable background includes the theory of differential
operators preserving real zeros; see \cite{CravenCsordas94}.
In the multiaffine setting, Purbhoo \cite[Proposition~4.5]{Purbhoo18}
constructs preserving generators from matrices with nonnegative
off-diagonal entries. Theorems~\ref{gen:main-real} and
\ref{gen:main-complex} give necessary and sufficient conditions on
arbitrary finite coordinate boxes. Their proof also gives Gram
certificates of orders at most four and an intertwining with a
multiaffine preserving generator (Corollaries~\ref{gen:gram-cones}
and~\ref{gen:slot-lift}). For one-site and two-site binary inputs with
Gaussian-rational entries, the classification yields a polynomial-time
recognition algorithm (Corollary~\ref{gen:recognition}). It tests the
sum of the supplied local matrices and allows cancellation between them.

\subsection{Damping and the spectrum}

A Cayley transformation carries the classification to the unit disk
\(\D\). Write \(E_i=z_i\partial_i\). If \(A_0\) generates a
semigroup preserving disk stability, consider
\begin{equation}\label{intro:damping}
 A=A_0-2\sum_i\mu_iE_i,\qquad \mu_i>0.
\end{equation}
The added term contracts the polynomial variables towards zero.
We prove that, for every \(t>0\), the kernel of \(e^{tA}\) is nonzero
on the closed unit polydisk, and hence on a larger open polydisk.
For each \(t_0>0\), a common enlarged radius works for all \(t\ge t_0\)
(Theorem~\ref{spec:strict}). The complex-cone theorem of Rugh
\cite{Rugh10} then gives a simple eigenvalue \(a_\star\) of largest
real part; we prove the sharp bound
\begin{equation}\label{intro:gap}
 \Re a_\star-\Re a\ge2\min_i\mu_i
 \qquad(a\in\operatorname{spec}(A),\ a\ne a_\star),
\end{equation}
independent of the number of variables and their degree bounds
(Theorem~\ref{spec:gap}).

Bravyi, Gosset, Liu, and Wong \cite{BGLW26} proved the Hermitian
\(2\mu\) gap for Suzuki--Fisher Hamiltonians. We extend their
radius-to-spectrum argument to arbitrary complex stability-preserving
generators. These spectral results use stability preservation directly;
the classification determines the admissible generators and makes the
hypothesis decidable from local input. An oscillation estimate for
quotients of polynomials also gives contraction along time-dependent
flows and a resolvent bound in an adapted norm
(Section~\ref{spec:oscillation-subsection}).

\subsection{Lee--Yang Hamiltonians and classical computation}

The connection with quantum Hamiltonians comes from the Lee--Yang
property of a matrix kernel. In a fixed computational basis, associate
to a qubit matrix \(M\) the polynomial in independent variable
groups \(z,w\),
\[
 F_M(z,w)=\sum_{a,b\in\{0,1\}^n}M_{ab}z^aw^b.
\]
For invertible \(M\),
nonvanishing of \(F_M\) on \(\D^{2n}\) is equivalent to preservation
of disk stability by its coefficient operator. Thus the all-temperature
Lee--Yang property of \(e^{-\beta H}\) is a generator question.

We prove that a Hermitian qubit Hamiltonian has this property if and
only if, up to a scalar, it belongs to the phase-rotated Suzuki--Fisher
family
\begin{equation}\label{intro:sf}
 \begin{split}
 H={}&-\sum_{i<j}\left(J^{zz}_{ij}Z_iZ_j+
       \sum_{a,b\in\{x,y\}}J^{ab}_{ij}\sigma_i^a\sigma_j^b\right)
       -\sum_i(h_i^xX_i+h_i^yY_i+h_i^zZ_i),\\
 &J^{zz}_{ij}\ge\|J^\perp_{ij}\|_{\rm op},\qquad h_i^z\ge0,
 \end{split}
\end{equation}
where \(X,Y,Z\) are the Pauli matrices and
\(J^\perp_{ij}=(J^{ab}_{ij})_{a,b\in\{x,y\}}\) is real.
Sufficiency, including this transverse norm condition, is classical
\cite{SuzukiFisher71,AsanoJapanese70,Dunlop79}. Theorem~\ref{eq:sf}
proves the converse among all Hermitian qubit Hamiltonians; in particular,
no interaction on three or more sites is possible. Ruelle's converse
\cite[Theorem~9]{Ruelle10} concerns classical coefficientwise
Boltzmann polynomials. For the full Gibbs tensor in the fixed basis,
our converse answers the question about Hamiltonians beyond the
Suzuki--Fisher family in \cite[Section~6]{WBG26}. Appendix~\ref{spin:section}
extends the matrix-kernel classification to arbitrary finite spins.

When every longitudinal field is positive,
Theorem~\ref{eq:strict-field} gives a strict Gibbs radius at each
\(\beta>0\), answering the question in \cite[Section~5.3]{WBG26}.
The ground state is unique and its polynomial also has a radius
greater than one, which may depend on the Hamiltonian. An unweighted
six-vertex tree disproves the stronger explicit EPR ground-state radius
proposed in \cite[Conjecture~1]{WBG26}; its rational certificate is in
Appendix~\ref{cx:section}.

For bounded degree, bounded local strength, and longitudinal fields
bounded below by a fixed \(\mu>0\), we give deterministic classical
algorithms for the ground energy and relative amplitudes and
probabilities against stable product states, together with a classical
sampler for adaptive equatorial measurements. With Gaussian-rational
local matrices as input, the running times are polynomial in the
system size, input bit length, and inverse requested error, with
degree depending on the fixed bounds
(Corollaries~\ref{alg:sf} and~\ref{alg:sf-state}).
Adding a field of size proportional to \(\varepsilon\) gives an
additive \(\varepsilon n\) energy-value approximation at zero field.
For unweighted bounded-degree EPR and bipartite Quantum MaxCut this
becomes a multiplicative value approximation scheme
(Corollary~\ref{alg:qmc}).

The algorithm is proved first for the larger class
\eqref{intro:damping} with complex one-site and two-site input.
The gap selects a holomorphic principal eigenvalue in a half-plane
of complex damping parameters, which a reciprocal change of variable
joins to the disk of efficiently computable perturbation coefficients
of Bravyi, DiVincenzo, and Loss \cite{BDL08}. An explicit conformal
map then reduces evaluation to logarithmically many coefficients,
and nonvanishing product tests of the principal eigenvectors admit
the same continuation. Connected-coefficient computation and
zero-free interpolation are established methods
\cite{PatelRegts17,HMS20,YYZ22,WildAlhambra23}; the spectral results provide
the global parameter domain used here.

Classical randomized algorithms include Bravyi and Gosset's treatment
of quantum ferromagnets \cite{BG17} and Rayudu and Takahashi's
operator-loop method for a class of stoquastic XY Hamiltonians
\cite{RT25}. For Heisenberg antiferromagnets, Takahashi, Slezak, and
Crosson \cite{TSC24} proved rapid mixing of ground-state quantum Monte
Carlo on bipartite graphs with one part of bounded size and polynomially
bounded ratios of nonzero couplings.
The general classical computation question is raised in \cite{BGLW26};
related interpolation approaches appear in \cite{WongTalk23,WongThesis26}.
The amplitude results of \cite[Section~4]{WBG26} use a fixed spatial
zero-free margin, whereas our algorithms take the local Hamiltonian as
input and use continuation in the field.

For the EPR problem introduced by King \cite{King23}, Ju and Nagda
\cite{JuNagda25} give a \((1+\sqrt5)/4\approx0.809\) approximation,
and Apte and coauthors \cite{ALMPSS26} give a \(0.8395\) approximation.
These algorithms apply to general nonnegatively weighted graphs and
construct descriptions of approximate states, whereas our zero-field
result gives a value approximation scheme on unweighted graphs of
bounded degree, with polynomial running time for each fixed relative
error. The computation problem in \cite{MarwahaSud26} asks for the finer
scale of inverse-polynomial absolute accuracy.

\enlargethispage{-3\baselineskip}
\subsection{Further consequences and organization}

For coefficient-positive, degree-preserving stability semigroups,
Newton's inequalities compare the growth rates of homogeneous sectors
(Theorem~\ref{sector:growth}), giving the token-graph application of
Section~\ref{sector:section}.
If \(A_k,D_k\) are the adjacency and weighted degree matrices of the
\(k\)-token graph of a nonnegatively weighted graph, then
\[
 k\longmapsto\lambda_{\max}(A_k+aD_k)
\]
is concave for \(-1\le a\le1\), and this interval is sharp.
The sectors \(k\) and \(n-k\) are related by complementation and
have the same largest eigenvalue; together with concavity, this
symmetry implies monotonicity up to half filling.
The cases \(a=0,1\) resolve the adjacency and signless-Laplacian
monotonicity conjectures of Apte, Parekh, and Sud
\cite[Conjectures~5--6]{APS26}.
The same weighted token-graph concavity and monotonicity were obtained
independently by Jiang \cite[Theorem~1.1]{Jiang26}. Both proofs use
log-concavity under a preserving semigroup and pass to spectral growth
rates as time tends to infinity; Jiang works with Lorentzian polynomials,
while we use real stability.

There is also a distinction between continuous paths of preservers and
products of preserving flows. Every invertible multiaffine disk-stability
preserver can be joined to the identity through preservers, but a
nontrivial coordinate permutation is outside the closure of products of
admissible one-parameter flows (Appendix~\ref{glob:section}).

Section~\ref{pre:section} sets out the polynomial and kernel conventions.
Section~\ref{gen:section} proves the classification and recognition
theorems. Sections~\ref{spec:section} and~\ref{eq:section} develop the
spectral results and their Hamiltonian consequences, which provide the
input to the algorithms in Section~\ref{alg:section}.
Section~\ref{sector:section} treats token-graph spectra.
The appendices contain the higher-spin extension, bit-complexity proofs,
static-preserver results, examples concerning the spectral estimates,
and the EPR certificate.

\section{Stable polynomials and operator kernels}\label{pre:section}

\subsection{Regions, degrees, and preservation}
Write $\HH=\{z\in\C:\Im z>0\}$ and $\D_r=\{z\in\C:|z|<r\}$,
with $\D=\D_1$.
\begin{definition}\label{pre:stability}
A nonzero polynomial is \emph{$\Omega$-stable} if it
has no zero when all its variables lie in $\Omega$.
For a product of regions, stability refers to the product domain.
A polynomial is \emph{real stable} if its coefficients are real and it is
$\HH$-stable. In one variable this is equivalent to real-rootedness.
A \emph{positive stable polynomial} has nonnegative real coefficients and
is real stable. Positivity throughout refers to coefficients, unless an
operator on a Hilbert space is explicitly said to be positive semidefinite.
An operator \emph{preserves} one of these classes when it sends every
member to another member or to zero.
\end{definition}

For $\kappa\in\N^n$, where $\N=\{0,1,2,\ldots\}$, put
\[
 V_\kappa=\{f\in\C[z_1,\ldots,z_n]:\deg_{z_i}f\leq\kappa_i\},
 \qquad
 \binom{\kappa}{\alpha}=\prod_i\binom{\kappa_i}{\alpha_i}.
\]
We omit coordinates with \(\kappa_i=0\), obtaining the constant space
when none remain, and interpret every degree bound as an upper bound,
so that constants and polynomials of smaller degree lie in the same
space. Real operators act on complex polynomials by complex-linear
extension.
We use multi-index notation: \(\alpha!=\prod_i\alpha_i!\),
\(|\alpha|=\sum_i\alpha_i\), and
\((\alpha)_\beta=\prod_i\alpha_i!/(\alpha_i-\beta_i)!\) for
\(\beta\le\alpha\). The canonical differential representation is
\begin{equation}\label{pre:normal-order}
 A=\sum_{\alpha\le\kappa}P_\alpha(z)\partial^\alpha.
\end{equation}
Although the coefficient polynomials may have degrees exceeding
\(\kappa\), the derivative indices are bounded by \(\kappa\): the
\emph{order} of \(A\) is the largest \(|\alpha|\) for which
\(P_\alpha\ne0\), and \(\partial_i^2\) is absent at a binary
coordinate \(\kappa_i=1\). Throughout the paper,
\(P_{e_i+e_j}\) denotes the full coefficient of
\(\partial_i\partial_j\) for \(i<j\).

\begin{lemma}
\label{gen:canonical}
Every linear endomorphism of \(V_\kappa^\mathbb F\) has a unique
representation \eqref{pre:normal-order}.
\end{lemma}

\begin{proof}
The coefficients are determined recursively on the monomial basis:
the constant monomial gives \(P_0=A1\), and once \(P_\beta\) is
known for \(|\beta|<|\alpha|\), evaluation on \(z^\alpha\) gives
\[
 P_\alpha
 =\frac1{\alpha!}
 \left(Az^\alpha-
 \sum_{\substack{\beta\leq\alpha\\\beta\neq\alpha}}
 P_\beta(z)(\alpha)_\beta z^{\alpha-\beta}\right).
\]
Since a derivative acts nontrivially on \(z^\alpha\) only when its
index is bounded by \(\alpha\), this recursion both constructs the
representation and makes its coefficients unique.
\end{proof}

All spaces in this paper are finite dimensional. We call \(A\) a
generator preserving a class when \(e^{tA}\) preserves that class for
every \(t\ge0\).

\begin{definition}\label{pre:complete}
An operator is \emph{completely stable} if its tensor product with the
identity in arbitrarily many extra polynomial variables preserves the
same complex stable class, for every finite choice of the extra degree
bounds.
\end{definition}
We use this tensor-extension property in the polarization argument of
Section~\ref{gen:section}.

\subsection{Polarization and the disk kernel}
The normalized polarization of $z_i^a$ into $\kappa_i$ binary slots is
\[
 z_i^a\longmapsto
 \binom{\kappa_i}{a}^{-1}
 e_a(z_{i1},\ldots,z_{i\kappa_i}).
\]
This symmetric polynomial specializes to $z_i^a$ on the diagonal.
The Grace--Walsh--Szeg\H{o} coincidence theorem shows that polarization
preserves stability in the open disks and half-planes used here,
while the reverse assertion follows by diagonal specialization
\cite{BB09,BB10II,BB09Circular}.
\label{pre:polarization}

For $T:V_\kappa\to V_\lambda$ define its disk kernel by
\begin{equation}\label{pre:kernel}
 K_T(z,w)=T_z\prod_i(1+z_iw_i)^{\kappa_i}
 =\sum_{\alpha\leq\kappa}\binom{\kappa}{\alpha}
       (Tz^\alpha)(z)w^\alpha.
\end{equation}
In the first expression, $T$ acts on the input variables $z_i$;
the resulting kernel has degree bounds $\lambda$ in its output variables
and $\kappa$ in $w$, with
$K_{\Id}(z,w)=\prod_i(1+z_iw_i)^{\kappa_i}$ for the identity.

For two polynomials of degree at most $\kappa$ in paired variables define
their bilinear contraction by
\begin{equation}\label{pre:contraction}
 [f(u),g(v)]_\kappa
   =\sum_{\alpha\leq\kappa}
      \binom{\kappa}{\alpha}^{-1} f_\alpha g_\alpha,
 \qquad
 f(u)=\sum_\alpha f_\alpha u^\alpha,\quad
 g(v)=\sum_\alpha g_\alpha v^\alpha .
\end{equation}
With this bilinear pairing, matrix composition becomes contraction of
the corresponding kernels:
\[
 K_{ST}(z,w)=[K_S(z,u),K_T(v,w)]_\kappa,
\]
where $\kappa$ is the vector of degree bounds for the contracted variables.

\begin{lemma}\label{pre:grace}
Suppose $f(z,u)$ and $g(v,w)$ are stable in products of disks and
$\deg_{u_i}f,\deg_{v_i}g\leq\kappa_i$. If their paired radii
$r_i,s_i$ satisfy $r_is_i\geq1$, then their contraction
\eqref{pre:contraction} is stable in the remaining disks or is identically
zero. If all $r_is_i>1$, contraction at any point in the remaining
disks is nonzero.
\end{lemma}
\begin{proof}
For a multiaffine polynomial
$F(u,v)=a+bu+cv+duv$ nonvanishing on $\D_r\times\D_s$,
the diagonal restriction $F(rt,st)$ is nonvanishing for $|t|<1$.
Its constant coefficient is nonzero, and when $d\ne0$ the product
of its roots has modulus at least one, giving $|d|rs\le|a|$;
the same inequality holds when $d=0$. In particular $a+d\ne0$
if $rs>1$, so replacing $a+bu+cv+duv$ by $a+d$ preserves
nonvanishing after any uncontracted variables have been fixed in
their disks. This replacement can be iterated over pairs in a
general multiaffine polynomial, without requiring it to factor.

Apply normalized polarization to the paired variables of $f$ and $g$,
and contract corresponding slots in their product by this replacement.
For each multi-index $\alpha$, there are
$\binom{\kappa}{\alpha}$ matching slot subsets, whereas the two
polarizations contribute $\binom{\kappa}{\alpha}^{-2}$; the resulting
coefficient is therefore the one in \eqref{pre:contraction}.
This gives strict nonvanishing for all degrees allowed by the bounds,
and is the weighted circular-domain contraction used in
\cite{AsanoJapanese70,BB09,BB10II}.
When some paired radii have product one, replace every paired variable
by $(1-\varepsilon)$ times that variable and apply the strict result.
Letting $\varepsilon\downarrow0$, Hurwitz's theorem gives the
stable-or-zero conclusion throughout the product of the remaining
disks.
\end{proof}

The finite-box symbol theorem \cite{BB09,BB10II} characterizes complex
stability preservers of rank at least two by stability of their symbols.
We will use its disk form for invertible operators; a direct proof in
our kernel convention is given in Lemma~\ref{spec:kernel-criterion}.

The Cayley isomorphism is
\begin{equation}\label{pre:cayley}
 (\mathcal C_\kappa f)(x)
   =\prod_i(x_i+\mathrm i)^{\kappa_i}
      f\left(\frac{x_1-\mathrm i}{x_1+\mathrm i},\ldots,
              \frac{x_n-\mathrm i}{x_n+\mathrm i}\right).
\end{equation}
It is a linear isomorphism between the degree boxes, and $f$ is
$\D$-stable if and only if $\mathcal C_\kappa f$ is $\HH$-stable.
Thus the generator classifications in half-plane coordinates and
the spectral results in disk coordinates concern conjugate
finite-dimensional semigroups.

\subsection{Coherent coordinates and physical tensors}
Equip $V_\kappa$ with the inner product for which
$\binom{\kappa}{\alpha}^{1/2}z^\alpha$ is an orthonormal basis.
Equivalently, identify the standard basis $|\alpha\rangle$ of
$\bigotimes_i\C^{\kappa_i+1}$ with
$\binom{\kappa}{\alpha}^{1/2}z^\alpha$. The polynomial of a vector
$\psi$ is
\[
 f_\psi(z)=\sum_{\alpha\leq\kappa}
          \binom{\kappa}{\alpha}^{1/2}\psi_\alpha z^\alpha.
\]
An operator $M$ on this Hilbert space has coherent tensor
\begin{equation}\label{pre:coherent}
 F_M(z,w)=\sum_{\alpha,\beta\leq\kappa}
 \binom{\kappa}{\alpha}^{1/2}\binom{\kappa}{\beta}^{1/2}
 M_{\alpha\beta}z^\alpha w^\beta .
\end{equation}
Thus $F_M$ is the kernel \eqref{pre:kernel} of the operator in coherent
coordinates, with independent complex variables for the row and column
indices. A positive semidefinite density matrix $\rho$ has the
\emph{Lee--Yang property} when $F_\rho$ is $\D$-stable.

We use the Pauli matrices
\[
 X=\begin{pmatrix}0&1\\1&0\end{pmatrix},\qquad
 Y=\begin{pmatrix}0&-\mathrm i\\\mathrm i&0\end{pmatrix},\qquad
 Z=\begin{pmatrix}1&0\\0&-1\end{pmatrix}.
\]
Thus $E_i=z_i\partial_i$ represents $|1\rangle\langle1|_i$
and $Z_i=I-2E_i$ on a qubit.
For a coordinate of degree bound $\kappa_i$, the operator
$E_i=z_i\partial_i$ has eigenvalues
$0,\ldots,\kappa_i$. A \emph{principal eigenvalue} below means an
algebraically simple eigenvalue whose real part is strictly greater
than the real parts of all other eigenvalues; for a discrete operator
it means the corresponding strict modulus dominance.

\section{Proof of the generator classification}
\label{gen:section}

The necessity in Theorems~\ref{gen:main-real} and~\ref{gen:main-complex}
comes from root collisions: the first variation of a stable product
polynomial isolates a canonical coefficient and forces it to vanish
above order two, or to have the required sign at order two.
For sufficiency, sums of squares express the allowed principal
coefficients in terms of generators on two multiaffine variables;
we then polarize to obtain the prescribed coordinate degrees, leaving
only first-order M\"obius generators.

Throughout this section the stability domain is \(\HH^d\), and
\(\kappa_i\ge1\). We use the canonical representation
\begin{equation}
 A=\sum_{\alpha\le\kappa}P_\alpha(z)\partial^\alpha.
 \label{gen:eq-canonical}
\end{equation}
For later use, the three one-site operators associated with
\(q=1,z_i,z_i^2\) in \eqref{gen:eq-J} are
\[
 J_i^-=\partial_i,\qquad
 J_i^0=z_i\partial_i-\frac{\kappa_i}{2},\qquad
 J_i^+=z_i^2\partial_i-\kappa_i z_i.
\]
They span a copy of \(\mathfrak{sl}_2\) at each site.

\subsection{Root clusters and infinitesimal necessity}
\label{gen:subsec-roots}

The local factor associated with a separated root cluster is
differentiable even when the ambient polynomial changes degree.
The following moment argument extracts its possible first variations;
after rotation of the half-plane, it is related to the polynomial-abscissa
tangent theorem of Burke and Overton
\cite[Theorem~1.2 and Lemma~1.4]{BurkeOverton01}.

\begin{lemma}
\label{gen:cluster-moments}
Let \(p_t(s)\), \(t\geq0\), be a curve of complex polynomials of uniformly
bounded degree whose coefficients are right differentiable at zero.
Suppose
\[
 p_0(s)=s^m g(s),\qquad m\geq1,\qquad g(0)\neq0,
\]
and that \(p_t\) has no zero in \(\HH\) for all sufficiently small \(t\).
Let \(C_t\) be the monic factor consisting of the \(m\) roots converging
to zero.  Its coefficients are right differentiable, and
\begin{equation}
 \left.\frac{d}{dt}C_t(s)\right|_{t=0}
 =c_1s^{m-1}+c_2s^{m-2},
 \qquad
 \Im c_1\geq0,\quad c_2\in\R,\quad c_2\leq0.
 \label{gen:eq-cluster-variation}
\end{equation}
When \(m=1\), the \(c_2\) term is absent.  If every \(p_t\) is real
rooted, then \(c_1\) is real as well.
\end{lemma}

\begin{proof}
Choose a small circle enclosing zero and no other root of \(p_0\).
By coefficient continuity and Rouch\'e's theorem, it encloses \(m\)
roots \(\lambda_1(t),\ldots,\lambda_m(t)\) for small \(t\), whose
power sums can be written without choosing differentiable root labels:
\[
 M_k(t)=\sum_{j=1}^m\lambda_j(t)^k
       =\frac1{2\pi i}\int s^k\frac{p_t'(s)}{p_t(s)}\,ds,
       \qquad k\geq1.
\]
The contour integrands are right differentiable and their denominators
are bounded away from zero, so the power sums are right differentiable
with \(M_k(t)=O(t)\). Newton's identities transfer this differentiability
to the coefficients of \(C_t\).

Write \(\lambda_j=a_j+ib_j\).  The roots lie in the closed lower
half-plane, so \(b_j\leq0\), and
\[
 \sum_j|b_j|=-\Im M_1(t)=O(t),\qquad
 \sum_jb_j^2=O(t^2).
\]
Combining these estimates with
\[
 \sum_ja_j^2=\Re M_2(t)+\sum_jb_j^2=O(t)
\]
gives
\[
 \Im M_2(t)=2\sum_ja_jb_j=O(t^{3/2})=o(t),
\]
and, for \(k\geq3\),
\[
 |M_k(t)|
 \leq\sum_j|\lambda_j(t)|^k
 \leq\left(\sum_j|\lambda_j(t)|^2\right)^{k/2}
 =O(t^{k/2})=o(t).
\]
The second inequality follows by bounding each \(|\lambda_j(t)|^{k-2}\)
by \(\bigl(\sum_{\ell=1}^m|\lambda_\ell(t)|^2\bigr)^{(k-2)/2}\).
Writing
\(C_t=s^m+\gamma_1(t)s^{m-1}+\cdots+\gamma_m(t)\), we can differentiate
Newton's identities at zero, where all \(\gamma_j\) and \(M_j\)
vanish, to obtain \(\gamma_j'(0)=-M_j'(0)/j\). Thus the coefficients
below \(s^{m-2}\) have zero first variation, while
\(\Im\gamma_1'(0)\ge0\) follows from \(\Im M_1(t)\le0\).
By the estimate on \(\Im M_2\), \(M_2'(0)\) is real, and
\[
 \Re M_2'(0)=\lim_{t\downarrow0}
       \frac{\sum_ja_j(t)^2-\sum_jb_j(t)^2}{t}\geq0.
\]
The required sign of \(c_2=\gamma_2'(0)\) follows; when the roots
are real, \(M_1\) and hence \(c_1\) are real as well.
\end{proof}

\begin{proposition}
\label{gen:contact-necessity}
Suppose \(A\in\operatorname{End}_{\C}(V_\kappa^\C)\) generates a
complex stability-preserving semigroup.  Then \(A\) has canonical
order at most two, every second-order canonical coefficient is real
and nonpositive on \(\R^d\), and
\(\Im P_{e_i}\geq0\) on \(\R^d\).
If \(A\) is real and its semigroup is only assumed to preserve real
stability, the order bound and the nonpositivity of its
second-order coefficients still follow.
\end{proposition}

\begin{proof}
For \(0\neq\alpha\leq\kappa\) and \(y\in\R^d\), use the stable
product
\[
 f_{\alpha,y}(z)=\prod_i(z_i-y_i)^{\alpha_i}.
\]
In the canonical representation, every derivative at \(z=y\)
vanishes except for \(\partial^\alpha\), and therefore
\begin{equation}
 (Af_{\alpha,y})(y)=\alpha!P_\alpha(y).
 \label{gen:eq-contact-isolation}
\end{equation}
For \(v\in\R_{>0}^d\), restriction of \(e^{tA}f_{\alpha,y}\)
to the line \(z=y+sv\) gives a polynomial curve with initial value
\[
 \left(\prod_iv_i^{\alpha_i}\right)s^{|\alpha|}.
\]
This curve has differentiable coefficients and degree at most
\(|\kappa|\). Since \(e^{tA}\) is invertible, the multivariate
output is nonzero. Its restriction is nonzero and has no root in
\(\HH\): for \(s\in\HH\), every coordinate
\(y_i+sv_i\) lies in \(\HH\), where the output cannot vanish.

Lemma~\ref{gen:cluster-moments} therefore applies to the local
factorization \(C_t(s)G_t(s)\), in which the constant first
variation is
\[
 G_0(0)\left.\frac{d}{dt}C_t(0)\right|_{t=0}
\]
because \(C_0(0)=0\), whereas
\(G_0(0)=\prod_iv_i^{\alpha_i}>0\). For \(|\alpha|\ge3\), the
variation is zero by the lemma, so
\eqref{gen:eq-contact-isolation} implies \(P_\alpha(y)=0\);
for \(|\alpha|=2\), the same identity shows that \(P_\alpha(y)\)
is real and nonpositive, and for \(\alpha=e_i\) that
\(\Im P_\alpha(y)\ge0\).
Since \(y\) is arbitrary in \(\R^d\), the higher-order coefficients
vanish identically and each second-order coefficient is a real
polynomial. If the semigroup is real
and preserves real stability, the same restrictions are real rooted,
so the real version of the root-cluster lemma gives the stated order
and sign conditions.
\end{proof}

The order bound thus comes from two estimates on the colliding roots:
the imaginary parts have total size \(O(t)\), and the squared real
parts have total size \(O(t)\), forcing every higher moment to be
\(o(t)\). The dependence of the surviving coefficients on the individual
variables is an algebraic consequence of preserving the degree box.

\subsection{Second-order operators on a degree box}
\label{gen:subsec-box}

The algebraic restrictions imposed by the degree bounds hold over
either \(\R\) or \(\C\), independently of the stability inequalities.

Differential operators on finite polynomial modules are classically
described by enveloping algebras of \(\mathfrak{sl}_2\)
\cite{Turbiner92,Turbiner94}. We need the degree-two part of this
description, including coordinates of degree one.

\begin{proposition}
\label{gen:box-decomposition}
Every box endomorphism of canonical order at most two is a linear
combination of \(\Id\), the operators \(J_i^-,J_i^0,J_i^+\), and their
products of length two, including products at the same site.
Every endomorphism of order at most one has a unique expression
\begin{equation}
 c\Id+\sum_{i=1}^dJ_{i,\kappa_i}(q_i),
 \qquad c\in\mathbb F,\quad q_i\in\mathbb F[z_i]_{\leq2}.
 \label{gen:eq-first-order}
\end{equation}
Consequently the second-order coefficients have the variable dependence
and degree bounds stated in Theorem~\ref{gen:main-real}.  The vector space
of endomorphisms of order at most two has dimension
\begin{equation}
 1+3d+5\#\{i:\kappa_i\geq2\}+9\binom d2.
 \label{gen:eq-order-two-dimension}
\end{equation}
\end{proposition}

\begin{proof}
Write \(p_{\alpha,\beta}\) for the coefficient of \(z^\beta\) in
\(P_\alpha\), and group terms according to the integer monomial shift
\(\nu=\beta-\alpha\).  For input exponent
\(x\in\prod_i\{0,\ldots,\kappa_i\}\), the coefficient of \(z^{x+\nu}\) is
the restriction to this grid of
\begin{equation}
 F_\nu(x)=
 \sum_{\substack{|\alpha|\leq2,\ \alpha\leq\kappa\\
                  \alpha+\nu\geq0}}
 p_{\alpha,\alpha+\nu}(x)_\alpha.
 \label{gen:eq-shift-polynomial}
\end{equation}
Although it was obtained from coefficients on the input grid,
\eqref{gen:eq-shift-polynomial} defines a polynomial for every \(x\),
of total degree at most two and coordinate degrees at most \(\kappa_i\).

When \(\nu_i<0\), the condition \(\alpha_i\ge-\nu_i\) in every
summand makes \(F_\nu\) divisible by \((x_i)_{-\nu_i}\), or makes
the sum empty if \(-\nu_i>\kappa_i\). It follows that \(F_\nu\)
vanishes at any grid point that would give a negative output exponent.

If \(\nu_i>0\), box preservation gives vanishing on every overflow face
\[
 x_i\in\{\kappa_i-\nu_i+1,\ldots,\kappa_i\}
          \cap\{0,\ldots,\kappa_i\}.
\]
Fix one of these values of \(x_i\). At every point of the remaining
tensor grid, \(F_\nu\) vanishes: a negative output coordinate is
covered by the falling-factorial argument, and any other output lies
outside the degree box. The coordinate degree bounds permit
tensor-product interpolation, so the restriction to the entire face
is the zero polynomial and its defining linear factor divides
\(F_\nu\). If \(\nu_i>\kappa_i\), this gives \(\kappa_i+1\)
distinct zeros in \(x_i\), forcing \(F_\nu=0\).

For every remaining nonzero shift polynomial, these factors give
\begin{equation}
 F_\nu=D_\nu G_\nu,\qquad
 D_\nu(x)=\prod_i(x_i)_{\nu_i^-}
                       (\kappa_i-x_i)_{\nu_i^+},
 \qquad
 \deg G_\nu\leq2-|\nu|_1.
 \label{gen:eq-shift-factor}
\end{equation}
Since the displayed linear factors are pairwise coprime, their
product divides \(F_\nu\), and a nonzero shift must satisfy
\(|\nu|_1\le2\).

Put \(E_i=z_i\partial_i\).  For \(\nu=0\), every possible \(F_\nu\)
is realized by a combination of
\(\Id,E_i,E_i^2,E_iE_j\).
For \(\nu=\pm e_i\), the required constant or affine factor \(G_\nu\)
is realized by \(J_i^\pm\) and \(J_i^\pm E_j\).
For \(\nu=\pm2e_i\) use \((J_i^\pm)^2\); these shifts can be nonzero
only if \(\kappa_i\geq2\).
For \(\nu=\pm e_i\pm e_j\), with \(i\neq j\), use the corresponding
products \(J_i^\pm J_j^\pm\).
The sign of the raising coefficient
\(x_i-\kappa_i=-(\kappa_i-x_i)\) is absorbed in the scalar multiplier.
This accounts for every possibility in \eqref{gen:eq-shift-factor}
on the input grid, and hence for every matrix entry of \(A\).

The same argument covers a binary site: \(F_\nu\) then has degree
at most one in \(x_i\), with the corresponding dependencies or zeros
among the listed operator products. At total degree at most one,
the construction leaves only a scalar and the three sitewise
operators in each coordinate. Their independence follows from the
separate raising and lowering shifts and, on the diagonal, from
uniqueness of an affine polynomial on the full tensor grid; hence
\eqref{gen:eq-first-order} is unique.

The first-order coefficients of \(J_i^-,J_i^0,J_i^+\) span
\(1,z_i,z_i^2\).  Products at one site therefore have principal
coefficients spanning all polynomials of degree at most four.
Products at two different sites span all biforms of bidegree at most
\((2,2)\).  At a site with \(\kappa_i\geq2\), these are canonical
second-order coefficients; at a binary site that derivative is absent.
The principal-coefficient map is therefore onto, with kernel the
order-at-most-one space of dimension \(1+3d\). Adding five dimensions
for each admissible one-site principal coefficient and nine for
each pair gives \eqref{gen:eq-order-two-dimension}.
\end{proof}

Propositions~\ref{gen:contact-necessity} and~\ref{gen:box-decomposition}
prove the necessity of the order and principal-sign conditions in
Theorem~\ref{gen:main-real}. We next construct preserving generators
with any prescribed principal coefficients satisfying those conditions.

\subsection{Binary apolar generators and sums of squares}
\label{gen:subsec-binary}

On two multiaffine variables, stability is described by a single
quadratic inequality. For
\(f=f_{00}+f_{10}s+f_{01}t+f_{11}st\in V_{(1,1)}^\R\), put
\[
 \Delta(f)=f_{10}f_{01}-f_{00}f_{11}.
\]
A nonzero real \(f\) is stable if and only if \(\Delta(f)\geq0\).
Indeed, if \(f=a+bs+ct+dst\) and \(b+dt\neq0\), its zero in \(s\)
is \(-(a+ct)/(b+dt)\), whose imaginary part equals
\(-\Delta(f)\Im t/|b+dt|^2\).
The cases in which \(f\) is independent of \(s\) follow directly.

\begin{lemma}
\label{gen:sos}
Every nonnegative real univariate polynomial of degree at most four is a sum
of squares of real polynomials of degree at most two.
Every nonnegative real polynomial of bidegree at most \((2,2)\)
is a sum of squares of real bilinear polynomials.
\end{lemma}

\begin{proof}
A nonzero nonnegative univariate polynomial is a positive constant
times a product of squares of real linear factors and positive
quadratic factors. Each factor is a sum of at most two squares of
linear polynomials, so repeated application of the identity
\[
 (a^2+b^2)(c^2+d^2)=(ac-bd)^2+(ad+bc)^2
\]
gives a sum of squares whose summands have degree at most half
that of the original polynomial, and hence at most two. The zero
polynomial is covered by the empty sum.

For the second assertion let \(R(s,t)\geq0\) and homogenize it to
\[
 \widetilde R(x,y,u)=u^4R(x/u,y/u).
\]
By the bidegree bound, \(\widetilde R\) is a ternary quartic,
nonnegative for \(u\ne0\) and hence everywhere by continuity.
By Hilbert's theorem, which also applies to degenerate nonnegative
forms, we can write it as a sum of squares of real quadratic forms
\cite{PRSS04}.
Its zero coefficients at \(x^4\) and \(y^4\) are sums of squared
coefficients of those forms, so every form lies in the span of
\(xy,xu,yu,u^2\); setting \(u=1\) gives the required bilinear
squares.
\end{proof}

\begin{lemma}
\label{gen:binary-primitive}
For a real bilinear polynomial \(F=a+bs+ct+dst\), define
\[
 \ell_F(f)=d f_{00}-c f_{10}-b f_{01}+a f_{11},
 \qquad
 B_Ff=-F\ell_F(f).
\]
The full mixed canonical coefficient of \(B_F\) is \(-F^2\).
Moreover
\begin{align}
 B_F^2&=2\Delta(F)B_F,
 \label{gen:eq-apolar-square}\\
 \Delta(f+kB_Ff)
 &=\Delta(f)+k\bigl(1+k\Delta(F)\bigr)\ell_F(f)^2
 \qquad(k\in\R).
 \label{gen:eq-apolar-discriminant}
\end{align}
For all \(t\geq0\), \(e^{tB_F}\) preserves real stability.
\end{lemma}

\begin{proof}
For any endomorphism \(B\) of the binary box, its mixed coefficient is
\begin{equation}
 P_B(s,t)=B(st)-sB(t)-tB(s)+stB(1).
 \label{gen:eq-binary-coefficient}
\end{equation}
Substitution of the four values of \(B_F\) yields \(P_{B_F}=-F^2\),
while \(\ell_F(F)=2ad-2bc=-2\Delta(F)\) gives
\eqref{gen:eq-apolar-square}. The remaining identity follows by
expanding \(\Delta(f-kF\ell_F(f))\) and using the polarization
identity \(D\Delta_f[F]=-\ell_F(f)\).

If \(D_F=\Delta(F)\), equation~\eqref{gen:eq-apolar-square} gives
\[
 e^{tB_F}=\Id+k_tB_F,\qquad
 k_t=
 \begin{cases}
 (e^{2tD_F}-1)/(2D_F),&D_F\neq0,\\
 t,&D_F=0.
 \end{cases}
\]
For \(t\geq0\), \(k_t\geq0\) and
\[
 1+k_tD_F=\frac{1+e^{2tD_F}}2>0.
\]
It follows from \eqref{gen:eq-apolar-discriminant} that the
exponential preserves \(\Delta(f)\ge0\); since it is invertible,
the output is nonzero whenever the input is nonzero. The formulas remain
valid when \(F=0\) or \(D_F=0\).
\end{proof}

To place this construction inside a larger set of variables, we need
complete complex preservation. We deduce it from the symbol theorem
using the real semigroup just constructed.

\begin{lemma}
\label{gen:binary-completeness}
Let \(B\) be a real endomorphism of \(V_{(1,1)}^\R\).
Then \(e^{tB}\) preserves real stability for all \(t\geq0\) if and
only if its mixed coefficient \(P_B\) is nonpositive on \(\R^2\).
In this case every \(e^{tB}\) is a complete complex stability
preserver.
\end{lemma}

\begin{proof}
For \(f=(s-x)(t-y)\), the identity \(\Delta(f)=0\) and real
preservation imply that the derivative of \(\Delta(e^{\tau B}f)\)
at zero is nonnegative. By \eqref{gen:eq-binary-coefficient} this
derivative is \(-P_B(x,y)\), which proves necessity.

Conversely, Lemma~\ref{gen:sos} gives
\(-P_B=\sum_\ell F_\ell^2\).
The operator \(B-\sum_\ell B_{F_\ell}\) has zero mixed coefficient
and hence has canonical order at most one.
Proposition~\ref{gen:box-decomposition} writes it as a real scalar
plus two sitewise operators \(J_{i,1}(q_i)\), with real quadratic
\(q_i\), without a sign restriction on these two quadratics.

The real sitewise flows preserve complex stability for every real
time. If \(q(z)=q_0+q_1z+q_2z^2\), let
\[
 M_\tau=\exp\left[
 \tau\begin{pmatrix}q_1/2&q_0\\-q_2&-q_1/2\end{pmatrix}
 \right]
 =\begin{pmatrix}a_\tau&b_\tau\\c_\tau&d_\tau\end{pmatrix}.
\]
Since \(M_\tau\) is real and has determinant one, its associated
polynomial action
\begin{equation}
 U_\tau f(z)
 =(c_\tau z+d_\tau)^\kappa
 f\left(\frac{a_\tau z+b_\tau}{c_\tau z+d_\tau}\right)
 \label{gen:eq-real-mobius}
\end{equation}
has infinitesimal generator \(J_\kappa(q)\) and preserves stability:
the denominator has no zero in \(\HH\), and the quotient has
imaginary part \(\Im z/|c_\tau z+d_\tau|^2>0\).
After cancellation of the denominator on each monomial,
\eqref{gen:eq-real-mobius} is polynomial for every real \(\tau\),
including when other variables are held inert.

Combining these flows with Lemma~\ref{gen:binary-primitive}, the
finite-dimensional Lie product formula expresses \(e^{\tau B}\)
as a coefficientwise limit of real-stability preservers; Hurwitz's
theorem and invertibility of the limiting exponential then give
real stability preservation.

It remains to justify complex preservation of the apolar factors
and of the resulting binary semigroup.  We use the finite-box
classification of Borcea and Br\"and\'en
\cite[Theorems~1.1 and~1.2]{BB09}.
A real stability preserver of rank greater than two has a stable
plus symbol or a stable reflected-minus symbol.  For the binary
box these are
\[
 G_T^+(s,t;u,v)=T\bigl((s+u)(t+v)\bigr),\qquad
 G_T^-(s,t;u,v)=T\bigl((s-u)(t-v)\bigr),
\]
where \(T\) acts only in \(s,t\).
Every \(e^{\tau B}\) has rank four, while at \(\tau=0\) the
minus symbol \((s-u)(t-v)\) vanishes at \(s=u\in\HH\).
Because the stable-or-zero cone is closed in this coefficient space,
the minus symbol remains outside it for sufficiently small
nonnegative \(\tau\). The plus symbol is therefore stable, so
we obtain complex preservation from the complex symbol theorem.
Adjoining an identity on a box with degree bounds \(\lambda\) multiplies
this symbol by the stable factor
\(\prod_j(\xi_j+\eta_j)^{\lambda_j}\), so the same argument
applies on every enlarged box. Subdividing any positive time into
sufficiently small equal pieces extends complete preservation to
all times.
\end{proof}

\subsection{Polarization and the proof of the real classification}
\label{gen:subsec-real-proof}

Let \(\mathcal P_\kappa\) be normalized polarization into
\(N=|\kappa|\) binary variables, grouped into \(\kappa_i\) slots at
site \(i\):
\[
 \mathcal P_\kappa(z_i^r)
 =\frac{e_r(z_{i,1},\ldots,z_{i,\kappa_i})}
             {\binom{\kappa_i}{r}}.
\]
Let \(\mathcal D_\kappa\) identify all slots belonging to each site.
The two maps are inverse between \(V_\kappa^\C\) and the subspace
symmetric within each group of slots.
Normalized polarization preserves upper-half-plane stability, and
diagonalization preserves it in the reverse direction
\cite[Proposition~2.4]{BB09}, with the normalization fixed in the preliminaries.

Differentiation of elementary symmetric polynomials and diagonal
specialization give the identities
\begin{align}
 \mathcal D_\kappa\partial_{i,a}\mathcal P_\kappa f
 &=\frac{\partial_i f}{\kappa_i},
 \label{gen:eq-polar-one}\\
 \mathcal D_\kappa\partial_{i,a}\partial_{j,b}
                          \mathcal P_\kappa f
 &=\frac{\partial_i\partial_j f}{\kappa_i\kappa_j},
 &&i\neq j,
 \label{gen:eq-polar-cross}\\
 \mathcal D_\kappa\partial_{i,a}\partial_{i,b}
                          \mathcal P_\kappa f
 &=\frac{\partial_i^2f}{\kappa_i(\kappa_i-1)},
 &&a\neq b.
 \label{gen:eq-polar-same}
\end{align}
The last identity applies when \(\kappa_i\ge2\); its denominator
accounts for the distinction between ordered derivatives and
unordered pairs of slots in the construction below.

\begin{proof}[Proof of Theorem~\ref{gen:main-real}]
For a real operator, complete complex preservation implies real
preservation, while Proposition~\ref{gen:contact-necessity} gives the
necessary coefficient conditions. For sufficiency, we construct a complete
generator with each prescribed principal coefficient and then
account for the first-order remainder.

Write \(a_i=P_{2e_i}\) and \(b_{ij}=P_{e_i+e_j}\).
For each \(i<j\), apply Lemma~\ref{gen:sos} to obtain
\[
 -b_{ij}(s,t)=\sum_\ell F_{ij,\ell}(s,t)^2
\]
with real bilinear \(F_{ij,\ell}\).
The binary operator
\(B_{ij}=\sum_\ell B_{F_{ij,\ell}}\) has mixed coefficient
\(b_{ij}\) and generates a complete preserving semigroup.
On the full slot space, place a copy of \(B_{ij}\) on each pair
consisting of one slot at site \(i\) and one at site \(j\), and
sum these \(\kappa_i\kappa_j\) copies.
By the Lie product formula, the sum is a complete generator;
since it commutes with permutations of the slots at either site,
it leaves the symmetric subspace invariant.

After descent by \(\mathcal D_\kappa\) and \(\mathcal P_\kappa\),
\eqref{gen:eq-polar-cross} identifies the full mixed coefficient as
\(b_{ij}(z_i,z_j)\). No other principal coefficient appears,
since diagonalization preserves the orders of the zeroth- and
first-order terms and the only second derivative in a binary
summand acts on its two selected slots.

For a site with \(\kappa_i\geq2\), write
\[
 -a_i(z)=\sum_\ell p_{i,\ell}(z)^2,\qquad
 p_{i,\ell}(z)=a_\ell z^2+b_\ell z+c_\ell,
\]
using the univariate part of Lemma~\ref{gen:sos}.
Set
\[
 F_{i,\ell}(s,t)=a_\ell st+\frac{b_\ell}{2}(s+t)+c_\ell.
\]
These forms are symmetric in their two variables and satisfy
\(F_{i,\ell}(z,z)=p_{i,\ell}(z)\).
The symmetric binary operator
\(2\sum_\ell B_{F_{i,\ell}}\) has mixed coefficient
\[
 P_i(s,t)=-2\sum_\ell F_{i,\ell}(s,t)^2,
 \qquad P_i(z,z)=2a_i(z).
\]
Both \(F_{i,\ell}\) and its apolar functional are symmetric,
so summing this operator over all unordered pairs of distinct slots
at site \(i\) gives a permutation-invariant complete generator.
By \eqref{gen:eq-polar-same}, its descended principal coefficient is
\[
 \binom{\kappa_i}{2}
 \frac{P_i(z_i,z_i)}{\kappa_i(\kappa_i-1)}
 =a_i(z_i).
\]
Thus the factor two compensates for the unordered slot count;
when \(\kappa_i=1\), there is no same-site principal coefficient
to construct.

Let \(A_{\mathrm{pr}}\) denote the sum of these descended generators.
Their permutation invariance intertwines each exponential on the
symmetric subspace with the exponential of its descent, so
polarization and diagonalization transfer complete preservation
from the slot space to the original box. By construction,
\(A-A_{\mathrm{pr}}\) has order at most one and hence, by
Proposition~\ref{gen:box-decomposition}, is a real scalar plus a
sum of real sitewise \(J_{i,\kappa_i}(q_i)\).
Their flows preserve stability by \eqref{gen:eq-real-mobius},
without sign restrictions on \(q_i\), and the Lie product formula
therefore yields complete preservation for \(e^{tA}\).
On every fixed enlarged box, the limiting exponential is invertible,
so the zero alternative in Hurwitz's theorem cannot occur for a
nonzero input.
\end{proof}

\subsection{Imaginary inward fields and the complex classification}
\label{gen:subsec-complex-proof}

The remaining complex generators have a particularly explicit
one-site realization.

\begin{lemma}
\label{gen:imaginary-primitive}
Let \(a,b\in\R\) and \(\kappa\geq1\).
The generator \(iJ_\kappa((az+b)^2)\) has the semigroup
\begin{equation}
 (U_tf)(z)=d_t(z)^\kappa f(\phi_t(z)),\qquad
 d_t(z)=1-iat(az+b),\qquad
 \phi_t(z)=\frac{z+ibt(az+b)}{1-iat(az+b)}.
 \label{gen:eq-imaginary-flow}
\end{equation}
For every \(t\geq0\), this is a complete complex stability preserver.
\end{lemma}

\begin{proof}
The corresponding fractional-linear matrix is
\[
 M_t=\Id+it
 \begin{pmatrix}ab&b^2\\-a^2&-ab\end{pmatrix}.
\]
The matrix multiplying \(it\) has square zero, whence
\(M_{t+s}=M_tM_s\) and \(\det M_t=1\). Since the denominators
cancel on every monomial of degree at most \(\kappa\),
\eqref{gen:eq-imaginary-flow} defines a polynomial semigroup
whose derivative at zero is
\[
 \left.\frac{d}{dt}U_tf\right|_{t=0}
 =i\bigl((az+b)^2f'-\kappa a(az+b)f\bigr)
 =iJ_\kappa((az+b)^2)f.
\]

For \(z\in\HH\) and \(t\geq0\), direct multiplication by the
conjugate denominator gives
\begin{equation}
 \Im\phi_t(z)
 =\frac{\Im z+t|az+b|^2}{|1-iat(az+b)|^2}>0.
 \label{gen:eq-imaginary-height}
\end{equation}
For \(a\ne0\) and \(t>0\), the denominator has its only zero
at \(-b/a-i/(a^2t)\), below the real axis; for \(a=0\), it is
one and \(\phi_t(z)=z+ib^2t\). The substitution and its nonvanishing
prefactor therefore preserve stability, also with any additional
variables held inert. At \(t=0\), or when \(a=b=0\), the same
formula gives the identity.
\end{proof}

\begin{proof}[Proof of Theorem~\ref{gen:main-complex}]
For a complex generator, take the entrywise decomposition \(A=R+iS\)
in the ordinary monomial basis. It consists of real box endomorphisms
and commutes with the canonical representation, whose defining
recursion in Lemma~\ref{gen:canonical} is real linear.

By Proposition~\ref{gen:contact-necessity}, \(A\) has order
at most two and its second-order coefficients are real and
nonpositive, so \(R\) satisfies Theorem~\ref{gen:main-real}
and \(S\) has order at most one. The box decomposition in
Proposition~\ref{gen:box-decomposition} therefore gives
\[
 S=c\Id+\sum_iJ_{i,\kappa_i}(q_i),
 \qquad c\in\R,\quad q_i\in\R[z_i]_{\leq2}.
\]
The product-contact test gives
\(q_i=\Im P_{e_i}\ge0\) on \(\R\), establishing the
necessary decomposition \eqref{gen:eq-complex-decomposition}.

Conversely, every real quadratic nonnegative on \(\R\) is a
sum of squares of real linear polynomials.
For \(q(z)=az^2+bz+c\) with \(a>0\), complete the square:
\[
 q(z)=\left(\sqrt a\,z+\frac b{2\sqrt a}\right)^2
            +c-\frac{b^2}{4a}.
\]
The remainder is nonnegative; if \(a=0\), nonnegativity instead
forces \(b=0\) and \(c\ge0\). Lemma~\ref{gen:imaginary-primitive}
and the Lie product formula consequently give complete preservation
for each \(iJ(q_i)\). The real part \(R\) has this property by
Theorem~\ref{gen:main-real}, and \(ic\Id\) generates a nonzero
scalar phase. Applying the product formula to their sum gives a
complete preserver, with zero limits excluded by invertibility on
each finite enlarged box; ordinary preservation follows by taking
no additional variables.
\end{proof}

For the upper-half-plane convention used here, \(i\partial\)
shifts the argument upward and the zeros downward, in agreement
with the sign in \eqref{gen:eq-imaginary-flow}.

\subsection{Certificates and multiaffine realizations}
\label{gen:subsec-certificates}

The proof also expresses membership in the generator cone by finitely
many positive semidefinite matrices, and realizes every generator on
a multiaffine space.

\begin{corollary}
\label{gen:gram-cones}
In Theorem~\ref{gen:main-real}, the principal inequalities are
equivalent to the existence of real positive semidefinite matrices
\(G_i\) of size three and \(G_{ij}\) of size four such that
\begin{align}
 -P_{2e_i}(s)
  &=(1,s,s^2)G_i(1,s,s^2)^{\top},
  &&\kappa_i\geq2,
 \label{gen:eq-gram-unary}\\
 -P_{e_i+e_j}(s,t)
  &=(1,s,t,st)G_{ij}(1,s,t,st)^{\top},
  &&i<j.
 \label{gen:eq-gram-pair}
\end{align}
The extra inequalities in Theorem~\ref{gen:main-complex} have
size-two certificates
\begin{equation}
 q_i(s)=(1,s)K_i(1,s)^{\top},\qquad K_i\succeq0.
 \label{gen:eq-gram-imaginary}
\end{equation}
These matrix sizes do not depend on the coordinate-degree bounds.

Let \(\mathcal N_4,\mathcal N_{2,2},\mathcal N_2\) denote,
respectively, the cones of nonnegative real univariate polynomials
of degree at most four, nonnegative real biforms of bidegree
at most \((2,2)\), and nonnegative real univariate polynomials
of degree at most two.
The real generator cone has lineality
\[
 \mathfrak g_\kappa
 =\R\Id+\sum_iJ_{i,\kappa_i}(\R[z_i]_{\leq2}),
\]
and its quotient by this lineality is naturally isomorphic to
\begin{equation}
 \prod_{i:\kappa_i\geq2}\mathcal N_4
 \ \times\!
 \prod_{i<j}\mathcal N_{2,2}.
 \label{gen:eq-real-quotient}
\end{equation}
The complex generator cone, regarded as a real cone, has
lineality
\[
 \C\Id+\sum_iJ_{i,\kappa_i}(\R[z_i]_{\leq2}),
\]
and its quotient is the product in
\eqref{gen:eq-real-quotient} times \(\prod_i\mathcal N_2\).
\end{corollary}

\begin{proof}
The square decompositions in Lemma~\ref{gen:sos} give
\eqref{gen:eq-gram-unary} and \eqref{gen:eq-gram-pair} by
taking sums of outer products of the coefficient vectors.
Conversely a real positive semidefinite matrix is a sum
of rank-one positive semidefinite matrices, yielding the
required square decomposition.
The quadratic case is the same argument with the vector
\((1,s)\).

By Proposition~\ref{gen:box-decomposition}, the kernel of the
negative-principal-coefficient map is the order-at-most-one space,
while the polarization construction makes it onto the product in
\eqref{gen:eq-real-quotient}. Both \(A\) and \(-A\) are admissible
precisely when their principal coefficients vanish, which identifies
this kernel with the real lineality and yields the stated quotient.

For a complex generator, \eqref{gen:eq-complex-decomposition}
adds the nonnegative polynomials \(q_i\) independently of the real
principal coefficients. Admissibility of both signs forces these
polynomials and the principal coefficients to vanish, leaving only
the real first-order operators and an arbitrary complex scalar;
this gives the complex lineality and quotient.
\end{proof}

\begin{corollary}
\label{gen:slot-lift}
For every generator \(A\) in Theorem~\ref{gen:main-complex},
there is a complete stability-preserving generator
\(\widetilde A\) on the full multiaffine space with
\(|\kappa|\) variables such that
\begin{equation}
 \widetilde A\mathcal P_\kappa=\mathcal P_\kappa A,
 \qquad
 e^{t\widetilde A}\mathcal P_\kappa
          =\mathcal P_\kappa e^{tA}\quad(t\geq0).
 \label{gen:eq-exact-lift}
\end{equation}
It commutes with every permutation of slots within a site.
It can be chosen as a sum of a scalar, one-slot real
M\"obius generators, one-slot imaginary parabolic generators,
and nonnegative multiples of the rank-one apolar generators
\(B_F\) on pairs of slots.
\end{corollary}

\begin{proof}
The slot sums from the proof of Theorem~\ref{gen:main-real}
realize the real principal coefficients and commute with all
within-site permutations. For the remaining first-order
part, lift each sitewise drift by
\[
 J_{i,\kappa_i}(q)
 \longmapsto
 \sum_{a=1}^{\kappa_i}
 \left(q(z_{i,a})\partial_{i,a}
                    -\frac12q'(z_{i,a})\right).
\]
On a polarized input, \eqref{gen:eq-polar-one} gives the descent
\(J_{i,\kappa_i}(q)\), and permutation invariance then gives
intertwining with \(\mathcal P_\kappa\). Apply the same lift to
the imaginary quadratics, decomposed into the parabolic primitives
of Lemma~\ref{gen:imaginary-primitive}, and lift the scalar by the
same scalar identity. Each summand generates a complete preserving
flow on the full slot space, so their sum has this property and
satisfies the first relation in \eqref{gen:eq-exact-lift};
exponentiation yields the second.
\end{proof}

\subsection{Exact recognition from local matrices}
\label{gen:recognition-subsection}

The small nonnegativity certificates also give an exact recognition
algorithm when the operator is supplied by its local terms.

\begin{corollary}
\label{gen:recognition}
Let \(A\) be specified as a finite sum of one-site and two-site operators
on \((\C^2)^{\otimes n}\), with Gaussian-rational entries in the
coefficient basis \(1,z_i\) at each site. There is a deterministic
algorithm, polynomial in \(n\) and the total input bit length, deciding
whether \(e^{tA}\) preserves every nonzero complex disk-stable
multiaffine polynomial for all \(t\ge0\).
The test is on the sum \(A\); the displayed summands need not individually
generate preserving semigroups. No degree bound on the interaction
graph is required.
\end{corollary}

\begin{proof}
In the basis \(1,z\), the Cayley map from disk to half-plane coordinates is
\[
 C=\begin{pmatrix}i&-i\\1&1\end{pmatrix},\qquad
 (Cf)(x)=(x+i)f\left(\frac{x-i}{x+i}\right).
\]
The conjugated half-plane generator is
\(B=C^{\otimes n}A(C^{-1})^{\otimes n}\). Each local conjugation
uses a matrix of order two or four and preserves its support, so
all entries remain Gaussian rational with polynomial bit length.

For a transformed two-site term \(U\), with variables \(x,y\), set
\[
 \begin{aligned}
 p_0&=U1,&p_x&=U(x)-xU1,&p_y&=U(y)-yU1,\\
 p_{xy}&=U(xy)-yU(x)-xU(y)+xyU1.
 \end{aligned}
\]
Checking the four monomials \(1,x,y,xy\) proves
\(U=p_0+p_x\partial_x+p_y\partial_y+p_{xy}\partial_x\partial_y\).
Using the same formula without \(y\)-terms at a single site,
then adding the expressions and merging equal monomials, gives the
canonical representation
\[
 B=P_0+\sum_i P_i\partial_i+\sum_{i<j}P_{ij}\partial_i\partial_j.
\]
The order bound is automatic: each \(P_{ij}\) has bidegree at
most \((2,2)\) in its two variables, and every coefficient has a
sparse representation with only constantly many contributions from
each input term.

By Theorems~\ref{gen:main-real} and \ref{gen:main-complex}, acceptance
is equivalent to the following finite tests on the merged coefficients:
\begin{enumerate}
\item each \(P_{ij}\) is real and \(P_{ij}(s,t)\le0\) for every
\((s,t)\in\R^2\);
\item each \(\Im P_i\) is a polynomial \(q_i(z_i)\) of degree at most
two, and \(q_i(s)\ge0\) for every \(s\in\R\);
\item \(\Im P_0+\tfrac12\sum_iq_i'\) is a real constant.
\end{enumerate}
These conditions make the real part of \(B\) a real box generator
and give its allowed imaginary part, while conversely each is
forced by the complex classification. Since the Cayley map and
\(e^{tA}\) are invertible, the resulting stable-or-zero criterion
is equivalent to preservation of the nonzero stable class.

The identity tests require rational arithmetic and merging only
polynomially many sparse records. Forming common denominators by
products shows that the bit lengths of all merged coefficients are
polynomial in the input length. After clearing denominators by a
positive multiplier, each remaining sign test is a real sentence
in at most two variables, with one polynomial of total degree at
most four. Fixed-dimensional, fixed-degree real quantifier elimination
has polynomial bit complexity \cite[Chapter~14]{BasuPollackRoy06}.
There are polynomially many such sentences, giving the stated
bit bound without expanding the global matrix; their weak
inequalities include zero coefficients, repeated roots, and
degree drops.
\end{proof}

Efficient tests for bivariate real stability and static univariate
real-rootedness preservers were obtained by Raghavendra, Ryder, and
Srivastava \cite{RaghavendraRyderSrivastava17}. Here the finite-box
generator classification reduces recognition from a compact many-site
local description to fixed-dimensional nonnegativity tests.

\section{Damping and spectral bounds}\label{spec:section}

On a finite degree box, positive degree damping strengthens disk
stability preservation to strict nonvanishing and separates a principal
eigenvalue from the rest of the spectrum. The strict kernel property
provides the eigenvector needed for the ratio argument below, which
yields both the spectral gap and contraction along the evolution.

\subsection{Weighted kernels and the stable cone}

Fix \(d\ge1\) and \(\kappa\in\N_{>0}^d\), write
\[
 V_\kappa=\{p\in\C[z_1,\ldots,z_d]:\deg_{z_i}p\le\kappa_i\},
 \qquad b_\alpha=\binom\kappa\alpha,
\]
and let \(\D_r^d=\{z:|z_i|<r\text{ for all }i\}\).
Let \(\mathcal C_\kappa\) consist of zero and all polynomials in
\(V_\kappa\) that do not vanish on \(\D_1^d\).  We call a polynomial
\emph{strictly stable} if it does not vanish on \(\overline{\D_1^d}\).
For \(T\in\operatorname{End}_{\C}(V_\kappa)\), use the kernel
of \eqref{pre:kernel}, namely
\begin{equation}\label{spec:kernel}
 K_T(z,w)=T_z\prod_{i=1}^d(1+z_iw_i)^{\kappa_i}
 =\sum_{\alpha\le\kappa}b_\alpha(Tz^\alpha)(z)w^\alpha.
\end{equation}
The corresponding coefficient pairing is bilinear:
\begin{equation}\label{spec:pairing}
 [f,g]_\kappa=\sum_{\alpha\le\kappa}b_\alpha^{-1}f_\alpha g_\alpha.
\end{equation}
In particular, \([K_T(z,\cdot),p]_\kappa=Tp(z)\).  In the coherent basis \(\sqrt{b_\alpha}z^\alpha\), this is the
matrix kernel with square-root binomial weights on its two legs.
Coordinates with degree bound zero may be omitted. If none remain,
the kernel and eigenvector assertions reduce to nonzero constants;
the spectral gap statements concern the case in which at least one
coordinate remains.

Lemma~\ref{pre:grace} shows that contraction preserves stability when
the paired radii have product at least one, and is nonzero when every
product is greater than one, retaining the uncontracted radii.
For two operator kernels, the pairing gives the kernel of their product.

\begin{lemma}\label{spec:kernel-criterion}
An invertible operator \(T\) preserves \(\mathcal C_\kappa\) if and
only if \(K_T\) is zero-free on \(\D_1^d\times\D_1^d\).
If \(K_T\) is zero-free on \(\D_r^d\times\D_r^d\), where \(r>1\),
then \(Tp\) is nonzero and zero-free on \(\D_r^d\) for every
\(p\in\mathcal C_\kappa\setminus\{0\}\), without an invertibility
assumption on \(T\) in this last assertion.
\end{lemma}
\begin{proof}
For fixed \(w\in\D_1^d\), the stable polynomial
\(\prod_i(1+z_iw_i)^{\kappa_i}\) has a nonzero stable image
under an invertible preserver, so evaluation at \(z\in\D_1^d\)
gives nonvanishing of the kernel. Conversely, contraction of a
nonvanishing kernel with a stable input yields a stable polynomial
or zero by Lemma~\ref{pre:grace}; for nonzero input the latter
is impossible because the operator is invertible. When the kernel radius is \(r>1\),
the paired radii have product \(r\), so strict contraction itself
excludes zero and retains output radius \(r\).
\end{proof}

\begin{lemma}\label{spec:cone-interior}
The set \(\mathcal C_\kappa\) is a closed, complex-scalar-invariant
cone containing no complex two-dimensional subspace.  Its interior consists
exactly of the strictly stable polynomials.  Every invertible operator
preserving this cone maps its interior into its interior.
\end{lemma}
\begin{proof}
By Hurwitz's theorem on the connected polydisk,
\(\mathcal C_\kappa\) is closed, and scalar invariance follows
from its definition. Every complex two-dimensional subspace contains
a nonzero polynomial annihilated by evaluation at zero, and hence
cannot be contained in this cone.
By compactness, a polynomial nonvanishing on the closed polydisk
remains so under small coefficient perturbations. In the other
direction, a boundary zero \(\zeta\) of \(p\) becomes an
interior zero \(\zeta/(1+\varepsilon)\) of the arbitrarily
nearby polynomial \(p((1+\varepsilon)z)\), while small nonzero
multiples of \(z_1\) show that zero is not an interior point.
This identifies the interior; an invertible cone-preserving operator
maps it to an open subset of the cone, since the operator is open.
\end{proof}

\subsection{Strictification by degree damping}

Put \(E_i=z_i\partial_i\), \(E=\sum_iE_i\), and \(P_0p=p(0)\).
The proof of the next theorem treats the damping rate as a complex
parameter. Hurwitz's theorem then reduces boundary nonvanishing to
a limit at large positive damping.

\begin{theorem}\label{spec:strict}
Suppose \(A_0\in\operatorname{End}_{\C}(V_\kappa)\) satisfies
\(e^{tA_0}\mathcal C_\kappa\subseteq\mathcal C_\kappa\) for
\(t\ge0\).  Let \(\mu_i>0\) and
\[
 A=A_0-2\sum_{i=1}^d\mu_i E_i.
\]
For every \(t>0\), the kernel \(K_{e^{tA}}\) is nonzero on the
closed unit polydisk in all \(2d\) variables.  Consequently it has a
zero-free radius \(r_t>1\).  The same radius works for the image of
every nonzero unit-disk-stable input polynomial.

Moreover, for each \(t_0>0\), one radius \(r>1\) works simultaneously
for the kernels \(K_{e^{tA}}\), \(t\ge t_0\), and for their action
on all nonzero unit-disk-stable inputs.
\end{theorem}

\begin{proof}
For uniform damping, fix \(t>0\) and regard
\(T_t(\mu)=e^{t(A_0-2\mu E)}\) as a function on \(\Re\mu>0\).
The dilation
\[
 e^{-2s\mu E}p(z)=p(e^{-2s\mu}z)
\]
preserves disk stability for \(s\ge0\), so each Lie product
\[
 \left(e^{tA_0/m}e^{-2t\mu E/m}\right)^m
\]
preserves stability, as does its invertible limit \(T_t(\mu)\)
by the finite-dimensional product formula and Hurwitz's theorem.
Lemma~\ref{spec:kernel-criterion} therefore gives
\begin{equation}\label{spec:open-complex-field}
 K_{T_t(\mu)}(z,w)\ne0
 \quad\text{if }z,w\in\D_1^d,\quad\Re\mu>0.
\end{equation}
We next exclude the possibility that a boundary value of the kernel
vanishes identically in \(\mu\). To this end set
\(a_{00}=(A_01)(0)\) and let \(\mu\to+\infty\).
In the Euclidean monomial norm, \(E\) is diagonal and nonnegative
with kernel the constants, whence
\[
 \|e^{-2s\mu E}\|\le1,
 \qquad e^{-2s\mu E}\longrightarrow P_0\quad(s>0).
\]
The Dyson expansion of \(T_t(\mu)\) about \(-2\mu E\) has
order-\(k\) term, for \(k\ge1\), equal to the integral of
\[
 e^{-2\mu(t-s_k)E}A_0e^{-2\mu(s_k-s_{k-1})E}
 \cdots A_0e^{-2\mu s_1E}
\]
over \(0<s_1<\cdots<s_k<t\). Almost everywhere on this simplex,
all time intervals are positive, so the integrand converges to
\(a_{00}^kP_0\) and is bounded in norm by \(\|A_0\|^k\).
By dominated convergence, the integral tends to
\(t^ka_{00}^kP_0/k!\), and the summable majorant \(t^k\|A_0\|^k/k!\) allows passage
through the full series, including its order-zero term:
\begin{equation}\label{spec:large-damping}
 T_t(\mu)\longrightarrow e^{ta_{00}}P_0
 \quad\text{in operator norm as real }\mu\longrightarrow+\infty.
\end{equation}
This finite-dimensional strong-coupling limit, comparable to
\cite[Theorem~1]{Burgarth19}, gives the nonzero constant kernel
\(e^{ta_{00}}\), since \(K_{P_0}=1\).

For fixed \((z,w)\in\overline{\D_1^d}\times\overline{\D_1^d}\)
and \(0<\rho<1\), the functions
\[
 f_\rho(\mu)=K_{T_t(\mu)}(\rho z,\rho w)
\]
are holomorphic and nonvanishing on \(\Re\mu>0\), and converge
locally uniformly as \(\rho\uparrow1\), since the kernel has
only finitely many entire coefficient functions. By Hurwitz's theorem,
the limit \(K_{T_t(\mu)}(z,w)\) is either identically zero
or nowhere zero as a function of \(\mu\). It cannot vanish identically,
since its large-damping limit in \eqref{spec:large-damping} is nonzero.
Thus the kernel is nonvanishing on the closed polydisk at every
real \(\mu>0\). By compactness there is a radius \(r_t>1\)
on which it remains nonvanishing; by Lemma~\ref{spec:kernel-criterion},
the same radius works for the image of every nonzero stable input.

For nonuniform rates, put \(\mu=\min_i\mu_i>0\) and absorb
\(-2\sum_i(\mu_i-\mu)E_i\) into \(A_0\). The product-formula
argument still gives a stability generator, to which the uniform
result applies with the remaining damping \(-2\mu E\).

Finally fix \(t_0>0\), and choose \(r>1\) for
\(B=e^{t_0A/2}\).  For every \(t\ge t_0\),
\[
 e^{tA}=B e^{(t-t_0)A}B.
\]
The middle kernel has radius one, also when \(t=t_0\), while
the two outer kernels have radius \(r\). Both contractions therefore
pair radii with product \(r>1\) and retain radius \(r\) on the
external legs, uniformly for \(t\ge t_0\); contraction with a
stable input retains the same output radius.
\end{proof}

\begin{remark}\label{spec:strict-scope}
The radius depends on the generator, the coordinate degrees, and
\(t_0\). Example~\ref{spec:no-uniform-radius} shows that dependence
on the generator is necessary even in one variable. The strict bound
refers to the damped variables. Multiplication
by a nonzero scalar leaves the kernel zeros unchanged.
\end{remark}

\subsection{Principal eigenvectors}

We use the following finite-dimensional complex-cone theorem of Rugh
\cite[Theorem~8.4]{Rugh10}: if a nontrivial closed complex-scalar-invariant cone
contains no complex plane and a linear map sends its nonzero part into
its interior, then the map has an algebraically simple nonzero
eigenvalue of strictly largest modulus.

\begin{proposition}\label{spec:perron}
Under the hypotheses of Theorem~\ref{spec:strict}, \(A\) has an
algebraically simple eigenvalue \(a_\star\) uniquely largest in real
part.  Choose its right eigenvector \(v\) and left eigenfunctional
\(\ell\) so that \(\ell(v)=1\).  Both \(v\) and
\begin{equation}\label{spec:dual-polynomial}
 Q_\ell(w)=\sum_{\alpha\le\kappa}b_\alpha\ell(z^\alpha)w^\alpha
\end{equation}
are zero-free on a common polydisk of radius greater than one.
Furthermore \(\ell(p)\ne0\) for every nonzero stable \(p\), and
\(e^{-ta_\star}e^{tA}p\to\ell(p)v\).
\end{proposition}
\begin{proof}
For fixed \(t>0\), Theorem~\ref{spec:strict} and
Lemma~\ref{spec:cone-interior} place \(T=e^{tA}\) under
Rugh's theorem, since the cone contains the constants and is
therefore nontrivial. As \(A\) commutes with \(T\), the simple
dominant eigenspace of \(T\) is an \(A\)-eigenline, say for
\(a_\star\). By spectral mapping, a distinct eigenvalue of \(A\)
with the same real part would give either a second eigenvalue of
\(T\) of maximal modulus or a second copy of its dominant eigenvalue,
contrary to strict dominance and simplicity. A Jordan block at \(a_\star\)
would likewise persist under exponentiation, because
\(t e^{ta_\star}\ne0\), contradicting simplicity of \(T\).

Fix a kernel radius \(r>1\) for \(T\), and put
\(\Lambda=e^{ta_\star}\). By finite-dimensional spectral theory,
\[
 \Lambda^{-m}T^m\longrightarrow P=v\otimes\ell,
 \qquad \ell(v)=1.
\]
Internal contraction preserves radius \(r\) for every power,
since the paired radii have product \(r^2>1\). Because the kernel
transform is injective, \(K_P\) is nonzero, and hence by Hurwitz's theorem
\[
 K_P(z,w)=v(z)Q_\ell(w)\ne0\qquad(z,w\in\D_r^d).
\]
Both factors have radius \(r\), and strict weighted contraction
of \(Q_\ell\) with a nonzero unit-stable \(p\) shows that
\(\ell(p)\ne0\). After normalization, every other term in the
finite-dimensional spectral expansion decays by the strict real-part
separation, even in the presence of the polynomial factors arising
from Jordan blocks.
\end{proof}

\subsection{Sharp spectral bounds}

The spectral and dynamical estimates both use the following diameter
form of the Schwarz lemma, as in \cite[Theorem~2]{BGLW26}.

\begin{lemma}\label{spec:schwarz}
If \(g\) is holomorphic on \(\D_1^d\), continuous on its closure,
and \(0\le q\le1\), then
\[
 \operatorname{diam}g(q\overline{\D_1^d})
 \le q\operatorname{diam}g(\overline{\D_1^d}).
\]
The same conclusion holds for a coordinatewise dilation all of whose
moduli are at most \(q\).
\end{lemma}
\begin{proof}
It suffices to take \(0<q<1\) and positive diameter. For
\(x,y\in q\overline{\D_1^d}\) with
\(\rho=\max(\|x\|_\infty,\|y\|_\infty)>0\), the function
\[
 f(\zeta)=
 \frac{g(\zeta x/\rho)-g(\zeta y/\rho)}
 {\operatorname{diam}g(\overline{\D_1^d})}
\]
is holomorphic on the disk, vanishes at zero, and has modulus
at most one. Schwarz's lemma at \(\zeta=\rho\) gives
\(|g(x)-g(y)|\le\rho\operatorname{diam}g(\overline{\D_1^d})\);
taking the supremum proves the dilation estimate, and the
coordinatewise version follows because its image lies in
\(q\overline{\D_1^d}\).
\end{proof}

\begin{proposition}\label{spec:radius-gap}
Suppose \(r>1\) and \(K_T\) is zero-free on
\(\D_r^d\times\D_r^d\).  Then \(T\) has an algebraically simple
nonzero eigenvalue \(\Lambda_\star\) of strictly largest modulus,
with an eigenvector of radius \(r\).  Every other eigenvalue satisfies
\begin{equation}\label{spec:static-ratio}
 |\Lambda/\Lambda_\star|\le r^{-2}.
\end{equation}
No invertibility, Hermiticity, or diagonalizability is assumed.
\end{proposition}
\begin{proof}
Lemma~\ref{spec:kernel-criterion} sends the nonzero stable cone
into its interior, giving the simple dominant eigenvalue by Rugh's
theorem. Its spectral projection is nonzero on some interior
cone vector, since the kernel of a nonzero projection cannot
contain an open set. Normalized iterates of this vector converge
to a nonzero eigenvector \(\psi\) in the closed cone, and one
further application of \(T\) gives it radius \(r\).

Let \(\phi\) be an eigenvector of a different nonzero eigenvalue,
and put \(g=\phi/\psi\), \(\alpha=\Lambda/\Lambda_\star\).
Since \(g\) is nonconstant and holomorphic on \(\D_r^d\), we
can compare its image with that of \(\alpha g\): for
\(1/r<s<r\),
\[
 \alpha g(\D_r^d)\subseteq g(\D_s^d).
\]
Otherwise, for some \(v\in\D_r^d\), the nonzero polynomial
\(f=\phi-\alpha g(v)\psi\) would be stable on \(\D_s^d\), but
\(Tf=\Lambda(\phi-g(v)\psi)\) would vanish at \(v\).
This contradicts strict contraction, since \(rs>1\).

For \(1/r<s<t<r\), let
\(\Delta_u=\operatorname{diam}g(\overline{\D_u^d})\).
These finite positive diameters satisfy
\(|\alpha|\Delta_t\le\Delta_s\le(s/t)\Delta_t\), by the image
inclusion and Lemma~\ref{spec:schwarz} after rescaling radius \(t\).
Cancelling \(\Delta_t\) and taking \(s\downarrow1/r\),
\(t\uparrow r\) gives \eqref{spec:static-ratio}; zero
eigenvalues already satisfy the bound.
\end{proof}

The quantitative mechanism in this proof is the ratio-squeezing argument
of \cite{BGLW26}; we use the complex-cone theorem to obtain the eigenvector
without a positive-semidefinite hypothesis.

\begin{theorem}\label{spec:gap}
Under the hypotheses of Theorem~\ref{spec:strict}, set
\(\mu=\min_i\mu_i\).  The dominant generator eigenvalue satisfies
\begin{equation}\label{spec:real-gap}
 \Re a_\star-\Re a\ge2\mu
 \qquad(a\in\operatorname{spec}(A),\ a\ne a_\star).
\end{equation}
The constant is sharp for every vector of positive finite degree bounds.
\end{theorem}
\begin{proof}
For small \(h>0\), put
\[
 D_h=e^{-h\sum_i\mu_iE_i},\qquad
 T_h=D_he^{hA_0}D_h.
\]
The kernel of \(e^{hA_0}\) has radius one by Lemma
\ref{spec:kernel-criterion}.  If \(q_i=e^{-h\mu_i}\), direct
coefficient calculation gives
\[
 K_{T_h}(z,w)=K_{e^{hA_0}}(q_1z_1,\ldots,q_dz_d,
                              q_1w_1,\ldots,q_dw_d).
\]
The kernel therefore has radius \(e^{h\mu}\). By
Proposition~\ref{spec:radius-gap}, there is a simple dominant eigenvalue
\(\Lambda_\star(h)\), and every other eigenvalue has modulus at
most \(e^{-2h\mu}|\Lambda_\star(h)|\).
Since \(B_h=(T_h-I)/h\to A\) in norm, the transformed eigenvalues
\(b(h)=(\Lambda(h)-1)/h\) form bounded multisets converging,
with algebraic multiplicity, to the spectrum of \(A\).
The expansion
\[
 h^{-1}\log|1+hb|=\Re b+O(h),
\]
uniform for bounded \(b\), converts the modulus bound into
\(\Re b_\star(h)-\Re b(h)\ge2\mu+O(h)\).
Along a subsequence on which \(b_\star(h)\) converges, remove
its single copy from each spectral multiset; the remaining copies
converge, with multiplicity, to those of \(A\) other than the
distinguished limit. The limiting inequality separates that limit
from all the remaining copies by \(2\mu\), proving both its
algebraic simplicity and its identification with \(a_\star\).
For \(A_0=0\), a degree-one monomial at a site of minimum rate
has eigenvalue \(-2\mu\), while the full spectrum is
\(-2\sum_i\mu_i\alpha_i\), so equality is attained.
\end{proof}

\subsection{Contraction along the evolution}\label{spec:oscillation-subsection}

The same diameter estimate also controls the evolution of polynomial
quotients. This gives a norm adapted to the principal eigenvector and
a resolvent estimate. For strictly stable \(p\) and
\(h\in V_\kappa\), define
\begin{equation}\label{spec:osc-definition}
 \operatorname{Osc}_p(h)
 =\operatorname{diam}\{h(z)/p(z):z\in\overline{\D_1^d}\}.
\end{equation}
It is a complex seminorm with kernel \(\C p\): homogeneity and the
triangle inequality follow from the diameter definition, and zero
diameter means that the rational function is constant.

\begin{lemma}\label{spec:range-exclusion}
If \(T\) is an invertible stability preserver, \(p\) is strictly
stable, and \(h\in V_\kappa\), then \(Tp\) is strictly stable and
\[
 (Th/Tp)(\D_1^d)\subseteq(h/p)(\D_1^d).
\]
In particular \(\operatorname{Osc}_{Tp}(Th)\le
\operatorname{Osc}_p(h)\).
\end{lemma}
\begin{proof}
By Lemma~\ref{spec:cone-interior}, \(Tp\) is strictly stable, so
the output ratio is defined on the closed polydisk. For
\(w\notin(h/p)(\D_1^d)\), the polynomial \(h-wp\) is nonzero
and stable, as is its image \(Th-wTp\), excluding \(w\) from
the output ratio image.
The diameter inequality follows from this inclusion and continuity
on the closed polydisk, where density makes the diameter agree
with that over the open polydisk.
\end{proof}

\begin{proposition}\label{spec:oscillation}
Let \(A_0\) generate a stability-preserving semigroup, and let
\(A=A_0-2\sum_i\mu_iE_i\), with \(\mu_i\ge\mu\ge0\).
For every strictly stable \(p\), every \(h\in V_\kappa\), and
\(t\ge0\),
\begin{equation}\label{spec:osc-bound}
 \operatorname{Osc}_{e^{tA}p}(e^{tA}h)
 \le e^{-2\mu t}\operatorname{Osc}_p(h).
\end{equation}
The constant is sharp.
\end{proposition}
\begin{proof}
The damping operator \(D_\tau=e^{-2\tau\sum_i\mu_iE_i}\) acts
by a coordinatewise dilation of maximum modulus at most \(e^{-2\mu\tau}\).
Applying Lemma~\ref{spec:schwarz} to \(h/p\) and
Lemma~\ref{spec:range-exclusion} to \(e^{\tau A_0}\) gives
contraction by \(e^{-2\mu\tau}\) and nonexpansion, respectively.
Iteration along \((e^{tA_0/m}D_{t/m})^m\) therefore gives the
factor \(e^{-2\mu t}\), with strictly stable denominators at
every intermediate step.

The limiting exponential \(e^{tA}\) is invertible and preserves
stability by the product formula and Hurwitz's theorem, so its
image of \(p\) remains strictly stable. The denominator thus
has a positive minimum modulus on the closed polydisk, allowing
coefficient convergence to pass to uniform convergence of the
ratios and then to convergence of their diameters.
We obtain \eqref{spec:osc-bound}; equality holds for pure damping
with \(\mu_i=\mu\) for all \(i\), \(p=1\), and \(h=z_j\).
\end{proof}

\begin{corollary}\label{spec:nonautonomous}
On each compact time interval, suppose \(A_0(t)\) and \(\mu_i(t)\)
are bounded and piecewise continuous, with finitely many continuity
pieces, each \(A_0(t)\) a stability generator and \(\mu_i(t)\ge0\).
Let \(U(t,s)\) solve the linear evolution equation with generator
\(A_0(t)-2\sum_i\mu_i(t)E_i\). Then, for strictly stable \(p\)
and \(t\ge s\),
\[
 \operatorname{Osc}_{U(t,s)p}(U(t,s)h)
 \le \exp\!\left(-2\int_s^t\min_i\mu_i(u)\,du\right)
       \operatorname{Osc}_p(h).
\]
\end{corollary}
\begin{proof}
On each continuity piece, step functions taking actual coefficient
values approximate the generator in \(L^1\). Their frozen
evolutions satisfy Proposition~\ref{spec:oscillation}, including at
zero minimum rate, and multiplication gives the exponential of
the integrated stepwise minimum rate. These integrals converge
to the stated one because the minimum of finitely many real
coordinates is Lipschitz. By the integral equation for finite matrix
ODEs, the propagators converge uniformly to the invertible
fundamental matrix, which preserves stability by Hurwitz's theorem.
The limiting denominator is therefore strictly stable by
Lemma~\ref{spec:cone-interior}, so we can pass to the limit in the
ratios as in the autonomous argument.
\end{proof}

\begin{corollary}\label{spec:resolvent}
Assume positive damping, put \(\mu=\min_i\mu_i\), and normalize \(v,\ell\) as in
Proposition~\ref{spec:perron}.  On \(W=\ker\ell\), set
\(\|h\|_v=\operatorname{Osc}_v(h)\) and
\(B=(A-a_\star I)|_W\).  Then \(\|\cdot\|_v\) is a norm and
\[
 \|e^{tB}\|_{v\to v}\le e^{-2\mu t},\qquad
 \|(\zeta I-B)^{-1}\|_{v\to v}
 \le\frac1{\Re\zeta+2\mu}\quad(\Re\zeta>-2\mu).
\]
Every eigenvalue of \(B\) with real part \(-2\mu\) is semisimple.
\end{corollary}
\begin{proof}
Since \(\ell(v)=1\), no nonzero multiple
of \(v\) belongs to \(W\), so the seminorm restricts to a norm there.
As the kernel of a left eigenfunctional, \(W\) is invariant. Substituting
\(e^{tA}v=e^{ta_\star}v\) into
Proposition~\ref{spec:oscillation} gives the first estimate, including
the phase of the scalar.
For \(\Re\zeta>-2\mu\), the absolutely convergent integral
\(\int_0^\infty e^{-t\zeta}e^{tB}\,dt\) is the inverse of
\(\zeta I-B\), by integration of its derivative, and its norm is bounded by
integrating the first estimate.  A nontrivial Jordan block at a
boundary eigenvalue would produce unbounded polynomial growth after
multiplication by \(e^{2\mu t}\), contradicting the first estimate
and norm equivalence on its finite-dimensional generalized eigenspace.
\end{proof}

\begin{remark}\label{spec:projective-attribution}
For the stable cone, the forbidden-pencil projective distance of
\cite{Dubois09} is
\(\operatorname{osc}_{\D_1^d}\log|q/p|\).  Its real infinitesimal
form is \(\operatorname{osc}\Re(h/p)\); maximizing over a scalar
phase of \(h\) gives \eqref{spec:osc-definition}.  The estimates above give the degree-damping rate in this geometry,
including time-dependent flows. Examples~\ref{spec:no-uniform-radius}
and~\ref{spec:nonnormal} in Appendix~\ref{spec:examples-section}
show why the strict radius and the adapted norm depend on the operator.
\end{remark}

\section{All-temperature Lee--Yang Hamiltonians}\label{eq:section}

For Hermitian qubit operators, the generator classification identifies
precisely the phase-rotated Suzuki--Fisher family: the Cayley transform
turns the Gibbs tensor into an operator kernel, and the coefficient
conditions in Theorem~\ref{gen:main-complex} become inequalities on the
Pauli interactions. Positive longitudinal fields then give strict Gibbs
kernels and a unique ground state, whose energy and scalar queries will
be computed in Section~\ref{alg:section} by analytic continuation in
the field.

\subsection{From matrix tensors to polynomial generators}

Fix the computational basis \(\{|a\rangle:a\in\{0,1\}^n\}\),
with \(Z|0\rangle=|0\rangle\).  Associate to a matrix \(M\) the
polynomial
\begin{equation}\label{eq:matrix-tensor}
 F_M(z,w)=\sum_{a,b\in\{0,1\}^n}M_{ab}z^aw^b.
\end{equation}
We say that \(M\) has the tensor Lee--Yang property at radius \(r\)
if \(F_M\) is nonzero on \(\D_r^n\times\D_r^n\), with the two
groups of variables independent as in \eqref{pre:coherent}.

For a Hermitian Hamiltonian \(H\), the Gibbs operator is
\(M_\beta=e^{-\beta H}\), and its normalized density matrix is
\(\rho_\beta=M_\beta/\Tr M_\beta\).  Since
\(\Tr e^{-\beta H}>0\) for real \(\beta\), the two matrices have
exactly the same polynomial zeros.

The following unitary normalization of the binary Cayley transform
relates the Gibbs tensor to the semigroup in half-plane coordinates.

\begin{lemma}\label{eq:cayley-bridge}
Let
\[
 C_1=\frac1{\sqrt2}\begin{pmatrix}i&-i\\1&1\end{pmatrix},
 \qquad C=C_1^{\otimes n},\qquad D_1=\diag(-1,1).
\]
For an invertible matrix \(M\), its tensor has the unit-disk
Lee--Yang property if and only if \(CMC^{-1}\), acting on
multiaffine coefficient vectors, preserves upper-half-plane stability.
Consequently, the Gibbs family corresponds to the generator
\begin{equation}\label{eq:generator}
 A=-CHC^*.
\end{equation}
Here operator conjugation is a similarity, whereas transforming both
groups of tensor variables gives the congruence \(CMC^T\).
\end{lemma}
\begin{proof}
The coefficient action of \(C\) is
\[
 (Cf)(s)=2^{-n/2}\prod_i(s_i+i)
 f\!\left(\frac{s_1-i}{s_1+i},\ldots,
           \frac{s_n-i}{s_n+i}\right).
\]
Since the fractions map the upper half-plane onto the unit disk and
the prefactors are nonzero there, \(C\) identifies the two stability
classes. For invertible \(M\), the binary contraction theorem
(Lemma~\ref{spec:kernel-criterion}) identifies disk stability of its
tensor with preservation of disk-stable coefficient vectors, giving
the asserted equivalence by similarity.

To check the tensor convention directly, observe that
\(C_1C_1^T=D_1\), so the coefficient matrix of
\[
 \prod_i(s_it_i-1)
\]
is \(D_1^{\otimes n}=CC^T\).  Each factor is stable in the two
upper-half-plane variables: if \(s_it_i=1\), then
\(t_i=1/s_i\) has negative imaginary part.  Acting on its first
group of variables by \(N=CMC^{-1}\) gives coefficient matrix
\(NCC^T=CMC^T\).  For fixed \(t\in\HH^n\), invertibility
excludes a zero output, so stability preservation gives nonvanishing
jointly in \((s,t)\in\HH^{2n}\).  Finally,
\[
 F_{CMC^T}(s,t)
 =2^{-n}\prod_i(s_i+i)(t_i+i)
 F_M\!\left(\frac{s-i}{s+i},\frac{t-i}{t+i}\right).
\]
The tensor transformation is therefore a congruence, whereas
\(C^{-1}=C^*\) gives the exponential similarity
\(Ce^{-\beta H}C^{-1}=e^{\beta A}\) with the asserted generator.
\end{proof}

\subsection{The Suzuki--Fisher family}

Write \(X,Y,Z\) for the Pauli matrices, with
\(Y=\left(\begin{smallmatrix}0&-i\\ i&0\end{smallmatrix}\right)\).
For a real \(2\times2\) matrix, \(\|\cdot\|_{\mathrm{op}}\)
denotes its largest singular value.

\begin{theorem}\label{eq:sf}
For a Hermitian operator \(H\) on \((\C^2)^{\otimes n}\), the
following conditions are equivalent:
\begin{enumerate}
\item \(e^{-\beta H}\) has the tensor Lee--Yang property at radius
one for every \(\beta\ge0\).
\item The same property holds for every \(\beta\) in some interval
\([0,\varepsilon)\), where \(\varepsilon>0\).
\item There are real coefficients such that
\begin{align}
 H={}&h_0 I+\sum_i(h_i^xX_i+h_i^yY_i+h_i^zZ_i)\notag\\
 &+\sum_{i<j}\bigl(a_{ij}X_iX_j+b_{ij}X_iY_j
           +c_{ij}Y_iX_j+d_{ij}Y_iY_j+e_{ij}Z_iZ_j\bigr),
 \label{eq:sf-form}
\end{align}
where
\begin{equation}\label{eq:sf-inequality}
 h_i^z\le0,\qquad
 e_{ij}\le-
 \left\|\begin{pmatrix}a_{ij}&b_{ij}\\c_{ij}&d_{ij}\end{pmatrix}
 \right\|_{\mathrm{op}}.
\end{equation}
There is no restriction on \(h_0,h_i^x,h_i^y\).
\end{enumerate}
Equivalently, in the ferromagnetic convention
\begin{equation}\label{eq:ferromagnetic-convention}
 H=h_0I-\sum_i\sum_{a=x,y,z}\mu_i^a\sigma_i^a
 -\sum_{i<j}\left(J_{ij}^{zz}Z_iZ_j+
       \sum_{a,b=x,y}J_{ij}^{ab}\sigma_i^a\sigma_j^b\right),
\end{equation}
the conditions are
\(\mu_i^z\ge0\) and
\(J_{ij}^{zz}\ge\|J_{ij}^{\perp}\|_{\mathrm{op}}\), where the
rows and columns of \(J_{ij}^{\perp}\) are ordered \(x,y\).
\end{theorem}

\begin{proof}
The first condition implies the second, which by
Lemma~\ref{eq:cayley-bridge} says that \(e^{tA}\) preserves
upper-half-plane stability for small \(t\ge0\). For arbitrary
\(t\ge0\), choose an integer \(m\) with \(t/m<\varepsilon\);
composing \(e^{(t/m)A}\) \(m\) times extends preservation to all
times and permits the use of Theorem~\ref{gen:main-complex} on the
multiaffine box.

The order bound first rules out interactions on three or more sites.
On one binary site a basis of endomorphisms is
\[
 I,\qquad J^- =\partial_s,\qquad
 J^0=s\partial_s-\tfrac12,\qquad
 J^+=s^2\partial_s-s.
\]
The last three operators are traceless and have the linearly
independent first-order coefficients \(1,s,s^2\). Expanding \(A\)
in tensor products of this basis, choose a support set \(S\) of
maximal cardinality occurring with a nonzero coefficient.
The canonical coefficient of \(\partial_S=\prod_{i\in S}\partial_i\)
is then the sum of its exact-support terms, because no strictly larger
support occurs.  Those terms have distinct principal monomials
\(\prod_{i\in S}s_i^{r_i}\), with \(r_i\in\{0,1,2\}\).
If \(|S|\ge3\), this coefficient vanishes by the order bound in
Theorem~\ref{gen:main-complex}, so by linear independence every
term with support \(S\) vanishes. Descending in support removes all
interactions on three or more sites; the same conclusion holds for
\(H\), since local similarity preserves the identity and traceless
subspaces at each site.

Next write \(A=R+iS\) with entrywise real matrices \(R,S\).
The imaginary-part conclusion of Theorem~\ref{gen:main-complex}
has the form
\begin{equation}\label{eq:imaginary-generator}
 S=c_0I+\sum_iJ_i(q_i),\qquad
 q_i(s)=\alpha_i+\eta_i s+\gamma_i s^2\ge0\quad(s\in\R),
\end{equation}
where \(J(q)=q\partial-q'/2\).  On coefficients \((1,s)\),
\[
 J(q)=\begin{pmatrix}-\eta/2&\alpha\\-\gamma&\eta/2\end{pmatrix}.
\]
Since the matrix is Hermitian, \(R^T=R\) and \(S^T=-S\); taking the trace
in \eqref{eq:imaginary-generator} yields \(c_0=0\). Since the
one-site traceless components are independent, each \(J(q_i)\)
is skew-symmetric, whence \(\eta_i=0\), \(\gamma_i=\alpha_i\), and
\(\alpha_i\ge0\) by nonnegativity. Consequently,
\begin{equation}\label{eq:imaginary-local}
 q_i(s)=\alpha_i(1+s^2),\qquad
 iS=-\sum_i\alpha_iY_i.
\end{equation}

The local Pauli conjugations are
\begin{equation}\label{eq:pauli-cayley}
 C_1XC_1^*=-Z,\qquad C_1YC_1^*=X,\qquad C_1ZC_1^*=-Y.
\end{equation}
Thus the physical field \(h_i^xX_i+h_i^yY_i+h_i^zZ_i\)
contributes \(h_i^xZ_i-h_i^yX_i+h_i^zY_i\) to \(A\), and
comparison with \eqref{eq:imaginary-local} gives
\(h_i^z=-\alpha_i\le0\). Each two-site physical Pauli term with
one \(Z\) factor and one \(X\) or \(Y\) factor becomes an
entrywise imaginary term with one \(Y\) factor; these four terms
are linearly independent on each pair. Since \(S\) has no two-site
component, the physical \(XZ,YZ,ZX,ZY\) coefficients all vanish,
leaving the form \eqref{eq:sf-form}.

For the remaining pair inequality, suppress the pair indices and note
that a physical term \(aXX+bXY+cYX+dYY+eZZ\) contributes
\[
 -aZZ+bZX+cXZ-dXX-eYY
\]
to \(A\).  The first-order coefficients of the coefficient operators
\(X,Z,Y\), respectively, are
\[
 p_X(s)=1-s^2,\qquad p_Z(s)=-2s,\qquad
 p_Y(s)=-i(1+s^2).
\]
Hence the full coefficient of \(\partial_i\partial_j\) is
\begin{align}
 q_{ij}(s,t)={}&e(1+s^2)(1+t^2)-4ast
       -2bs(1-t^2)\notag\\
       &-2ct(1-s^2)-d(1-s^2)(1-t^2).
 \label{eq:mixed-coefficient}
\end{align}
Normalize this coefficient using
\[
 u(s)=\left(\frac{2s}{1+s^2},\frac{1-s^2}{1+s^2}\right),
 \qquad J=\begin{pmatrix}a&b\\c&d\end{pmatrix},
\]
which gives
\begin{equation}\label{eq:opnorm-identity}
 \frac{q_{ij}(s,t)}{(1+s^2)(1+t^2)}=e-u(s)^TJu(t).
\end{equation}
Since \(u(\R)\) is dense in the unit circle, the required inequality
\(q_{ij}(s,t)\le0\) for all real \(s,t\) is equivalent to
\[
 e\le\min_{\|u\|_2=\|v\|_2=1}u^TJv=-\|J\|_{\mathrm{op}},
\]
which establishes necessity of \eqref{eq:sf-inequality}.

Conversely, under \eqref{eq:sf-form}--\eqref{eq:sf-inequality},
the displayed conjugations give canonical differential order at
most two, with real nonpositive mixed coefficients by
\eqref{eq:opnorm-identity} and imaginary part
\eqref{eq:imaginary-generator} with
\(q_i(s)=-h_i^z(1+s^2)\ge0\).  Its real part therefore satisfies
the real criterion in Theorem~\ref{gen:main-complex}, and its
imaginary part satisfies that theorem's one-site condition.  The
theorem proves preservation for all \(\beta\ge0\); the bridge
then gives the tensor Lee--Yang property of \(e^{-\beta H}\).
Finally, changing the signs of the coefficients converts
\eqref{eq:sf-inequality} into \eqref{eq:ferromagnetic-convention}.
\end{proof}

\begin{remark}\label{eq:sf-attribution}
The sufficient Suzuki--Fisher family originates in
\cite{SuzukiFisher71,AsanoJapanese70,AsanoPRL70}. The full transverse
operator-norm condition is already contained, in equivalent algebraic
form, in \cite[equation~(2.47)]{SuzukiFisher71} and
\cite[Theorem~2, equation~(8)]{Dunlop79}; see also the modern notation
in \cite[Definition~1 and Fact~3]{BGLW26}. Indeed, for a real matrix
$J=\left(\begin{smallmatrix}a&b\\c&d\end{smallmatrix}\right)$,
\[
 2\|J\|_{\rm op}
 =\sqrt{(a-d)^2+(b+c)^2}+\sqrt{(a+d)^2+(b-c)^2}.
\]
To verify the identity, the squared radicals are
$\|J\|_F^2-2\det J$ and $\|J\|_F^2+2\det J$.
If the singular values are $s_1\geq s_2\geq0$, the two radicals
are therefore $s_1+s_2$ and $s_1-s_2$, in some order.
Full matrix generating functions and their
composition are also explicit in Asano's contemporaneous account for
the XXZ subfamily \cite[equations~(1.15), (3.7)--(3.10)]{AsanoJapanese70}.
Theorem~\ref{eq:sf} adds the converse among all finite Hermitian
qubit Hamiltonians. Testing the full Gibbs tensor on an interval
\([0,\varepsilon)\), however small, already forces the stated form,
including the exclusion of higher-support interactions.
\end{remark}

\begin{corollary}\label{eq:complex-support}
Let \(H\) be an arbitrary complex qubit matrix.  If
\(e^{-\beta H}\) has the unit-disk tensor Lee--Yang property for
all sufficiently small real \(\beta\ge0\), then its Pauli
expansion contains no term supported on three or more sites.
\end{corollary}
\begin{proof}
The Cayley bridge and the semigroup identity place the generator
under Theorem~\ref{gen:main-complex}, which bounds its canonical
order by two. The maximal-support argument in the proof of
Theorem~\ref{eq:sf} then applies without Hermiticity and eliminates
every Pauli term supported on three or more sites.
\end{proof}

\begin{remark}\label{eq:trace-caveat}
Hermiticity ensures that the Gibbs trace is positive. For a complex
Hamiltonian the trace can vanish even when its tensor is stable: \(H=\diag(0,i\pi)\) has Gibbs
polynomial \(1+e^{-i\pi\beta}zw\), which is nonzero on the open
unit bidisk for every real \(\beta\), but
\(\Tr e^{-H}=0\).  The support conclusion in
Corollary~\ref{eq:complex-support} applies to this complex setting;
the normalized Gibbs-state statements use Hermiticity.
\end{remark}

\subsection{Strict Gibbs radii at positive fields}

\begin{theorem}\label{eq:strict-field}
Let \(H\) have the form \eqref{eq:ferromagnetic-convention}, with
\(J_{ij}^{zz}\ge\|J_{ij}^{\perp}\|_{\mathrm{op}}\) and
\(\mu_i^z>0\) at every site.  Then, for every \(\beta>0\),
\(F_{e^{-\beta H}}\) is nonzero on the closed unit polydisk in
its full \(2n\) variables.  In particular, it has some radius
\(r_\beta>1\).

The ground eigenspace of \(H\) is one-dimensional.  Its nonzero
coefficient vector \(\psi\) defines a polynomial
\(f_\psi(z)=\sum_a\psi_a z^a\) of radius greater than one.
For every \(\beta_0>0\), one radius \(r>1\) works for all
Gibbs kernels at \(\beta\ge\beta_0\), and also for the image
under each \(e^{-\beta H}\) of every nonzero unit-disk-stable
coefficient vector.
\end{theorem}
\begin{proof}
Let \(\mu=\min_i\mu_i^z>0\), and write
\(H=H_0-\mu\sum_iZ_i\).  The Hamiltonian \(H_0\) satisfies the nonstrict Suzuki--Fisher
conditions, so Theorem~\ref{eq:sf} and the disk version of
Lemma~\ref{eq:cayley-bridge} show that \(A_0=-H_0\) generates
a disk-stability preserver. Since \(Z_i=I-2E_i\) in the
multiaffine coefficient basis,
\[
 -H=A_0-2\mu E+\mu nI.
\]
The harmless scalar factor \(e^{\beta\mu n}\) reduces every
kernel and evolution assertion to Theorem~\ref{spec:strict}.

By Proposition~\ref{spec:perron}, the dominant eigenvalue of
\(-H\) is algebraically simple and its right eigenvector has
radius greater than one. Since \(H\) is Hermitian, this eigenvalue
is the negative ground energy, and the ground eigenspace is
one-dimensional. The radius conclusion can also be read directly from the
Gibbs kernels: fix a kernel radius \(r>1\) and normalize its
powers by the largest eigenvalue. Contraction preserves radius
\(r\), while spectral convergence gives
\(|\psi\rangle\langle\psi|\) with nonzero polynomial
\[
 f_\psi(z)\overline{f_\psi(\bar w)}.
\]
By Hurwitz's theorem the limit is nonvanishing on \(\D_r^{2n}\),
and hence each factor is nonvanishing on \(\D_r^n\).
\end{proof}

\begin{remark}\label{eq:gap-attribution}
The spectral conclusion
\[
 E_1(H)-E_0(H)\ge2\min_i\mu_i^z
\]
is the Hermitian specialization of Theorem~\ref{spec:gap} and
was already proved for this family in
\cite[Theorem~1]{BGLW26}.  Here the additional conclusion is
strict nonvanishing of the actual Gibbs tensor at every positive
inverse temperature, and its transfer to the ground-state polynomial.
The radius depends on the Hamiltonian and on the positive time cutoff,
as in Theorem~\ref{spec:strict}.
\end{remark}

\subsection{A qualitative radius for deformed EPR ground states}

Consider a finite graph without isolated vertices, with positive edge
weights \(w_e\).  On an edge \(e=\{i,j\}\), fix
\(0\le s_e<1\) and \(\theta_e\in\R\), and define
\[
 |\eta_e\rangle=|00\rangle+s_e e^{i\theta_e}|11\rangle,
 \qquad H_{\mathrm{EPR}}=-\sum_e w_e
 |\eta_e\rangle\langle\eta_e|_{ij}.
\]
The edge vectors here are unnormalized; replacing them by normalized
vectors simply rescales the positive weights.

\begin{corollary}\label{eq:epr}
The ground space of \(H_{\mathrm{EPR}}\) is one-dimensional, and
its ground-state polynomial is zero-free on a polydisk of some radius
\(r>1\).
\end{corollary}
\begin{proof}
On one edge, direct expansion in Pauli matrices gives
\begin{align*}
 |\eta_e\rangle\langle\eta_e|
 ={}&\frac{1+s_e^2}{4}(I+Z_iZ_j)
    +\frac{1-s_e^2}{4}(Z_i+Z_j)\\
 &+\frac{s_e}{2}\bigl[
      \cos\theta_e(X_iX_j-Y_iY_j)
      +\sin\theta_e(X_iY_j+Y_iX_j)\bigr].
\end{align*}
Its ferromagnetic coefficients are therefore
\[
 J_e^{zz}=\frac{w_e(1+s_e^2)}4,\qquad
 J_e^\perp=\frac{w_es_e}{2}
 \begin{pmatrix}\cos\theta_e&\sin\theta_e\\
                 \sin\theta_e&-\cos\theta_e\end{pmatrix}.
\]
The displayed real matrix is orthogonal, so
\(\|J_e^\perp\|_{\mathrm{op}}=w_es_e/2\) and
\(J_e^{zz}-\|J_e^\perp\|_{\mathrm{op}}
 =w_e(1-s_e)^2/4\ge0\).  The longitudinal field at site \(i\) is
\[
 \mu_i^z=\sum_{e\ni i}\frac{w_e(1-s_e^2)}4>0.
\]
Every vertex has an incident edge, so these fields are strictly
positive and Theorem~\ref{eq:strict-field} gives both uniqueness
and the asserted radius.
\end{proof}

For a common deformation parameter \(s<1\), the corollary gives
the qualitative strict-radius conclusion in
\cite[Conjecture~1]{WBG26}, with a radius that may depend on the
graph and its parameters. The edge phases are arbitrary and
independent, since the proof uses only the norm of each transverse
interaction.

For the unweighted common-parameter model, the field calculation above
and \cite[Theorem~1]{BGLW26} give
\[
 \Delta(H_s)\geq \frac{d_{\min}(G)}2(1-s^2),
 \qquad d_{\min}(G)=\min_i\deg_G(i),
\]
where $\Delta$ denotes the gap between the two lowest eigenvalues.
Thus the lower bound proposed in \cite[Conjecture~3]{WBG26} already
follows when $d_{\min}(G)\geq2$. For graphs with leaves this estimate
gives $(1-s^2)/2$; attaining the conjectured lower bound requires
additional information. For stars with at least three vertices,
equality holds in the conjectured bound by \cite[Lemma~13]{WBG26};
the single-edge gap is $1+s^2$ by its rank-one spectrum.

\section{Classical computation of principal eigenvalues and states}
\label{alg:section}

The strict spectral estimates lead to algorithms whose input consists
of local matrices, without requiring construction of an eigenvector
in dimension \(2^n\). We prove the principal-eigenvalue algorithm for
complex stability-preserving generators and deduce ground-energy
algorithms for the Hamiltonians of Section~\ref{eq:section}; the same
continuation argument will also give scalar queries of principal
states in Section~\ref{alg:states-subsection}.

The gap in Theorem~\ref{spec:gap} selects a holomorphic principal
eigenvalue throughout a right half-plane of damping parameters. A
reciprocal change of variable joins this half-plane to the disk in
which the perturbation coefficients are computable by the algorithm
of Bravyi, DiVincenzo, and Loss \cite{BDL08}, and a further explicit
change of variable permits evaluation at fixed positive damping from
logarithmically many coefficients.

\subsection{Local input and the approximation theorem}
\label{alg:input-subsection}

Identify the coefficient space
\[
 \mathcal P_n
 =\{p\in\C[z_1,\ldots,z_n]:\deg_{z_i}p\le1\}
 \cong(\C^2)^{\otimes n},
 \qquad z^\alpha\longleftrightarrow|\alpha\rangle.
\]
We use the Euclidean norm in this coefficient basis and the induced
operator norm.  Let \(\mathcal S_n\) be the set of nonzero polynomials in
\(\mathcal P_n\) that do not vanish in \(\D^n\), where
\(\D=\{z\in\C:|z|<1\}\).  The degree operator is
\begin{equation}
 E=\sum_{i=1}^n |1\rangle\langle1|_i
   =\sum_{i=1}^n z_i\partial_{z_i}.
 \label{alg:degree}
\end{equation}
All local matrices below are supplied in the displayed tensor-product
basis.  A Gaussian rational means an element of \(\mathbb Q(i)\).

\begin{theorem}
\label{alg:complex}
Fix \(d\in\mathbb N\cup\{0\}\) and positive rational numbers \(J,\mu\).
For \(n\ge1\), suppose that the input specifies
\begin{equation}
 A_0=\sum_{i=1}^n K_i+
       \sum_{\{i,j\}\in\mathcal E}K_{ij},
 \label{alg:local-input}
\end{equation}
where the graph \((\{1,\ldots,n\},\mathcal E)\) has maximum degree at
most \(d\), every local matrix has Gaussian-rational entries and operator
norm at most \(J\), and
\begin{equation}
 e^{tA_0}\mathcal S_n\subseteq\mathcal S_n
 \qquad(t\ge0).
 \label{alg:promise}
\end{equation}
Let \(a_\star\) be the eigenvalue of
\(A=A_0-2\mu E\) with largest real part.  This eigenvalue is unique
and algebraically simple.  Given a positive rational \(\delta\), a
deterministic classical algorithm returns
\(\widehat a\in\mathbb Q(i)\) such that
\[
 |\widehat a-a_\star|\le\delta.
\]
For fixed \(d,J,\mu\), its bit complexity is polynomial in \(n\),
the input bit length, and \(\delta^{-1}\).
\end{theorem}

Corollary~\ref{gen:recognition} checks \eqref{alg:promise} in
polynomial bit time from the supplied local representation. The test
is applied to the sum \(A_0\), so terms which fail to generate
preserving semigroups individually are allowed. The remaining arguments
use the given degree and norm bounds to compute perturbation coefficients.

More explicitly, the arithmetic bound proved below is
\[
 n\left(2+\frac{C_{d,J,\mu}n}{\delta}\right)^%
 {O_d(1+(J/\mu)^2)},
\]
up to polynomial factors in the coefficient order, where
\(C_{d,J,\mu}\) is independent of \(n,\delta\). The bit overhead is
polynomial in the input length and is proved in
Appendix~\ref{alg:energy-bit-details}.
For \(\delta=n^{-b}\), the arithmetic bound becomes
\(n^{O_{d,b}(1+(J/\mu)^2)}\), up to a parameter-dependent factor
and polynomial factors in the coefficient order. The exponent grows
with \(J/\mu\), so the polynomial-time assertion treats
\(\mu>0\) as fixed.

\subsection{Suzuki--Fisher ground energies}
\label{alg:sf-subsection}

The preceding result applies to the Suzuki--Fisher class of
Theorem~\ref{eq:sf}.  We keep its sign convention:
the longitudinal term is
\(-\sum_i\mu_i^z Z_i\), with \(\mu_i^z\ge0\).
The Lee--Yang theorem gives disk stability preservation by the
Gibbs semigroup \(e^{-tH}\); see
\cite{SuzukiFisher71,WBG26,BGLW26}.

\begin{corollary}
\label{alg:sf}
Fix a maximum degree \(d\), a positive rational bound \(J\) on
the absolute values of the Pauli coefficients, and a positive
rational \(\mu\).
Let \(H\) be a Suzuki--Fisher Hamiltonian with rational Pauli
coefficients on an \(n\)-vertex graph of maximum degree at most
\(d\), and suppose
\[
 \mu_i^z\ge\mu\qquad(1\le i\le n).
 \]
For every positive rational \(\delta\), its ground energy
\(E_0(H)\) can be approximated to absolute error at most
\(\delta\) by a deterministic classical algorithm polynomial in
\(n\), the input bit length, and \(\delta^{-1}\).
The polynomial may depend on \(d,J,\mu\).
\end{corollary}

\begin{proof}
Define
\[
 V=H+\mu\sum_i Z_i.
 \]
Its longitudinal fields are \(\mu_i^z-\mu\ge0\), and its
interaction inequalities are unchanged, so \(V\) is again
Suzuki--Fisher and \(A_0=-V\) satisfies \eqref{alg:promise}.
Combining the Pauli terms on each edge or site gives Gaussian-rational
local matrices with a norm bound depending only on \(J,\mu\),
without changing the degree bound. The sum of the absolute Pauli
coefficients provides a rational norm bound at each site or edge,
avoiding any need to compute matrix norms exactly.

Since \(2E=n\Id-\sum_i Z_i\),
\begin{equation}
 A_0-2\mu E=-H-\mu n\Id.
 \label{alg:sf-shift}
\end{equation}
The matrix on the right is Hermitian with largest eigenvalue
\(-E_0(H)-\mu n\), so we apply Theorem~\ref{alg:complex} with
error \(\delta\) and, from its output \(\widehat a\), return
\[
 \widehat E=-\Re\widehat a-\mu n.
 \]
The error bound follows from \eqref{alg:sf-shift}, with the
complexity of Theorem~\ref{alg:complex}.
\end{proof}

Corollary~\ref{alg:sf} includes nonstoquastic interactions and
does not require conservation of total degree.  Its computational
output and parameter dependence differ from several earlier
results.  Bravyi and Gosset \cite{BG17} give randomized
algorithms for a two-axis stoquastic family on arbitrary graphs.
Harrow, Mehraban, and Soleimanifar \cite[Section~7 of the full version]{HMS20}
give an all-temperature quasipolynomial partition-function
algorithm for a number-conserving XXZ subclass.  Their
fugacity-polynomial reduction uses the commutation of the
interaction with total magnetization.
Bravyi, Gosset, Liu, and Wong \cite{BGLW26} give a polynomial-time
quantum ground-energy algorithm for the full Suzuki--Fisher
family, including zero fields and inverse-polynomial absolute
accuracy.  Corollary~\ref{alg:sf} gives a deterministic classical algorithm
under the bounded-degree and fixed-positive-field assumptions.

\subsection{Energy-value approximation at zero field}
\label{alg:value-subsection}

A fixed positive field can also be used as a controlled
perturbation when the requested error is proportional to the
number of sites.

\begin{corollary}
\label{alg:value}
Fix a maximum degree \(d\), a bound \(J\) on the absolute
Pauli coefficients, and a positive rational \(\varepsilon\).
For every rational Suzuki--Fisher Hamiltonian on an
\(n\)-vertex graph satisfying these bounds, a deterministic
classical algorithm polynomial in \(n\) and the input bit length
returns \(\widehat E\in\mathbb Q\) such that
\[
 |\widehat E-E_0(H)|\le\varepsilon n.
 \]
The longitudinal fields may vanish.
The polynomial may depend on \(d,J,\varepsilon\).
\end{corollary}

\begin{proof}
It suffices to consider \(0<\varepsilon\le1\), since a smaller
accuracy parameter also meets a larger requested tolerance.
Put \(\nu=\varepsilon/4\) and
\[
 H_\nu=H-\nu\sum_i Z_i.
 \]
The perturbed Hamiltonian is Suzuki--Fisher with longitudinal fields
at least \(\nu\) and local coefficients bounded in terms of
\(J,\varepsilon\). By the variational characterization of the
lowest eigenvalue,
\[
 |E_0(H_\nu)-E_0(H)|
 \le\|H_\nu-H\|\le\nu n.
 \]
Applying Corollary~\ref{alg:sf} to estimate \(E_0(H_\nu)\) with
absolute error at most \(\varepsilon n/2\) therefore produces
a value differing from \(E_0(H)\) by at most
\[
 \frac{\varepsilon n}{4}+\frac{\varepsilon n}{2}
   =\frac{3\varepsilon n}{4}
   \le\varepsilon n.
 \]
For fixed \(d,J,\varepsilon\), both \(\nu\) and the local bounds
are fixed, making the runtime in Corollary~\ref{alg:sf} polynomial
in the input size and \(n\).
\end{proof}

For Quantum MaxCut and EPR, the additive energy estimate gives
a multiplicative value guarantee on unweighted graphs.
For a graph \(G=(V,\mathcal E)\) with nonnegative rational weights, let
\[
 Q_G=\sum_{\{i,j\}\in\mathcal E}
      w_{ij}|\Psi^-\rangle\langle\Psi^-|_{ij},
 \qquad
 |\Psi^-\rangle=\frac{|01\rangle-|10\rangle}{\sqrt2},
 \]
and write \(\operatorname{OPT}(G)=\lambda_{\max}(Q_G)\).
Similarly, define the EPR objective
\[
 P_G=\sum_{\{i,j\}\in\mathcal E}
      w_{ij}|\Phi^+\rangle\langle\Phi^+|_{ij},
 \qquad
 |\Phi^+\rangle=\frac{|00\rangle+|11\rangle}{\sqrt2}.
 \]

\begin{corollary}
\label{alg:qmc}
For fixed degree and upper weight bounds and fixed rational
\(\varepsilon>0\), the optimal values of \(P_G\) on arbitrary
graphs and of \(Q_G\) on bipartite graphs admit deterministic
classical additive-\(\varepsilon |V|\) approximation in
polynomial time.
For both unweighted classes and each fixed rational \(0<\varepsilon<1\),
it returns a rational value \(L\) satisfying
\[
 (1-\varepsilon)\operatorname{OPT}\le L\le\operatorname{OPT},
\]
where \(\operatorname{OPT}\) denotes the largest eigenvalue of the
respective objective.
\end{corollary}

\begin{proof}
The negative EPR objective is Suzuki--Fisher up to a scalar:
\begin{equation}
 -P_G
  =-\frac14\sum_{\{i,j\}\in\mathcal E}w_{ij}
       \bigl(\Id+X_iX_j-Y_iY_j+Z_iZ_j\bigr).
 \label{alg:epr-expansion}
\end{equation}
The transverse matrix on each edge has operator norm \(w_{ij}/4\),
equal to its longitudinal interaction coefficient, so
Corollary~\ref{alg:value} applies after subtracting the known scalar
\(-\frac14\sum_{\{i,j\}}w_{ij}\). Restoring this scalar gives
the asserted estimate for \(P_G\).

For a bipartite graph with parts \(A,B\), the unitary
\(U=\prod_{i\in A}Y_i\) maps the singlet vector on every edge
to the EPR vector up to a scalar of modulus one, giving
\[
 UQ_GU^*=P_G.
 \]
Thus \(Q_G\) has the same spectrum and inherits the additive
estimate, as in the local change of basis used in
\cite[Corollary~2]{BGLW26}.

For the unweighted assertions, remove isolated vertices and let
\(N,m\) be the remaining vertex and edge counts; return zero if
\(m=0\), and otherwise use \(N\le2m\). The EPR product vector
\(|0^N\rangle\) has expectation \(m/2\), as does the bipartite
Quantum MaxCut product vector that is \(0\) on one part and \(1\)
on the other. Hence \(\operatorname{OPT}\ge m/2\ge N/4\) in
both cases.
If the additive algorithm is run with error
\(\eta=\varepsilon N/8\) and returns \(\widehat v\), the choice
\(L=\max(0,\widehat v-\eta)\) satisfies
\[
 L\le\operatorname{OPT},\qquad
 L\ge\operatorname{OPT}-2\eta
   \ge(1-\varepsilon)\operatorname{OPT}.
\]
For each fixed \(\varepsilon\), the running time is polynomial.
\end{proof}

The multiplicative estimate uses the lower bound
\(\operatorname{OPT}\ge N/4\) for the unweighted objectives. For
general bounded weights, the conclusion is the additive estimate.
These are value algorithms; state-producing approximations are studied
in \cite{ALMPSS26}. The inverse-polynomial absolute-accuracy question
of \cite{MarwahaSud26} would require a field of order \(\delta/n\)
in this reduction, making the guaranteed running-time exponent depend
on \(n\) through \(J/\mu\). The polynomial-time bound proved here
therefore applies at fixed relative error.

\subsection{Proof of the principal-eigenvalue algorithm}

We construct an eigenvalue function on a domain containing both the
perturbative disk and the positive real axis, using the gap to identify
the continuation of its computable Taylor series. An explicit map from
the unit disk will then give a truncation bound at the desired field.

\subsubsection{Perturbation coefficients and complexification}
\label{alg:perturbation-subsection}

We use the perturbation expansion about the product vacuum.
For a graph on \(N\) qubits with maximum degree at most \(D\ge1\), set
\begin{equation}
 H_0=2\sum_{u=1}^N |1\rangle\langle1|_u,
 \qquad \Omega=|0\rangle^{\otimes N}.
 \label{alg:unperturbed}
\end{equation}
For any matrix \(V\), the eigenvalue of \(H_0+xV\) issuing from the
simple eigenvalue \(0\) has a convergent germ at \(x=0\), which we write
\begin{equation}
 F_V(x)=\sum_{k\ge1}e_k(V)x^k.
 \label{alg:energy-germ}
\end{equation}

\begin{proposition}[Bravyi--DiVincenzo--Loss]
\label{alg:bdl}
Let \(W=\sum_{\{u,v\}}W_{uv}\) be Hermitian, with each
\(W_{uv}\) Hermitian and of norm at most \(L>0\).
For \(H_0\) as in \eqref{alg:unperturbed}, put
\[
 R_L=\frac{2^{-16}}{DL}.
 \]
Then
\begin{equation}
 |e_k(W)|\le2^{-15}N R_L^{-k}\qquad(k\ge1).
 \label{alg:bdl-bound}
\end{equation}
For every \(p\ge1\), the coefficients
\(e_1(W),\ldots,e_p(W)\) can be computed in
\(N\exp(O_D(p))\) arithmetic operations.
\end{proposition}

These are Lemma~6 and Theorem~3 of \cite{BDL08}, specialized to
unperturbed gap \(2\).  Their Theorem~2 also gives convergence of the
energy expansion throughout the corresponding high-field disk.
The formal coefficient recursions allow a complexification of these
estimates, which we need in order to continue the eigenvalue for
arbitrary complex local matrices.

\begin{lemma}
\label{alg:complexification}
Let \(V=\sum_{\{u,v\}}V_{uv}\), where the local matrices are arbitrary
complex matrices with \(\|V_{uv}\|\le J\).
The germ \eqref{alg:energy-germ} extends to a holomorphic eigenvalue
function on
\[
 |x|<\rho,\qquad \rho=\frac{2^{-18}}{DJ},
 \]
and its coefficients satisfy
\begin{equation}
 |e_k(V)|\le2^{-15}N\rho^{-k}\qquad(k\ge1).
 \label{alg:complex-bound}
\end{equation}
If the entries of the local matrices are Gaussian rational, for each
integer \(p\ge1\) its first \(p\) coefficients can be computed exactly using
\(N\exp(O_D(p))\) arithmetic operations on Gaussian rationals.
\end{lemma}

\begin{proof}
Let \(Q=\Id-|\Omega\rangle\langle\Omega|\), and let \(S\) be the
inverse of \(H_0\) on \(Q(\C^2)^{\otimes N}\), extended by zero on
\(\Omega\).  Normalize the eigenvector germ by
\[
 v(x)=\Omega+\sum_{k\ge1}v_kx^k,\qquad
 \langle\Omega,v_k\rangle=0.
 \]
Writing \(v_0=\Omega\), coefficient comparison in
\((H_0+xV)v(x)=F_V(x)v(x)\) gives
\begin{align}
 e_k(V)&=\langle\Omega,Vv_{k-1}\rangle,
 \label{alg:formal-energy}\\
 v_k&=-S Q\left(
        Vv_{k-1}-\sum_{j=1}^{k-1}e_j(V)v_{k-j}
       \right).
 \label{alg:formal-vector}
\end{align}
Since the vacuum coordinate functional is fixed, induction shows that
\(v_k\) and \(e_k(V)\) are homogeneous polynomials of degree \(k\)
in the entries of \(V\). The recursion uniquely determines the
normalized formal eigenpair, so the coefficients computed in
Proposition~\ref{alg:bdl} are the restrictions of these polynomials
to Hermitian inputs.

Decompose the local terms as
\[
 B_{uv}=\frac{V_{uv}+V_{uv}^*}{2},\qquad
 C_{uv}=\frac{V_{uv}-V_{uv}^*}{2i},
 \]
and write \(V=B+iC\) for the resulting sums, whose local terms
are Hermitian of norm at most \(J\). For real \(\theta\), the
perturbation
\[
 W_\theta=B\cos\theta+C\sin\theta
 =\frac12 e^{-i\theta}V+\frac12 e^{i\theta}V^*
 \]
is Hermitian, with each local norm at most \(2J\).
A monomial of degree \(k\) contributes Fourier frequency \(-k\)
only when its \(V\)-term is selected in every factor. By homogeneity,
the coefficient of that frequency in
\(e_k(W_\theta)\) is \(2^{-k}e_k(V)\), so
\begin{equation}
 e_k(V)=\frac{2^k}{2\pi}
       \int_0^{2\pi} e^{ik\theta}e_k(W_\theta)\,d\theta.
 \label{alg:fourier}
\end{equation}
Applying \eqref{alg:bdl-bound} with \(L=2J\) gives radius
\(2^{-17}/(DJ)\), and hence \eqref{alg:fourier} yields
\[
 |e_k(V)|
 \le 2^k\,2^{-15}N
          \left(\frac{2^{-17}}{DJ}\right)^{-k}
 =2^{-15}N\rho^{-k}.
 \]
The series consequently converges normally on compact subsets of
\(|x|<\rho\), agreeing near zero with the eigenvalue germ by
\eqref{alg:formal-energy}--\eqref{alg:formal-vector} and simplicity
of the vacuum eigenvalue. The identity theorem applied to
\[
 \det\bigl(H_0+xV-F_V(x)\Id\bigr)
 \]
extends the eigenvalue identity to the entire disk.

For the coefficient algorithm, take the distinct rational nodes
\[
 s_\ell=-1+\frac{2\ell}{p},\qquad 0\le\ell\le p.
 \]
At each of the \(p+1\) nodes, \(B+s_\ell C\) is a Hermitian
local perturbation with norm bound \(2J\), to which
Proposition~\ref{alg:bdl} applies. For \(k\le p\), the function
\[
 s\longmapsto e_k(B+sC)
 \]
is a polynomial of degree at most \(k\), so its values at the nodes
determine \(e_k(V)\) by rational interpolation at \(s=i\).
Both the factor \(p+1\) in the number of calls and the polynomial
interpolation cost are absorbed in \(N\exp(O_D(p))\), with all
operations performed on Gaussian rationals.
\end{proof}

\begin{remark}
\label{alg:complexification-scope}
The factor four between \(R_J\) and \(\rho\) has two sources:
the Hermitian family has local norm bound \(2J\), and extracting
the extremal Fourier coefficient contributes \(2^k\).
The global spectral identification for complex inputs will follow
from Lemma~\ref{alg:half-plane}.
\end{remark}

To incorporate one-site perturbations into the edge-based input of
Proposition~\ref{alg:bdl}, attach a private auxiliary qubit \(i'\)
to each physical site \(i\), and replace \(V_i\) by
\(V_i\otimes\Id_{i'}\) on the edge \(\{i,i'\}\). Giving each
auxiliary qubit the unperturbed term \(2|1\rangle\langle1|_{i'}\)
produces the operator
\begin{equation}
 \widetilde H(x)
  =(2E+xV)\otimes\Id_{\mathrm{aux}}
   +\Id_{\mathrm{phys}}\otimes
        2\sum_{i'}|1\rangle\langle1|_{i'}.
 \label{alg:auxiliary}
\end{equation}
The auxiliary vacuum factors out of the eigenvalue germ, so the germ
and all its coefficients agree with those of the physical operator.
The edge-based coefficient algorithm thus applies with \(N=2n\),
maximum degree at most \(D=d+1\), and the original local norm bound.

For the original \(n\)-site input, the auxiliary construction gives
the following parameters:
\begin{equation}
 N=2n,\qquad D=d+1,\qquad
 \rho=\frac{2^{-18}}{DJ},\qquad c=\frac1\mu.
 \label{alg:parameters}
\end{equation}
The number \(p\) of perturbation coefficients required by the proof
satisfies
\begin{equation}
 p=O\!\left(
   \left(1+\frac{c}{\rho}\right)^2
   \log\!\left(2+\frac{C_{d,J,\mu}n}{\delta}\right)
 \right),
 \label{alg:order-bound}
\end{equation}
for a constant \(C_{d,J,\mu}\) independent of \(n,\delta\).
The arithmetic operation count is
\(N\exp(O_D(p))\), with a polynomial bit overhead.
In particular, after substitution the coefficient computation takes
\[
 N\left(2+\frac{C_{d,J,\mu}n}{\delta}\right)^%
 {O_D((1+2^{18}DJ/\mu)^2)}
\]
arithmetic operations, up to polynomial factors in the coefficient
order. The constant implicit in \(O_D\) retains the graph-degree
dependence of the perturbation algorithm.

\subsubsection{A global eigenvalue branch}
\label{alg:branch-subsection}

\begin{lemma}
\label{alg:half-plane}
Under \eqref{alg:promise}, the matrix
\[
 A(h)=A_0-2hE
 \]
has an algebraically simple eigenvalue \(a(h)\), uniquely maximizing
real part, for every \(h\) with \(\Re h>0\).
The function \(a\) is holomorphic on that half-plane.
For \(V=-A_0\), the function
\begin{equation}
 F(x)=-x\,a(1/x),\qquad \Re x>0,
 \label{alg:reciprocal}
\end{equation}
is a holomorphic eigenvalue of \(2E+xV\).
It agrees with the vacuum germ of
Lemma~\ref{alg:complexification} near the positive real origin.
\end{lemma}

\begin{proof}
Write \(h=u+iv\), where \(u>0\).
The imaginary Euler flow acts on polynomials by
\[
 e^{-2itvE}p(z_1,\ldots,z_n)
   =p(e^{-2itv}z_1,\ldots,e^{-2itv}z_n),
 \]
and therefore preserves \(\mathcal S_n\).
By the finite-dimensional product formula,
\[
 e^{t(A_0-2ivE)}
   =\lim_{m\to\infty}
       \bigl(e^{tA_0/m}e^{-2itvE/m}\bigr)^m.
 \]
Every approximating output is stable on a stable input. By Hurwitz's
theorem the limit is stable or zero, and it is nonzero because the
limiting exponential is invertible. Thus \(A_0-2ivE\) also
generates a preserving semigroup.

By Theorem~\ref{spec:gap}, applied with damping \(u\), there is an
algebraically simple eigenvalue of
\((A_0-2ivE)-2uE=A(h)\) whose real part exceeds those of all the
others by at least \(2u\). By the implicit function theorem, this
eigenvalue is locally holomorphic, and by continuity of the spectrum
the strict real-part inequality persists in a sufficiently small
neighborhood. These
local branches agree on overlaps because each selects the unique
eigenvalue with largest real part, and therefore define a holomorphic
function on the entire half-plane.

Since the reciprocal map preserves the right half-plane,
\eqref{alg:reciprocal} is holomorphic there; from the matrix identity
\[
 -x A(1/x)=2E+xV
 \]
we obtain the eigenvalue assertion. For positive real \(x\), the
scalar \(-x\) reverses real-part order, so \(F(x)\) is the unique
eigenvalue of \(2E+xV\) with smallest real part.

To identify this branch near the origin, choose positive \(x\) with
\(x\|V\|<1/4\). Outside the disks of radius \(1/4\) about
\(\{0,2,\ldots,2n\}\), the resolvent of the normal matrix \(2E\)
has norm at most \(4\), so its product with \(xV\) has norm less than \(1\).
By the Neumann series, every eigenvalue of \(2E+xV\) lies in
these disks; this argument also covers \(V=0\).
The Riesz projection inside \(|z|=1\) has rank one:
the contour stays in the resolvent set throughout
\(2E+\tau xV\), \(0\le\tau\le1\), and its rank at \(\tau=0\)
is one.  Thus the eigenvalue near zero is the vacuum branch, whereas all
other eigenvalues have real part greater than one. Since the eigenvalue
with smallest real part is unique, this branch agrees with
\(F(x)\) on a positive real interval.
\end{proof}

By the identity theorem on the connected overlap
\(\{\Re x>0\}\cap\{|x|<\rho\}\),
Lemmas~\ref{alg:half-plane} and \ref{alg:complexification}, with
the auxiliary construction when needed, yield a single holomorphic
function on
\begin{equation}
 \Omega_\rho=\{x\in\C:\Re x>0\}\cup\{x\in\C:|x|<\rho\}.
 \label{alg:glued-domain}
\end{equation}
Throughout this domain, \eqref{alg:auxiliary} and the
characteristic-polynomial identity keep \(F(x)\) an eigenvalue
of \(\widetilde H(x)\), giving the bound
\begin{equation}
 |F(x)|\le\|\widetilde H(x)\|
          \le2N+|x|W,\qquad W=NDJ,
 \label{alg:norm-bound}
\end{equation}
with \(W\) a rational upper bound on the sum of the local norms.
The continuation step must evaluate this function at \(1/\mu\),
since the desired eigenvalue is
\begin{equation}
 a_\star=-\mu F(1/\mu).
 \label{alg:target}
\end{equation}

\subsubsection{Effective continuation}
\label{alg:continuation-subsection}

The continuation step is a scalar analytic estimate, for which we
give the map and truncation error explicitly.

\begin{lemma}
\label{alg:map}
Let \(\rho,c>0\), and suppose that \(F\) is holomorphic on
\(\Omega_\rho\).
Define
\begin{equation}
 \eta=\frac{\rho^2}{4c+2\rho},\qquad
 R=c+\eta,\qquad q=\frac cR,\qquad
 b=\frac{R^2-c^2}{R},\qquad
 \phi(w)=\frac{bw}{1-qw}.
 \label{alg:map-parameters}
\end{equation}
Then \(0<q<1\), \(\phi(\D)\subset\Omega_\rho\),
\(\phi(0)=0\), and \(\phi(q)=c\).
If \(|F(x)|\le L_0+L_1|x|\) on \(\Omega_\rho\), the function
\(G=F\circ\phi\) satisfies
\[
 |G(w)|\le M:=L_0+L_1(c+R)\qquad(w\in\D).
 \]
Writing \(F(x)=\sum_{j\ge0}f_jx^j\) at zero and
\(G(w)=\sum_{k\ge0}g_kw^k\), one has \(g_0=f_0\) and
\begin{equation}
 g_k=\sum_{j=1}^k
        f_j b^j q^{k-j}\binom{k-1}{j-1}
       \qquad(k\ge1).
 \label{alg:coefficient-transform}
\end{equation}
For every \(p\ge0\),
\begin{equation}
 \left|F(c)-\sum_{k=0}^p g_kq^k\right|
    \le\frac{Mq^{p+1}}{1-q}.
 \label{alg:tail}
\end{equation}
\end{lemma}

\begin{proof}
The definitions give \(q\in(0,1)\), \(\eta\le\rho/2\), and
\[
 0<R^2-c^2=2c\eta+\eta^2<\rho^2.
 \]
If \(|x-c|<R\) and \(\Re x\le0\), then
\[
 |x|^2=|x-c|^2+2c\Re x-c^2<R^2-c^2<\rho^2.
 \]
Thus the disk of center \(c\) and radius \(R\) lies in
\(\Omega_\rho\).
The formula
\[
 \phi(w)=c+R\frac{w-q}{1-qw}
 \]
shows that \(\phi\) maps \(\D\) biholomorphically onto this
disk, with the stated values at \(0\) and \(q\). The inequality
\(|\phi(w)|<c+R\) then gives the bound on \(G\).

For \(j\ge1\), expansion at zero gives
\[
 \phi(w)^j=b^j w^j(1-qw)^{-j}
  =b^j\sum_{k\ge j}
       \binom{k-1}{j-1}q^{k-j}w^k.
 \]
Only \(j\le k\) contributes to the coefficient of \(w^k\),
which proves \eqref{alg:coefficient-transform}.
By Cauchy's estimate on circles with radii increasing to one,
\(|g_k|\le M\); summing the geometric tail at \(w=q\), we obtain
\eqref{alg:tail}.
\end{proof}

\begin{proof}[Proof of Theorem~\ref{alg:complex}]
By Theorem~\ref{spec:gap}, \(a_\star\) exists and is unique
and algebraically simple. With \(V=-A_0\), the auxiliary construction
\eqref{alg:auxiliary} and parameters \eqref{alg:parameters},
we obtain from Lemmas~\ref{alg:complexification} and
\ref{alg:half-plane} a function \(F\) on \(\Omega_\rho\)
satisfying \eqref{alg:norm-bound}, so we may apply Lemma~\ref{alg:map} with
\[
 L_0=2N,\qquad L_1=W,\qquad
 M=2N+(c+R)W.
 \]
All the parameters in \eqref{alg:map-parameters} are rational.

Choose the smallest integer \(p\ge0\) for which
\begin{equation}
 \mu M q^{p+1}\le\delta(1-q).
 \label{alg:stopping}
\end{equation}
The algorithm finds this integer by rational multiplication and
comparison; evaluation of logarithms is unnecessary.
If \(p>0\), compute \(e_1(V),\ldots,e_p(V)\) by
Lemma~\ref{alg:complexification}, and then compute
\(g_0,\ldots,g_p\) by
\eqref{alg:coefficient-transform}.
Since \(g_0=F(0)=0\), this gives the Gaussian-rational output
\begin{equation}
 \widehat a=-\mu\sum_{k=0}^p g_kq^k.
 \label{alg:output}
\end{equation}
When \(p=0\), the same formula returns zero without a coefficient
call; in either case, \eqref{alg:target}, \eqref{alg:tail}, and
\eqref{alg:stopping} give \(|\widehat a-a_\star|\le\delta\).

To bound the number of arithmetic operations, observe that
\begin{equation}
 \frac1{1-q}=\frac R\eta
     =1+\frac{4c^2}{\rho^2}+\frac{2c}{\rho}.
 \label{alg:contraction-parameter}
\end{equation}
Since \(-\log q\ge1-q\), the stopping rule implies
\[
 p=O\!\left(
   \frac1{1-q}
   \log\!\left(2+\frac{\mu M}{\delta(1-q)}\right)
 \right).
 \]
Since \(M=O_{d,J,\mu}(n)\), this gives
\eqref{alg:order-bound}, and in particular
\(p=O_{d,J,\mu}(\log(2+n/\delta))\) for fixed \(d,J,\mu\).
The coefficient algorithm requires \(N\exp(O_D(p))\) arithmetic
operations, to which interpolation, the triangular transform, and
evaluation in \eqref{alg:output} add only polynomially many
operations in \(p\).

The coefficient computation, interpolation, and evaluation have
polynomial bit overhead by Appendix~\ref{alg:energy-bit-details},
which, together with \eqref{alg:order-bound}, gives the asserted
bit complexity.
\end{proof}

\begin{remark}
\label{alg:hermitian-radius}
In the Suzuki--Fisher application of Corollary~\ref{alg:sf}, the
perturbation \(V=H+\mu\sum_i Z_i\) is Hermitian, so the
complexification step can be omitted.  Applying
Proposition~\ref{alg:bdl} directly gives the larger guaranteed
disk \(2^{-16}/(DJ')\), where \(J'\) is the local norm bound
after the auxiliary construction.  The same continuation proof
then applies to the ground-energy branch of \(2E+xV\).
The factor-four improvement in this disk changes constants,
not the fixed-parameter complexity conclusion.
\end{remark}

\subsection{Stable tests of principal eigenvectors}
\label{alg:states-subsection}

For the same input \eqref{alg:local-input}--\eqref{alg:promise},
we can compute scalar queries of the principal eigenvectors by
continuing logarithms of stable product tests. Transposition exchanges
the two variable groups of a binary matrix kernel, so
Proposition~\ref{spec:perron} gives strictly disk-stable right
eigenvectors of both
\(A=A_0-2\mu E\) and \(A^{\mathsf T}\), which we normalize by
\[
 v_\varnothing=w_\varnothing=1.
\]
The transpose, rather than the adjoint, makes the dependence on a
complex parameter holomorphic.  A \emph{stable product test} is
\(B=\bigotimes_{i=1}^n B_i\), where \(B_i\in M_2(\C)\),
\((B_i)_{00}=1\), and
\[
 K_{B_i}(z,w)=
 (B_i)_{00}+(B_i)_{10}z+(B_i)_{01}w+(B_i)_{11}zw
\]
does not vanish in \(\D^2\), allowing singular as well as invertible
complex matrices.

\begin{theorem}
\label{alg:product-test}
Fix \(d,J,\mu\) as in Theorem~\ref{alg:complex}, and suppose that
\(A_0\) satisfies that theorem's input assumptions.  Given a stable
product test with Gaussian-rational entries, both
\[
 G_B=w^{\mathsf T}Bv,\qquad
 \Tr(B\Pi)=\frac{w^{\mathsf T}Bv}{w^{\mathsf T}v},
 \qquad \Pi=\frac{vw^{\mathsf T}}{w^{\mathsf T}v},
\]
are nonzero.  For every rational \(0<\epsilon<1\), a deterministic
classical algorithm returns Gaussian-rational estimates of these
two quantities with relative complex error at most \(\epsilon\).
Its bit complexity is
\[
 \operatorname{poly}_{d,J,\mu}
       (n,\epsilon^{-1},L_{\mathrm{in}}),
\]
where \(L_{\mathrm{in}}\) includes the local generator, the test
matrices, and the accuracy.  Relative complex error means
\(|\widehat z/z-1|\le\epsilon\).
\end{theorem}

The algorithm computes these scalar tests by continuing their
logarithms in the reciprocal-field parameter. The necessary bound is
on the logarithm of the test, rather than on the norm of the spectral
projection.

\subsubsection{A uniform domain for normalized eigenvectors}

Write \(V=-A_0\), \(H(x)=2E+xV\), and set
\begin{equation}\label{alg:state-radii}
 r_0=\frac{2^{-18}}{(d+1)J},\qquad \rho_s=\frac{r_0}{8}.
\end{equation}
To construct the eigenvectors, we use the creation coefficients in
the Kirkwood--Thomas expansion: for nonempty \(M\subseteq[n]\), put
\[
 a_M^\dagger=\prod_{i\in M}|1\rangle\langle0|_i .
\]
These operators commute, and products with overlapping supports vanish.
The vacuum-normalized eigenvector germ has a unique expression
\[
 v(x)=\exp\left(-\sum_{\varnothing\ne M\subseteq[n]}
                   C(M;x)a_M^\dagger\right)\Omega,
 \qquad C(M;x)=\sum_{k\ge1}C_k(M)x^k.
\]
Uniqueness follows by taking the finite logarithm in the nilpotent
commutative algebra generated by the \(a_i^\dagger\).

\begin{lemma}
\label{alg:creation-disk}
For every local complex \(V\) with the bounds above,
\begin{equation}\label{alg:creation-bound}
 \max_i\sum_{M\ni i}|C_k(M)|\le2^{-15}r_0^{-k},
 \qquad C_k(M)=0\quad\hbox{if }|M|>k+1.
\end{equation}
The normalized right eigenvector germs of \(H(x)\) and
\(H(x)^{\mathsf T}\) extend holomorphically to \(|x|<\rho_s\).
Their coefficient polynomials are nonzero on the closed polydisk
\(\{|z_i|\le2\}\) throughout this parameter disk.
\end{lemma}

\begin{proof}
For Hermitian perturbations with local norm bound \(L\) on a graph
of maximum degree \(D\), Lemma~5, equation~(32), and Lemma~7,
equation~(47), of \cite{BDL08}, with unperturbed gap \(2\), give
\[
 \max_i\sum_{M\ni i}|C_k(M)|\le2^{-15}R_L^{-k},
 \qquad R_L=2^{-16}/(DL),
\]
with every nonzero \(C_k(M)\) supported inside a connected set
of at most \(k+1\) vertices, and hence \(|M|\le k+1\).
For one-site terms, the private auxiliary qubits in
\eqref{alg:auxiliary} give \(D=d+1\) and a normalized eigenvector
germ equal to the physical germ tensored with the auxiliary vacuum.
Creation coefficients involving an auxiliary vertex therefore vanish,
while the remaining coefficients agree with the physical ones.

The recursion \eqref{alg:formal-energy}--\eqref{alg:formal-vector}
and the finite logarithm show that \(C_k(M)\) is a homogeneous
polynomial of degree \(k\) in the local entries of \(V\), without
their conjugates.  With \(W_\theta\) from the proof of
Lemma~\ref{alg:complexification},
\[
 C_k(M;V)=\frac{2^k}{2\pi}\int_0^{2\pi}
              e^{ik\theta}C_k(M;W_\theta)\,d\theta.
\]
Taking absolute values, summing over \(M\ni i\), and applying the
Hermitian estimate with \(L=2J\) gives
\eqref{alg:creation-bound}, with its support assertion inherited
from the same identity. Applying the argument to \(V^{\mathsf T}\)
proves the corresponding estimates for the left germ, and for
\(|x|<r_0/8\) we obtain
\begin{equation}\label{alg:weighted-creation}
 \max_i\sum_{M\ni i}4^{|M|}|C(M;x)|
 \le 2^{-15}\sum_{k\ge1}4^{k+1}(|x|/r_0)^k
 <2^{-13}.
\end{equation}
The series converge locally uniformly.

To deduce zero-freeness, use the following elementary polymer
estimate. For weights \(u_M\) on the nonempty subsets of a finite
set, define
\[
 Z(S)=\sum_{\substack{\mathcal A\text{ a collection of}\\
              \text{pairwise disjoint nonempty subsets of }S}}
                       \prod_{M\in\mathcal A}u_M .
\]
Under the bound
\(\sup_i\sum_{M\ni i}|u_M|2^{|M|-1}\le1/2\), induction on
\(|S|\) gives
\begin{equation}\label{alg:polymer-ratio}
 Z(S)\ne0,\qquad
 \left|\frac{Z(S)}{Z(S\setminus\{i\})}-1\right|\le\frac12
 \quad(i\in S).
\end{equation}
Indeed,
\[
 Z(S)=Z(S\setminus\{i\})+
           \sum_{\substack{M\subseteq S\\i\in M}}u_M Z(S\setminus M).
\]
By the induction hypothesis, deleting the \(|M|-1\) additional
vertices gives
\(\left|Z(S\setminus M)/Z(S\setminus\{i\})\right|
 \le2^{|M|-1}\).
This bounds the displayed ratio within \(1/2\) of one, establishing
nonvanishing and completing the induction from the empty set.

The exponential formula for \(v(x)\) is this polymer sum with
\(u_M=-C(M;x)z^M\), since its factorials cancel the orderings of
pairwise disjoint sets. When \(|z_i|\le2\), the left side of the
polymer condition is at most half the left side of
\eqref{alg:weighted-creation}, and is therefore less than \(1/2\),
giving the asserted zero-freeness.

The convergent creation series and their finite exponentials define
holomorphic vectors on \(|x|<\rho_s\), with formal eigenvalue
given by the energy series of Lemma~\ref{alg:complexification}.
Their eigenvector equations hold near zero and extend throughout
the disk by the identity theorem.
\end{proof}

\begin{lemma}\label{alg:test-domain}
The vectors in Lemma~\ref{alg:creation-disk} and the normalized
principal vectors of \(A(1/x)\) glue to holomorphic vectors
\(v(x),w(x)\) on
\[
 \Omega_{\rho_s}=\{\Re x>0\}\cup\{|x|<\rho_s\}.
\]
For every stable product test,
\begin{equation}\label{alg:test-function}
 G_B(x)=w(x)^{\mathsf T}Bv(x)
\end{equation}
is nonvanishing on this domain, \(G_B(0)=1\), and
\begin{equation}\label{alg:test-log-bound}
 |G_B(x)|\le4^n,\qquad
 \Re F_B(x)\le n\log4,\quad F_B=\log G_B,\quad F_B(0)=0 .
\end{equation}
\end{lemma}

\begin{proof}
The principal eigenvalue is simple in the half-plane, so it admits
local analytic eigenvectors whose vacuum coordinates are nonzero
by strict stability of their coefficient polynomials. Vacuum
normalization makes the local choices agree on overlaps, and the
same argument applied to \(A(h)^{\mathsf T}\) gives the left
vectors. On a small positive real interval, the Riesz-projection
argument of Lemma~\ref{alg:half-plane} identifies these vectors
with the vacuum germs. The identity theorem on the connected
overlap then glues the two constructions, using the size-independent
radius \(\rho_s\) supplied by Lemma~\ref{alg:creation-disk}.

Both \(v(x)\) and \(w(x)\) have polydisk radius greater than
one in the half-plane, and radius at least two in the small disk
by Lemma~\ref{alg:creation-disk}. Contracting their legs against
the two legs of each \(K_{B_i}\) therefore pairs radii whose
product is greater than one, so polarized apolarity
(Lemma~\ref{pre:grace}) makes the resulting scalar
\eqref{alg:test-function} nonzero. This strict separation of radii
also applies when the test matrices are singular.

For the bound, a multiaffine disk-stable polynomial
\(f(z)=\sum_S f_Sz^S\) with \(f_\varnothing=1\) has
\(|f_S|\le1\). Indeed, set the variables outside \(S\) to zero
and the remaining variables to \(t\); if \(f_S\ne0\), the
resulting polynomial has degree \(|S|\), constant term one,
leading coefficient \(f_S\), and roots of modulus at least one.
Taking the product of the roots proves the bound, which also holds when
\(f_S=0\). Applying it to \(v,w\) and every \(K_{B_i}\) bounds
each of the \(4^n\) terms in \(w^{\mathsf T}Bv\) by one, while
\(G_B(0)=\prod_i(B_i)_{00}=1\). Since the domain is star-shaped
about zero, the nonvanishing function \(G_B\) has a unique
holomorphic logarithm vanishing there, whose real part satisfies
\eqref{alg:test-log-bound}.
\end{proof}

\begin{lemma}
\label{alg:log-continuation}
Use the parameters \eqref{alg:map-parameters} with
\(\rho=\rho_s\) and \(c=1/\mu\).  Write
\(F_B(x)=\sum_{j\ge1} f_jx^j\) and
\(F_B(\phi(u))=\sum_{k\ge1}g_ku^k\).
The transform \eqref{alg:coefficient-transform} remains valid,
and
\begin{equation}\label{alg:log-tail}
 |g_k|\le4n,\qquad
 \left|F_B(c)-\sum_{k=1}^p g_kq^k\right|
       \le \frac{4nq^{p+1}}{1-q}.
\end{equation}
In particular \(|F_B(c)|\le Kn\), where \(K=4q/(1-q)\).
\end{lemma}

\begin{proof}
Lemma~\ref{alg:map} puts \(\phi(\D)\) in \(\Omega_{\rho_s}\),
where the logarithmic bound makes
\(U(u)=2n-F_B(\phi(u))\) holomorphic with nonnegative real part
and \(U(0)=2n\). On every circle of radius \(r<1\),
\(\Re U(re^{it})\) has average \(2n\) and \(k\)-th Fourier
coefficient \(-g_kr^k/2\) for \(k\ge1\), giving
\(|g_k|r^k\le4n\). Letting \(r\uparrow1\) and summing the
geometric tail proves the bounds; the coefficient transform is the
formal identity from Lemma~\ref{alg:map}.
\end{proof}

\subsubsection{Computing the connected coefficients}

The coefficients of \(F_B\) can be computed on small connected
subsets, without expanding the full eigenvectors, by the method of
connected logarithms for multiplicative analytic graph functions;
compare \cite{PatelRegts17,YYZ22}. Multiplicativity follows from
the formal expansion at the vacuum, so the induced subgraphs may
be used without a stability-preservation assumption of their own.

For \(U\subseteq[n]\), keep just the local terms supported inside
\(U\), obtaining \(V_U\), and let \(B_U=\bigotimes_{i\in U}B_i\).
Define the vacuum-normalized formal vectors \(v_U,w_U\) for
\(2E_U+xV_U\) and its transpose.  Put
\[
 F_{B,U}(x)=\log(w_U(x)^{\mathsf T}B_Uv_U(x)),\qquad
 \Phi_{B,U}(x)=\sum_{W\subseteq U}(-1)^{|U|-|W|}F_{B,W}(x).
\]
For the empty subsystem all logarithms are zero.

\begin{lemma}
\label{alg:connected-tests}
For \(k\ge1\),
\[
 [x^k]F_B(x)=
 \sum_{\substack{U\subseteq[n]\ \mathrm{connected}\\ |U|\le k+1}}
                       [x^k]\Phi_{B,U}(x).
\]
For Gaussian-rational input, the first \(p\) coefficients can
be computed exactly with
\(n\exp(O_d(p))\operatorname{poly}(L_{\mathrm{in}},n,p)\)
bit operations.
\end{lemma}

\begin{proof}
When the induced graph on \(U\) splits into components, both the
normalized formal vectors and the product test factor across them.
Thus \(F_{B,U}\) is the sum of the component logarithms, and
subset inversion gives \(\Phi_{B,U}=0\) for disconnected \(U\).

For the order bound, give each local perturbation term its own
parameter, so that the formal eigenvector recursion makes every
order-\(k\) coefficient homogeneous of degree \(k\) in these
parameters. If a parameter monomial has a disconnected union of
supports, setting all unused parameters to zero factors the vectors
across those components and makes the logarithm additive; its
coefficient must therefore vanish. A connected union of at most
\(k\) one- or two-site supports has at most \(k+1\) vertices.
Moreover, the coefficient of a monomial with support union \(T\)
is unchanged in every induced subsystem containing \(T\), so the
alternating sum over \(W\subseteq U\) retains it only when
\(T=U\). This gives both vanishings and the stated formula.

For an explicit implementation, apply
\eqref{alg:formal-energy}--\eqref{alg:formal-vector} and their
transposes through order \(p\) in each subsystem of size
\(m\le p+1\). Writing
\[
 a_k=\sum_{i+j=k}w_{U,i}^{\mathsf T}B_Uv_{U,j},
 \qquad \log\left(1+\sum_{k\ge1}a_kx^k\right)
                  =\sum_{k\ge1}f_kx^k,
\]
coefficient comparison after logarithmic differentiation gives
\begin{equation}\label{alg:formal-test-log}
 f_k=a_k-\frac1k\sum_{j=1}^{k-1}j f_j a_{k-j}.
\end{equation}
The vectors have dimension at most \(2^{p+1}\), so even dense
matrix arithmetic and enumeration of all \(W\subseteq U\)
require only \(\exp(O(p))\operatorname{poly}(p)\) operations per
\(U\). Connected subsets of size at most \(p+1\) can be enumerated,
with deduplication, through spanning-tree traversals of length at most
\(2p\), taking \(n\exp(O_d(p))\) time; for \(d=0\), only
singletons occur.

Appendix~\ref{alg:test-bit-details} bounds all denominators and
intermediate numerators by polynomial bit lengths, proving the
asserted bit complexity.
\end{proof}

\begin{proof}[Proof of Theorem~\ref{alg:product-test}]
Use \(c=1/\mu\), \(\rho_s\) from \eqref{alg:state-radii}, and the
rational map parameters of Lemma~\ref{alg:map}.
Given \(0<\tau<1\), choose the least \(p\ge0\) for which
\[
 4nq^{p+1}\le\tau(1-q).
\]
Rational multiplication and comparison find \(p\);
\eqref{alg:contraction-parameter} gives
\(p=O_{d,J,\mu}(\log(2+n/\tau))\).
Use Lemma~\ref{alg:connected-tests}, the triangular transform,
and \eqref{alg:log-tail} to compute a Gaussian rational \(S_B\)
with \(|S_B-F_B(c)|\le\tau\) and \(|S_B|\le Kn+1\),
at polynomial bit cost by the preceding lemma and the rational
coefficient transform.

To obtain a rational approximation to the exponential, set
\(T_0=\lceil Kn+1\rceil\) and choose the least integer \(\ell\) with
\(2^{-\ell}\le\tau\).
Truncate the Taylor series of \(e^{S_B}\) at degree
\[
 M=\left\lceil10(2T_0+\ell+1)\right\rceil
\]
by taking \(\sum_{j=0}^{M}S_B^j/j!\). The remainder is bounded
by \(e^{T_0}T_0^{M+1}/(M+1)!<\tau e^{-T_0}\), using
\(r!\ge(r/3)^r\); since \(|e^{S_B}|\ge e^{-T_0}\), the
relative error is less than \(\tau\). All terms are Gaussian
rational, with \(M=O_{d,J,\mu}(n+\log\tau^{-1})\) and
polynomial bit lengths for their numerators and denominators.
Combining this truncation error with
\(|e^{S_B}/G_B(c)-1|\le e^\tau-1\), we choose
\(\tau=\epsilon/32\) to leave the required margin.
Repeat for \(B=I\), and divide the two nonzero rational
estimates to obtain the estimate of \(G_B/G_I\).
For errors at most \(\epsilon/8\) in each factor, the ratio
error is at most
\((\epsilon/4)/(1-\epsilon/8)<\epsilon\).

Lemma~\ref{alg:test-domain} gives nonvanishing, including for
\(B=I\), whose local kernel is \(1+zw\). The simple-eigenvalue
projection \(vw^{\mathsf T}/(w^{\mathsf T}v)\) identifies the
computed quotient with \(\Tr(B\Pi)\).
\end{proof}

\subsection{Product-state queries and equatorial measurement sampling}
\label{alg:measurements-subsection}

For \(s\in\C\) with \(|s|\le1\), define
\[
 |s\rangle=\frac{|0\rangle+s|1\rangle}{\sqrt{1+|s|^2}} .
\]
These states form the closed northern hemisphere of the Bloch sphere
in our fixed basis, throughout which amplitude and probability
queries are available. For complete two-outcome projective
measurements we use the equator, where both orthogonal states
\(|s\rangle,|-s\rangle\) satisfy the condition.

\begin{corollary}
\label{alg:sf-state}
Fix \(d,J,\mu\) as in Corollary~\ref{alg:sf}.
Let \(H\) be a rational Suzuki--Fisher Hamiltonian on \(n\ge1\)
qubits satisfying that corollary's hypotheses, and let
\(0<\epsilon<1\) be rational.
Its unique unit ground vector \(\psi\) can be normalized by
\(\langle0^n|\psi\rangle>0\).  The query algorithms in (1)--(2)
have bit complexity
\(\operatorname{poly}_{d,J,\mu}(n,\epsilon^{-1},L_{\mathrm{in}})\).
\begin{enumerate}
\item Given \(s_1,\ldots,s_n\in\mathbb Q(i)\) with
      \(|s_i|\le1\), approximate the nonzero complex amplitude
      \(\langle s_1\cdots s_n|\psi\rangle\) with relative
      complex error at most \(\epsilon\).
\item Given \(S\subseteq[n]\) and \(s_i\in\mathbb Q(i)\),
      \(|s_i|\le1\) for \(i\in S\), approximate the positive
      probability
      \[
      \left\langle\psi\left|
        \bigotimes_{i\in S}|s_i\rangle\langle s_i|
                \otimes I_{S^c}\right|\psi\right\rangle
      \]
      by a positive rational with relative error at most
      \(\epsilon\).
\item Given a polynomial-time adaptive strategy that measures
      each site at most once, in a basis
      \(\{|s\rangle,|-s\rangle\}\) with
      \(s\in\mathbb Q(i)\), \(|s|=1\), sample its outcome
      record by a classical randomized algorithm with total
      variation error at most \(\epsilon\).
\end{enumerate}
The strategy in (3) has a finite description, takes previous
outcomes as input, and outputs the next unmeasured site and a
phase.  More precisely, if its total evaluation time along
any record is at most \(T_{\rm strategy}\), and all generated
phases have bit length at most \(L_{\rm phase}\), the sampler
uses at most
\[
 T_{\rm strategy}+
 \operatorname{poly}_{d,J,\mu}
       (n,\epsilon^{-1},L_{\rm in}+L_{\rm phase})
\]
bit operations.  In particular the sampler is polynomial-time
for each fixed polynomial-time strategy.
\end{corollary}

\begin{proof}
Use \(A_0=-(H+\mu\sum_iZ_i)\) as in
\eqref{alg:sf-shift}.  Its local norm bound is a fixed rational
\(J'\) depending on the Pauli bound \(J\) and on \(\mu\);
all radii in the preceding lemmas are evaluated with this
bound.
At positive real \(x\), \(H(x)\) is Hermitian, so its
vacuum-normalized transpose eigenvector is \(w=\bar v\).
Consequently \(G_I=\|v\|^2>0\) and
\(\psi=v/\sqrt{G_I}\).  Along the positive real path the
normalized logarithm \(F_I\) is the real logarithm, because
it starts at zero and \(G_I\) stays positive.

For amplitudes take
\[
 B_i^{\mathrm{amp}}=
 \begin{pmatrix}1&\bar s_i\\0&0\end{pmatrix}.
\]
Its kernel \(1+\bar s_iw\) is disk-stable, and
\[
 \langle s_1\cdots s_n|\psi\rangle
 =\left(\prod_i(1+|s_i|^2)^{-1/2}\right)
       \exp(F_{B^{\mathrm{amp}}}(c)-F_I(c)/2).
\]
For the event in (2), take \(B_i=I\) off \(S\) and
\[
 B_i^{\mathrm{prob}}=
 \begin{pmatrix}1&\bar s_i\\s_i&|s_i|^2\end{pmatrix}
 =(1+|s_i|^2)|s_i\rangle\langle s_i|
 \quad(i\in S).
\]
The factorization \((1+s_iz)(1+\bar s_iw)\) makes its kernel
stable, and gives
\[
 P_S=\left(\prod_{i\in S}(1+|s_i|^2)^{-1}\right)
                    \exp(F_B(c)-F_I(c)).
\]
Nonvanishing follows from Lemma~\ref{alg:test-domain};
since a projector expectation is nonnegative, we obtain \(P_S>0\).

Compute the logarithms with a fixed fraction of the requested
error, using the constructive proof of
Theorem~\ref{alg:product-test}.  For (2), take the real part
of the logarithmic difference before exponentiating; its error
does not increase because the exact difference is real on
the positive path.  For either logarithmic difference use
\(T_0=\lceil2Kn+2\rceil\) in that finite Taylor procedure.
In case (2), the rational estimate is real and positive once its
relative error is less than one, and the rational probability
normalization introduces no further error.
For (1), the product \(R_s=\prod_i(1+|s_i|^2)\) is a positive
rational between \(1\) and \(2^n\).  Rational bisection for
\(\sqrt{R_s}\), followed by inversion, computes its reciprocal
square root to any required relative error with
\(O(n+\log\epsilon^{-1}+L_{\mathrm{in}})\) bits and
polynomial arithmetic cost.  Splitting the error budget
between the two logarithms, exponentiation and this factor
proves (1) and (2), including the phase assertion.

For (3), fix an already generated history, whose two extensions have
true probabilities \(p_+,p_->0\), and apply (2) to obtain positive
estimates \(\widehat p_\pm\) with relative error at most
\(\eta=\epsilon/(16n)\).
Normalizing their sum gives a conditional probability with
error at most \(2\eta/(1-\eta)<4\eta\), regardless of the
probability of the parent history.  Round it to a dyadic
probability within \(\epsilon/(4n)\) and draw the corresponding
Bernoulli variable using
\(O(\log(n/\epsilon))\) unbiased random bits.
On a fixed history, projectors on distinct sites commute, so
its probability is exactly the product event in (2), even
though the bases were chosen adaptively.
Coupling the two processes at each step bounds their total
variation distance by the sum of the at most \(n\) conditional
errors, which is less than \(\epsilon\).
There are at most \(2n\) queries, whose accuracy and phase bit
lengths give the displayed polynomial overhead, in addition to the
strategy evaluation time \(T_{\rm strategy}\).
\end{proof}

For Lee--Yang states with a fixed spatial radius \(r>1\),
\cite[Section~4, Theorems~5--6]{WBG26} gives state-dependent
quasipolynomial circuits for relative \(X\)-basis amplitudes,
with probabilities and phases in the state-preparation proof; see also
\cite{WongThesis26}. Corollary~\ref{alg:sf-state} takes the local
Hamiltonian as input and obtains polynomial complexity at fixed positive
field from the common reciprocal-field domain. The connected-coefficient
and interpolation methods have antecedents in
\cite{PatelRegts17,YYZ22,WildAlhambra23}, including thermal product
measurements in \cite[Theorem~15]{YYZ22}.

The allowed projectors have stable kernels and nonzero vacuum entries,
which gives the northern-hemisphere queries and equatorial sampling in
the corollary. In particular, the projector onto \(|1\rangle\) is
excluded. Positive field is used throughout the state-query argument;
at zero field Section~\ref{alg:value-subsection} gives energy values.
The quantum algorithms of \cite{BGLW26,RT26} cover other parameter
ranges and state-preparation tasks.

\section{Sector growth and token-graph spectra}
\label{sector:section}

For evolutions preserving nonnegative coefficients and total degree,
Newton's inequalities constrain the coefficient sums in adjacent
homogeneous sectors, whose long-time growth rates are the spectral
bounds of the sector generators. This gives a concavity theorem
for those bounds, which we apply to weighted token graphs.

\begin{theorem}\label{sector:growth}
Let \(V_\kappa\) be a finite real coordinate box and
\(N=\sum_i\kappa_i\). Suppose that \(A\) has nonnegative off-diagonal
entries in the monomial basis, preserves every homogeneous subspace
\(V_{\kappa,k}\), and that \(e^{tA}\) preserves nonnegative real
stability for every \(t\ge0\). Write
\[
 s_k=\max\{\Re\lambda:\lambda\in
             \operatorname{spec}(A|_{V_{\kappa,k}})\},
 \qquad 0\le k\le N.
\]
Then
\[
 2s_k\ge s_{k-1}+s_{k+1}\qquad(1\le k<N).
\]
Neither irreducibility nor symmetry of \(A\) is assumed.
\end{theorem}

\begin{proof}
The stable polynomial \(f_0(z)=\prod_i(1+z_i)^{\kappa_i}\)
has a strictly positive coefficient at every allowed monomial.
Choose \(\alpha\) so that \(A+\alpha I\) is entrywise nonnegative;
then \(e^{tA}=e^{-\alpha t}e^{t(A+\alpha I)}\) is entrywise
nonnegative with strictly positive diagonal, and consequently
\(f_t=e^{tA}f_0\) retains all these positive coefficients.
Its diagonal restriction is a degree-\(N\) polynomial
\[
 P_t(x)=f_t(x,\ldots,x)=\sum_{k=0}^N c_k(t)x^k,
 \qquad c_k(t)>0,
\]
whose roots are all strictly negative.

Newton's inequalities give
\begin{equation}\label{sector:newton}
 \left(\frac{c_k(t)}{\binom Nk}\right)^2
 \ge
 \frac{c_{k-1}(t)}{\binom N{k-1}}
 \frac{c_{k+1}(t)}{\binom N{k+1}}.
\end{equation}
To verify the normalization, write the derivative
\(R=P_t^{(k-1)}\), which has \(m=N-k+1\) negative roots, as
\(R(x)=R(0)\prod_{j=1}^m(1+u_jx)\) with \(u_j>0\).
Cauchy's inequality yields
\[
 \sum_{i<j}u_iu_j
 \le\frac{m-1}{2m}\left(\sum_i u_i\right)^2.
\]
Substituting the first three coefficients of \(R\), namely
\((k-1)!c_{k-1}\), \(k!c_k\), and
\((k+1)!c_{k+1}/2\), gives \eqref{sector:newton} with the stated
binomial normalization.

Let \(A_k=A|_{V_{\kappa,k}}\), and let \(v_k\) be the coefficient
vector of the degree-\(k\) part of \(f_0\). Degree preservation gives
\[
 c_k(t)=\mathbf1^{\mathsf T}e^{tA_k}v_k.
\]
Since every entry of \(v_k\) is positive, this quantity is bounded
above and below by fixed positive multiples of the sum of all
entries of \(e^{tA_k}\). For an entrywise nonnegative \(m_k\)-by-\(m_k\)
matrix \(M\), that sum lies between \(\|M\|_1\) and
\(m_k\|M\|_1\), where \(\|\cdot\|_1\) is the induced maximum
column-sum norm. By spectral mapping,
\(\|e^{tA_k}\|_1\ge e^{ts_k}\), while Jordan normal form implies
\[
 \|e^{tA_k}\|_1
 \le C_k(1+t^{m_k-1})e^{ts_k}\qquad(t\ge0)
\]
for a fixed constant \(C_k\). Therefore
\[
 \lim_{t\to\infty}t^{-1}\log c_k(t)=s_k.
\]
Taking logarithms in \eqref{sector:newton}, dividing by \(t\),
and letting \(t\to\infty\) removes the fixed binomial factors
and gives the asserted concavity; when \(N\le1\), the conclusion
has no inequalities.
\end{proof}

The passage from log-concavity of coefficient sums to concavity of
growth exponents is a form of tropical stability
\cite[Theorem~5 and the discussion following it]{Branden10Discrete}.
The argument above proves the exponential version for sector generators,
including nonsymmetric and reducible blocks.

Token-graph spectral theory includes the nesting of Laplacian spectra
proved by Dalf\'o and coauthors \cite{DalfoEtAl21}, and the study of
adjacency spectral radii and new Laplacian spectral layers by Reyes,
Dalf\'o, and Fiol \cite{ReyesDalfoFiol24}. The monotonicity assertions
addressed below are those for adjacency and signless-Laplacian largest
eigenvalues in \cite[Conjectures~5--6]{APS26}. The new-layer Laplacian
quantities in \cite[Definition~3.1 and Conjecture~3.3]{ReyesDalfoFiol24}
are different spectral invariants.
Jiang independently proved the same weighted concavity and monotonicity
as in Theorem~\ref{sector:token} below, using Lorentzian semigroups
\cite[Theorem~1.1]{Jiang26}.

Let \(G\) be a finite simple graph on \([n]\), with nonnegative edge
weights \(w_{ij}\). Its \(k\)-token graph has the \(k\)-subsets as
vertices. The edge from \(S\) to \(S-\{i\}+\{j\}\) has weight
\(w_{ij}\) whenever \(i\in S\), \(j\notin S\), and \(ij\) is an
edge of \(G\). Denote its adjacency and weighted degree matrices by
\(A_k(G)\) and \(D_k(G)\). The sectors \(k=0,n\) are singletons
with zero matrices.
The matrix \(A_k(G)-D_k(G)\) generates symmetric exclusion with
\(k\) particles. Borcea, Br\"and\'en, and Liggett
\cite[Proposition~5.1]{BBL09} proved that symmetric exclusion preserves
strongly Rayleigh measures, or equivalently the real stability of their
probability generating polynomials.

\begin{theorem}\label{sector:token}
For every \(a\in[-1,1]\), put
\[
 \Lambda_k(a)=\lambda_{\max}(A_k(G)+aD_k(G)).
\]
For every nonnegatively weighted \(G\),
\[
 2\Lambda_k(a)\ge\Lambda_{k-1}(a)+\Lambda_{k+1}(a)
 \qquad(1\le k<n),\qquad
 \Lambda_k(a)=\Lambda_{n-k}(a).
\]
In particular the sequence is nondecreasing up to half filling.
The interval \([-1,1]\) is the exact interval of real parameters
for which this concavity holds for every such graph.
\end{theorem}

\begin{proof}
On the binary polynomial box define, for every unordered edge,
\[
 \mathcal B_{ij}=w_{ij}\bigl[
 (z_j+az_i)\partial_i+(z_i+az_j)\partial_j
 -(z_i^2+2az_iz_j+z_j^2)\partial_i\partial_j\bigr].
\]
It annihilates \(1,z_iz_j\) and sends
\(z_i\) to \(w_{ij}(az_i+z_j)\) and
\(z_j\) to \(w_{ij}(z_i+az_j)\).
Thus \(\mathcal B=\sum_{ij}\mathcal B_{ij}\) preserves total
degree, is a real symmetric Metzler matrix, and restricts on the
degree-\(k\) sector to \(A_k(G)+aD_k(G)\).
Its mixed second-order coefficient is
\[
 -w_{ij}\bigl[(s+at)^2+(1-a^2)t^2\bigr]\le0.
\]
Theorem~\ref{gen:main-real} gives stability preservation, and
Theorem~\ref{sector:growth} yields concavity because the spectral
bound of each real symmetric sector matrix is its largest eigenvalue.

By complementing subsets we obtain a weight-preserving isomorphism
between the \(k\)- and \((n-k)\)-token graphs, proving symmetry.
Set \(\delta_k=\Lambda_k-\Lambda_{k-1}\). Concavity says that
the \(\delta_k\) are nonincreasing, while symmetry gives
\(\delta_{n-k+1}=-\delta_k\). If \(k\le(n+1)/2\), it follows that
\(\delta_k\ge-\delta_k\), proving monotonicity.

For sharpness, take the unweighted star \(K_{1,3}\), whose
two-token graph is a six-cycle. On the equal-leaf subspace, its
one-token matrix is
\(\left(\begin{smallmatrix}3a&\sqrt3\\\sqrt3&a\end{smallmatrix}\right)\);
its orthogonal complement has eigenvalue \(a\). Thus
\[
 \Lambda_1(a)=2a+\sqrt{a^2+3},\qquad
 \Lambda_2(a)=2a+2.
\]
For \(a>1\), \(\Lambda_1(a)>\Lambda_2(a)\); symmetry then
contradicts concavity at \(k=2\). For \(a<-1\), a single edge
has \(\Lambda_0=\Lambda_2=0\) and \(\Lambda_1=a+1<0\), again
violating concavity. At \(a=1\), the star has
\(\Lambda_1=\Lambda_2=4\), so connectedness does not imply
strict monotonicity.
\end{proof}

Local preservation also follows directly from the Borcea--Br\"and\'en
symbol theorem. For \(\tau\ge0\), put
\(b=e^{a\tau}\sinh\tau\) and \(c=e^{a\tau}\cosh\tau\);
the unweighted edge exponential fixes \(1,xy\) and sends
\(x,y\) to \(cx+by,bx+cy\), so its algebraic symbol is
\[
 xy+uv+b(xu+yv)+c(xv+yu).
\]
Writing this symbol as \(\zeta^{\mathsf T}M\zeta/2\) in the
order \(\zeta=(x,y,u,v)^{\mathsf T}\), one has
\[
 M=\begin{pmatrix}
 0&1&b&c\\
 1&0&c&b\\
 b&c&0&1\\
 c&b&1&0
 \end{pmatrix}.
\]
Its eigenvalues are
\[
 1+e^{(a+1)\tau},\quad 1-e^{(a+1)\tau},\quad
 -1-e^{(a-1)\tau},\quad -1+e^{(a-1)\tau}.
\]
For \(-1\le a\le1\), precisely one eigenvalue is positive,
and the quadratic form is positive on every strictly positive real
vector \(Y\). A zero at \(X+\mathrm iY\) would imply
\(X^{\mathsf T}MY=0\) and
\(X^{\mathsf T}MX=Y^{\mathsf T}MY>0\), contradicting the
one-positive-eigenvalue property. The symbol is therefore stable,
and the classical symbol theorem \cite{BB09} gives complete
preservation for each edge. The finite-dimensional product formula
and coefficientwise closure extend preservation to their sum, with
invertibility of the limiting exponential excluding zero outputs.

Taking \(a=1\) and \(a=0\) gives, respectively, monotonicity of
the largest eigenvalues of the signless Laplacian and the adjacency
matrix, as conjectured in \cite[Conjectures~5--6]{APS26}.
Theorem~\ref{sector:token} strengthens these assertions to weighted
concavity throughout the interpolation interval. At \(a=0\), the
preserving generator is already in Purbhoo's construction
\cite[Proposition~4.5]{Purbhoo18}; the additional conclusion here
compares the spectra across sectors.
The bounds in Conjectures~1--4 of \cite{APS26} were settled by
\cite{BBKL26}.

In the occupation basis, the same block matrices arise from
\[
 \mathsf B_a(G)=\frac12\sum_{ij}w_{ij}
 \bigl[X_iX_j+Y_iY_j+a(I-Z_iZ_j)\bigr].
\]
Equivalently, for the XXZ Hamiltonian
\[
 H_\Delta(G)=\frac12\sum_{ij}w_{ij}
   \bigl(X_iX_j+Y_iY_j+\Delta Z_iZ_j\bigr)
 =\mathsf B_{-\Delta}(G)+\frac\Delta2\sum_{ij}w_{ij}I,
\]
the maximum energies in the magnetization sectors
\(\sum_i Z_i=n-2k\) form a concave sequence for \(|\Delta|\le1\).
Indeed, the sector maxima are
\(\Lambda_k(-\Delta)+(\Delta/2)\sum_{ij}w_{ij}\), with an additive
constant independent of \(k\). The negative Hamiltonian has convex
sector ground-state energies. At \(a=1\), the common one-qubit
rotation \(R=e^{-\mathrm i\pi X/4}\), with
\(RYR^*=Z\) and \(RZR^*=-Y\), identifies \(\mathsf B_1(G)\) with the
EPR Hamiltonian in \cite{APS26}. A maximum-energy state therefore exists in a half-filled sector of
the rotated basis. Throughout Theorem~\ref{sector:token},
\(\Lambda_k(a)\) denotes the largest eigenvalue; it agrees with
the spectral radius when \(a\ge0\).

\appendix
\section{All-temperature Hamiltonians at higher spins}\label{spin:section}

In coherent coordinates, the generator classification applies with
\(\kappa_i=2j_i\) and extends Theorem~\ref{eq:sf} to arbitrary
finite local spins. The pair interactions satisfy the same transverse
norm inequality, while higher local spins also admit single-site
quadrupoles.

\subsection{Coherent normalization and the Hamiltonian cone}

At site $i$, fix a positive integer $\kappa_i$ and put $j_i=\kappa_i/2$.
Let $\mathcal H_{\kappa_i}=\C^{\kappa_i+1}$ have orthonormal basis
$|a_i\rangle$, $0\leq a_i\leq\kappa_i$.  For multi-indices $a$ set
\begin{equation}\label{spin:weights}
 b_a=\prod_i\binom{\kappa_i}{a_i},\qquad
 |a\rangle\longleftrightarrow \sqrt{b_a}\,z^a.
\end{equation}
Thus the monomial inner product is
$\langle z^a,z^b\rangle=\delta_{ab}/b_a$.  All adjoints in this section
refer to the physical orthonormal basis, or equivalently to this weighted
polynomial inner product.  The coherent kernel of a matrix $M$ is
\begin{equation}\label{spin:kernel}
 F_M(z,w)=\sum_{a,b}\sqrt{b_a b_b}\,M_{ab}z^a w^b.
\end{equation}
In particular, $F_I(z,w)=\prod_i(1+z_iw_i)^{\kappa_i}$.
We call $M$ coherently Lee--Yang if $F_M$ has no zero when all its variables
belong to $\D$.  

The spin matrices have the following differential realization:
\begin{equation}\label{spin:differential}
 S_i^x=\frac{(1-z_i^2)\partial_i+\kappa_i z_i}{2},\qquad
 S_i^y=\frac{(1+z_i^2)\partial_i-\kappa_i z_i}{2i},\qquad
 S_i^z=\frac{\kappa_i}{2}-z_i\partial_i.
\end{equation}
They satisfy
$(S_i^x)^2+(S_i^y)^2+(S_i^z)^2=j_i(j_i+1)I$.
Write $S_i^\perp=(S_i^x,S_i^y)^{\mathsf T}$.

\begin{theorem}\label{spin:hamiltonian}
For a Hermitian operator $H$ on $\bigotimes_i\mathcal H_{\kappa_i}$, the
following conditions are equivalent:
\begin{enumerate}
\item $e^{-\beta H}$ is coherently Lee--Yang for every $\beta\geq0$;
\item this holds for every $0\leq\beta<\varepsilon$, for some
      $\varepsilon>0$;
\item $H$ admits the expression
\begin{align}\label{spin:hamiltonian-form}
 H={}&h_0 I+\sum_i(h_i^xS_i^x+h_i^yS_i^y+h_i^zS_i^z)
       +\sum_{i:\,\kappa_i\geq2}(S_i^\perp)^{\mathsf T}B_iS_i^\perp
       \notag\\
    &+\sum_{i<j}\bigl((S_i^\perp)^{\mathsf T}J_{ij}S_j^\perp
                         +e_{ij}S_i^zS_j^z\bigr),
\end{align}
where all coefficients are real, $B_i=B_i^{\mathsf T}$, and
\begin{equation}\label{spin:hamiltonian-cone}
 h_i^z\leq0,\qquad B_i\succeq0,\qquad
 e_{ij}\leq-\|J_{ij}\|_{\mathrm{op}}.
\end{equation}
\end{enumerate}
The coefficient $(B_i)_{12}$ multiplies
$S_i^xS_i^y+S_i^yS_i^x$.  At a spin-\(1/2\) site all single-site
quadrupoles reduce to scalars or zero and are omitted from
\eqref{spin:hamiltonian-form}.
\end{theorem}

To apply Theorem~\ref{gen:main-complex}, we realize the Cayley
transformation by a unitary spin rotation.

\begin{lemma}
\label{spin:cayley}
Define the Dicke isometry
\begin{equation}\label{spin:dicke}
 W_i|a\rangle=\binom{\kappa_i}{a}^{-1/2}
                  \sum_{|x|=a}|x\rangle,\qquad W=\bigotimes_iW_i.
\end{equation}
The binary matrix kernel of $WMW^*$ is the normalized polarization of
$F_M$, separately in each row and column block.  Moreover, let
\begin{equation}\label{spin:cayley-matrix}
 C_1=\frac1{\sqrt2}\begin{pmatrix}i&-i\\1&1\end{pmatrix},\qquad
 C_i=W_i^*C_1^{\otimes\kappa_i}W_i,\qquad C=\bigotimes_iC_i.
\end{equation}
For Hermitian $H$, coherent stability of $e^{-tH}$ for all sufficiently
small $t\geq0$ is equivalent to the condition that
$\exp(-tCHC^*)$ is a complete upper-half-plane stability preserver for
every $t\geq0$.
\end{lemma}

\begin{proof}
At binary words of weights $a,b$, the entry of $WMW^*$ is
$M_{ab}/\sqrt{b_ab_b}$; identifying the variables in their respective
blocks multiplies this by $b_ab_b$ and gives \eqref{spin:kernel}.
The disk polarization theorem therefore identifies stability of the
two kernels \cite[Theorem~2.1]{BB10II}. Polarization on the symmetric
subspace, followed by binary tensor contraction and diagonal
specialization, shows that a coherently stable matrix sends a stable
coherent vector to a stable vector or zero, with the latter excluded
for a nonzero input when the matrix is invertible.

The unitary matrix $C_i$ acts on polynomials by
\begin{equation}\label{spin:cayley-action}
 f(z_i)\longmapsto
 \left(\frac{s_i+i}{\sqrt2}\right)^{\kappa_i}
 f\left(\frac{s_i-i}{s_i+i}\right),
\end{equation}
which identifies disk stability with upper-half-plane stability.
Coherent stability of $e^{-tH}$ consequently gives ordinary stability
preservation by $e^{-tCHC^*}$ on the small time interval, and hence at
all times by composition. Theorem~\ref{gen:main-complex} upgrades
this to complete preservation.

For the converse, put
$D_i=\diag_{0\leq a\leq\kappa_i}((-1)^{\kappa_i-a})$ and
$D=\bigotimes_iD_i$.  Symmetric powers of $C_1C_1^{\mathsf T}$ give
$CC^{\mathsf T}=D$, and
\begin{equation}\label{spin:universal-halfplane}
 F_D(s,t)=\prod_i(s_it_i-1)^{\kappa_i}
\end{equation}
is upper-half-plane stable, since $st=1$ is impossible for two
points of the upper half-plane. Applying a complete preserver
$N=e^{-tCHC^*}=Ce^{-tH}C^*$ in the $s$ variables produces the
coherent coefficient matrix $ND=Ce^{-tH}C^{\mathsf T}$, corresponding
to the separate Cayley transforms of the two variable groups in
$F_{e^{-tH}}$. Invertibility of $N$ makes the output nonzero, and
applying the inverse transforms gives the required disk stability.
\end{proof}

\begin{proof}[Proof of Theorem~\ref{spin:hamiltonian}]
For $A=-CHC^*$, Lemma~\ref{spin:cayley} allows us to apply
Theorem~\ref{gen:main-complex}: the generator belongs to the
degree-at-most-two part of the enveloping algebra of the local
$\mathfrak{sl}_2$ actions, and its second-order coefficients are
real and nonpositive on the appropriate real lines or planes.
Writing $A=R+iS$ in the real monomial basis, we also have
\begin{equation}\label{spin:imaginary-wedge}
 S=cI+\sum_iJ_{i,\kappa_i}(q_i),\qquad
 J_{i,\kappa_i}(q)=q\partial_i-\frac{\kappa_i}{2}q',\qquad
 q_i\in\R[z_i]_{\leq2},\quad q_i\geq0\text{ on }\R.
\end{equation}
Conversely these conditions, together with preservation of the finite
degree box, characterize the generators.

The degree-two operator space of Proposition~\ref{gen:box-decomposition}
decomposes as the direct sum of the scalar
space, the local spin spaces, the local traceless symmetric quadrupoles
when $\kappa_i\geq2$, and the spaces spanned by
$S_i^aS_j^b$ for $i<j$.  Its dimension is
\begin{equation}\label{spin:degree-two-dimension}
 1+3n+5\#\{i:\kappa_i\geq2\}+9\binom n2.
\end{equation}
Locally, this follows by decomposing the square of the
three-dimensional spin representation into scalar, antisymmetric,
and traceless symmetric parts: these give the Casimir, the spin
commutator, and the rank-two component, respectively, with the last
nonzero precisely when $\kappa_i\geq2$. Independence of the tensor
factors on distinct sites proves directness. The Cayley action is a
local spin rotation and preserves this operator space, so a Hermitian
$H$ satisfying (2) has total spin degree at most two, with the stated
spin-\(1/2\) exceptions.

Hermiticity of $A$ and the real diagonal monomial inner product give
$R^*=R$ and $S^*=-S$.  For one site write
\[
 J^-=\partial,\qquad J^0=z\partial-\kappa/2,\qquad
 J^+=z^2\partial-\kappa z.
\]
In the weighted inner product,
$(J^-)^*=-J^+$ and $(J^0)^*=J^0$.
Consequently skew-adjointness of
$aJ^-+bJ^0+cJ^+$, for real $a,b,c$, forces $b=0$ and $a=c$.
The scalar in \eqref{spin:imaginary-wedge} vanishes, and hence
\begin{equation}\label{spin:imaginary-fields}
 q_i(s)=\alpha_i(1+s^2),\quad\alpha_i\geq0,
 \qquad iJ_{i,\kappa_i}(q_i)=-2\alpha_iS_i^y.
\end{equation}
The Cayley rotations are
\begin{equation}\label{spin:rotation}
 CS_i^xC^*=-S_i^z,\qquad CS_i^yC^*=S_i^x,\qquad
 CS_i^zC^*=-S_i^y.
\end{equation}
Thus the original longitudinal field contributes $h_i^zS_i^y$ to $A$,
and \eqref{spin:imaginary-fields} gives $h_i^z=-2\alpha_i\leq0$.
Mixed longitudinal--transverse quadratic terms in $H$ give imaginary
second-order coefficients in $A$, and the independent principal
polynomials at each site prevent cancellation among distinct such
terms. Since the imaginary part has order at most one, all these
terms vanish.

For the real second-order inequalities, the principal coefficients
of $(S^x,S^y,S^z)$ in \eqref{spin:differential} are
$((1-s^2)/2,-i(1+s^2)/2,-s)$; put
\[
 u(s)=\left(\frac{2s}{1+s^2},\frac{1-s^2}{1+s^2}\right)^{\mathsf T};
\]
its range is dense in the real unit circle.  For the pair term
$(S_i^\perp)^{\mathsf T}J_{ij}S_j^\perp+e_{ij}S_i^zS_j^z$,
the coefficient $q_{ij}$ of $\partial_i\partial_j$ in $A$ satisfies
\begin{equation}\label{spin:pair-principal}
 \frac{4q_{ij}(s,t)}{(1+s^2)(1+t^2)}
       =e_{ij}-u(s)^{\mathsf T}J_{ij}u(t).
\end{equation}
Its nonpositivity is therefore equivalent to
$e_{ij}\leq-\|J_{ij}\|_{\mathrm{op}}$.
Before fixing the local Casimir gauge, write a surviving local quadratic
term as $(S_i^\perp)^{\mathsf T}K_iS_i^\perp+c_i(S_i^z)^2$, with
$K_i$ real symmetric.  For $\kappa_i\geq2$ its coefficient $q_i$ of
$\partial_i^2$ satisfies
\begin{equation}\label{spin:local-principal}
 \frac{4q_i(s)}{(1+s^2)^2}=c_i-u(s)^{\mathsf T}K_iu(s).
\end{equation}
Nonpositivity gives $c_i\leq\lambda_{\min}(K_i)$, so the Casimir
identity rewrites the term as $(S_i^\perp)^{\mathsf T}B_iS_i^\perp$
plus a scalar with $B_i=K_i-c_iI_2\succeq0$, completing necessity.

Conversely, \eqref{spin:hamiltonian-cone} makes
\eqref{spin:pair-principal} and \eqref{spin:local-principal} real and
nonpositive and makes the imaginary part satisfy
\eqref{spin:imaginary-wedge}.  The spin operators preserve the degree box, so
Theorem~\ref{gen:main-complex} and Lemma~\ref{spin:cayley} give
coherent stability at every time, proving (3)$\Rightarrow$(1).
The implication (1)$\Rightarrow$(2) follows by restriction of the
time interval.
\end{proof}

\begin{corollary}\label{spin:qubit-realization}
Every Hamiltonian in Theorem~\ref{spin:hamiltonian} has a Hermitian qubit
extension $\widehat H$, invariant under permutations within each block
of \(\kappa_i\) qubits, whose pair coefficients satisfy the qubit Lee--Yang cone and such
that
\begin{equation}\label{spin:intertwining}
 \widehat HW=WH,\qquad e^{-\beta\widehat H}W=We^{-\beta H}
 \quad(\beta\in\R).
\end{equation}
\end{corollary}

\begin{proof}
For $B_i\succeq0$ and $\lambda_i=\lambda_{\max}(B_i)$, the
Casimir relation gives
\begin{equation}\label{spin:casimir-extension}
 (S_i^\perp)^{\mathsf T}B_iS_i^\perp
 =(S_i^\perp)^{\mathsf T}(B_i-\lambda_iI_2)S_i^\perp
   -\lambda_i(S_i^z)^2+\lambda_i j_i(j_i+1)I.
\end{equation}
Replace each spin in this expression and in
\eqref{spin:hamiltonian-form} by
$\widehat S_i^a=\frac12\sum_{r=1}^{\kappa_i}\sigma_{ir}^a$.
Retain the last term of \eqref{spin:casimir-extension} as the indicated
scalar.  Interblock pairs have transverse matrix $J_{ij}/4$ and
longitudinal coefficient $e_{ij}/4$.  Intrablock pairs have transverse
matrix $(B_i-\lambda_iI_2)/2$ and longitudinal coefficient
$-\lambda_i/2$.  If $b_i=\lambda_{\min}(B_i)$, then
\[
 \|B_i-\lambda_iI_2\|_{\mathrm{op}}=\lambda_i-b_i\leq\lambda_i.
\]
All pairs therefore satisfy the qubit cone, while Pauli
anticommutation makes identical-slot symmetric off-diagonal products
vanish and leaves only scalars from the other identical-slot terms.
Each collective spin preserves the symmetric sector and intertwines
the physical spin through $W_i$, giving the first identity in
\eqref{spin:intertwining} and, by exponentiation, the second.
\end{proof}

The reduction of higher spins to spin one-half and the corresponding
anisotropic sufficient Lee--Yang theorems are classical; see, in
particular, \cite[Theorem~2 and equations~(2.38)--(2.44)]{SuzukiFisher71}.
The corollary gives explicit qubit coefficients for each Hamiltonian
in the classified cone, including the single-ion quadrupoles, with
the original operator recovered on the symmetric sector.

\section{Exact arithmetic and bit bounds}\label{arith:section}

To convert the arithmetic-operation counts in
Section~\ref{alg:section} into bit-complexity bounds, we bound the
denominators and intermediate magnitudes in the rational recursions.

\subsection{Eigenvalue continuation}\label{alg:energy-bit-details}

Use the notation of Section~\ref{alg:section} and the coefficient
order \(p\) chosen in the proof of Theorem~\ref{alg:complex}.
Let \(B_0\) bound the bit length of each numerator and denominator
in the original local matrices.
For \(p>0\), the interpolation nodes
\(s_\ell=-1+2\ell/p\) have \(O(\log(p+1))\) bits.
Taking local Hermitian parts and forming \(B+s_\ell C\)
therefore produces Gaussian-rational inputs with bit lengths
polynomial in \(B_0\) and \(\log(p+1)\).

For a Hermitian coefficient computation at one of these nodes, let
\(Q_0\) be the product of the positive denominators of the real
and imaginary parts of the local entries. There are \(O_D(N)\)
such entries, making \(\log Q_0\) polynomial in the input length;
set
\begin{equation}
 K_p=48(p+1)!,\qquad
 D_j=(Q_0K_p)^j\quad(0\le j\le p).
 \label{alg:common-denominator}
\end{equation}
In the coefficient recursion of \cite[Sections~3--5]{BDL08}, an
order-\(j\) creation coefficient \(C_j(M)\) vanishes unless
\(M\) lies in a connected set of at most \(j+1\) vertices.  Each nonzero term in
its recursion contains one local matrix entry, previous creation
coefficients of total order \(j-1\), and a divisor
\[
 k!\,E_0(M),\qquad k\le4,\qquad E_0(M)=2|M|.
 \]
The basis creation matrices
\(\prod_{u\in M}|1\rangle\langle0|_u\) have entries zero or one;
expanding their nested commutators therefore produces signed sums
of local perturbation entries.
The energy recursion has the same form with the outer divisor
\(E_0(M)\) omitted, so the two recursions admit a common
order-by-order denominator bound.

Since \(k!E_0(M)\) divides \(K_p\) for \(j\le p\), induction
in the perturbation order gives
\begin{equation}
 D_j C_j(M)\in\mathbb Z[i],
 \qquad D_j e_j(B+s_\ell C)\in\mathbb Z[i].
 \label{alg:denominator-induction}
\end{equation}
Indeed, the previous coefficients contribute denominator
dividing \((Q_0K_p)^{j-1}\), the new local entry contributes
a factor dividing \(Q_0\), and the remaining divisor divides
\(K_p\).  All terms at a fixed order may be accumulated over
this common denominator.  Its bit length is
\[
 O\bigl(p\log Q_0+p^2\log(p+1)\bigr).
 \]

For the numerators, fully unfold a term of perturbation order at
most \(p\). It has at most \(p\) local insertions and
bounded-arity recursion nodes, whose local edges can be chosen in
at most \((ND)^{O(p)}\) ways. Their union has \(O(p)\) vertices
and contains all intermediate creation sets; choosing those sets
at \(O(p)\) nodes, together with the tree shape, order partitions,
and bounded-depth commutator terms, gives at most
\[
 \exp\bigl(O_D(p^2+p\log(N+1))\bigr)
 \]
fully unfolded summands.  Each summand contains at most \(p\)
bounded local entries; the positive integer energy and factorial
divisors cannot increase its absolute value.
Its magnitude is therefore at most
\(\max(1,2J)^{O(p)}\) times a fixed exponential factor.
After introducing \eqref{alg:common-denominator}, this bounds
the logarithmic magnitude of every intermediate numerator by
a polynomial in \(p,N,B_0\).
The tree expansion bounds the sizes of the integers stored by the
recursion, whose operation count remains \(N\exp(O_D(p))\).

The interpolation step has polynomial bit cost as well: at the nodes
\(s_\ell=-1+2\ell/p\), its weights for evaluation at \(i\) are
\begin{equation}
 L_\ell(i)=
 \frac{\displaystyle
       \prod_{\substack{0\le r\le p\\r\ne\ell}}
           (pi+p-2r)}
      {2^p(-1)^{p-\ell}\ell!(p-\ell)!}.
 \label{alg:interpolation-weights}
\end{equation}
Each numerator and denominator has \(O(p\log(p+1))\) bits, and
\(e_k(V)\) is obtained by multiplying the weights by the computed
node values and summing \(p+1\) terms.
The rational powers and binomial coefficients in
\eqref{alg:coefficient-transform}, and the powers in
\eqref{alg:stopping}--\eqref{alg:output}, increase bit length
only polynomially.
Thus a safe total bit bound is
\[
 N\exp(O_D(p))\operatorname{poly}(N,B_0,p,L_{\mathrm{in}}),
 \]
where \(L_{\mathrm{in}}\) includes the requested accuracy,
establishing the polynomial bit overhead in
Theorem~\ref{alg:complex}.

\subsection{Connected principal-state tests}\label{alg:test-bit-details}

We use the subsystem coefficients and notation from the proof of
Lemma~\ref{alg:connected-tests}.  Let \(Q_0\) be a common positive denominator for all local
entries of \(V\), let \(Q_B\) be the product of the denominators
of the entries of the \(B_i\), and put
\[
 D_p=2(p+1)!,\qquad T=Q_0D_p .
 \]
In a subsystem of size at most \(p+1\), each nonzero diagonal
entry of the reduced inverse \(S_U\) is \(1/(2r)\),
\(1\le r\le p+1\).  Induction in the vector recursion gives
denominators dividing \(T^{2k}\) for \(v_{U,k},w_{U,k}\)
and \(T^{2k-1}\) for the energy coefficient.
It follows that the denominator of \(a_k\) divides
\(Q_BT^{2k}\), and \eqref{alg:formal-test-log} bounds that of
\(f_k\) by \(k!Q_B^kT^{2k}\). Accumulating the sums at each
order over these common denominators uses a number of bits
polynomial in \(p\) and the total input length.

For the intermediate numerators, use
\(\|V_U\|\le m(d+1)J\) on a subsystem \(U\). When
\(|x|\le r_U=[16(1+m(d+1)J)]^{-1}\), the resolvent Neumann
series on \(|z|=1\) bounds the vacuum spectral projection and
its difference from the unperturbed projection.  In particular, the projection has
norm at most \(16/15\), its vacuum entry differs from one by
at most \(1/15\), and both vacuum-normalized eigenvectors have
norm less than two.  Cauchy's estimate gives
\(\|v_{U,k}\|,\|w_{U,k}\|\le2r_U^{-k}\).
As \(\|B_U\|\le2^m\),
\[
 |a_k|\le4(k+1)2^m r_U^{-k}.
\]
For \(k\le p\), choose \(A_p=4(p+1)2^{p+1}\).
The finite formula
\[
 \log(1+a)=\sum_{\ell=1}^p(-1)^{\ell+1}a^\ell/\ell
 \pmod{x^{p+1}}
\]
bounds \(|f_k|\) by \((2A_p)^kr_U^{-k}\).
Matrix products, vector sums and the partial sums in
\eqref{alg:formal-test-log} have the same type of bound, up to
factors \(2^{O(p)}\operatorname{poly}(p)\).
The subset sums contain at most \(2^{p+1}\) terms, and the final
connected sum contains at most \(n\exp(O_d(p))\), so these
magnitude estimates, combined with the common denominators, bound
the bit length of every stored numerator, including before
cancellation. The claimed bit complexity follows.

\section{Connected preservers and the obstruction to flow generation}
\label{glob:section}

The classification in Theorems~\ref{gen:main-real} and
\ref{gen:main-complex} describes one-parameter preserving semigroups.
We compare their products with all invertible preservers on the
multiaffine space
$V_n=\C[z_1,\ldots,z_n]_{\mathrm{ma}}\simeq(\C^2)^{\otimes n}$.
Let $S_n$ be the semigroup of invertible complex-linear operators that
preserve disk stability.  For its ordinary monomial coefficient matrix
$M$, write
\begin{equation}\label{glob:symbol}
 G_M(z,w)=\sum_{a,b\in\{0,1\}^n}M_{ab}z^aw^b.
\end{equation}
For \(M\in\operatorname{GL}(V_n)\), the symbol criterion is
\begin{equation}\label{glob:symbol-equivalence}
 M\in S_n\quad\Longleftrightarrow\quad
 G_M\text{ is stable on }\D^{2n}.
\end{equation}
For necessity, fix $w\in\D^n$ and apply $M$ to
$\prod_i(1+w_iz_i)$, whose image is stable and nonzero by
invertibility. Conversely, tensor contraction of $G_M$ with a stable
input gives preservation, with zero again excluded by invertibility;
this is the disk-symbol convention used throughout the paper.
Hurwitz's theorem makes $S_n$ closed relative to
$\operatorname{GL}(V_n)$, since an invertible limit cannot have zero
symbol, and $S_n$ is also invariant under nonzero scalar multiplication.

Define
\begin{equation}\label{glob:semigroups}
 \begin{split}
 W_n&=\{A:e^{tA}\in S_n\text{ for every }t\geq0\},\\
 E_n&=\langle e^A:A\in W_n\rangle,\qquad
 T_n=\overline{E_n}^{\,\operatorname{GL}(V_n)},\qquad
 H_n=S_n\cap S_n^{-1}.
 \end{split}
\end{equation}
Thus $E_n$ consists of finite products of admissible flows and $T_n$
is its relative closure, while $H_n$ is the group of reversible
preservers, with identity component $H_n^0$.

\begin{theorem}
\label{glob:main}
For every $n\geq1$, $S_n$ is path connected.  Nevertheless,
\begin{equation}\label{glob:reversible-parts}
 E_n\cap H_n=H_n^0,
 \qquad T_n\cap T_n^{-1}=H_n^0.
\end{equation}
For $n\geq2$, every nontrivial coordinate permutation $\pi$ satisfies
\begin{equation}\label{glob:permutation-obstruction}
 \pi\in S_n\setminus T_n.
\end{equation}
Thus a positive-radius matrix neighborhood of $\pi$ contains no finite
product of admissible one-parameter flows, even with unbounded product
lengths and generator norms.
\end{theorem}

\begin{proof}[Path connectedness]
Let $D_rz^a=r^{|a|}z^a$, $0<r\leq1$.  It is an invertible preserver,
and
\begin{equation}\label{glob:sandwich}
 G_{D_rMD_r}(z,w)=G_M(rz,rw).
\end{equation}
Since $M_{00}\ne0$ for $M\in S_n$, a path of nonzero scalars
normalizes it to one, after which $D_rMD_r$ converges to the vacuum
projector $|0\rangle\langle0|$ as $r\downarrow0$.

Consider the convex open set in the complex affine hyperplane $X_{00}=1$
given by
\begin{equation}\label{glob:coefficient-ball}
 \mathcal B=\left\{X:X_{00}=1,
        \ \sum_{(a,b)\ne(0,0)}|X_{ab}|<1\right\}.
\end{equation}
Every member has a kernel nonvanishing on the closed polydisk, by strict
constant-term domination.  For small positive $r,s$, both $D_rMD_r$
and $D_s$ belong to $\mathcal B$ and are invertible.  Indeed the
nonconstant coefficient sum for $D_s$ is $(1+s)^n-1$.

The invertible part of $\mathcal B$ is path connected.  To see this,
take invertible $X,Y\in\mathcal B$ and set
$L(\zeta)=(1-\zeta)X+\zeta Y$.  Its determinant is a nonzero
polynomial in $\zeta$, since it is nonzero at zero.  The real segment
$L([0,1])$ is compactly contained in $\mathcal B$, so a sufficiently
thin complex neighborhood of $[0,1]$ still maps into $\mathcal B$.
A path from $0$ to $1$ in this neighborhood can avoid the finitely
many determinant zeros, and its image connects $X$ to $Y$ through
invertible members of $\mathcal B$. We may therefore connect $M$
to $D_rMD_r$ by \eqref{glob:sandwich}, then to $D_s$ inside
$\mathcal B$, and finally to $I$ by increasing $s$ to one, all
through invertible preservers.
\end{proof}

An explicit path for two variables is obtained from the symmetric
and antisymmetric projections $P_\pm$ on
$\operatorname{span}(z_1,z_2)$. Let $U_\theta$ act as
$P_++e^{i\theta}P_-$ there and as the identity on
$\operatorname{span}(1,z_1z_2)$.  For $r=1/4$,
\begin{align}\label{glob:swap-path}
 G_{D_rU_\theta}
 ={}&1+ra_\theta(z_1w_1+z_2w_2)
       +rb_\theta(z_1w_2+z_2w_1)
       +r^2z_1z_2w_1w_2,\notag\\
 a_\theta={}&(1+e^{i\theta})/2,\qquad
 b_\theta=(1-e^{i\theta})/2.
\end{align}
Its nonconstant coefficient sum is at most
$2\sqrt2r+r^2<1$, and its determinant is $r^4e^{i\theta}\ne0$.
The paths $I\to D_r$, $D_rU_\theta$ for $0\leq\theta\leq\pi$,
and $D_s\pi$ for $r\leq s\leq1$ connect the identity to the
coordinate exchange.

\begin{lemma}\label{glob:unit-factors}
If $A,B\in S_n$ and $AB\in H_n$, then $A,B\in H_n$.
Moreover every element of $H_n^0$ is a nonzero scalar times a tensor
product of one-site invertible matrices, whereas no nontrivial coordinate
permutation has this form.
\end{lemma}

\begin{proof}
The identities $A^{-1}=B(AB)^{-1}$ and $B^{-1}=(AB)^{-1}A$ prove
the first assertion.  As a closed subgroup of $\operatorname{GL}(V_n)$, $H_n$ is a real
Lie group with Lie algebra $W_n\cap(-W_n)$, which in half-plane
coordinates is, by Corollary~\ref{gen:gram-cones},
\begin{equation}\label{glob:unit-algebra}
 \C I+\bigoplus_i\mathfrak{sl}_2(\R)_i.
\end{equation}
The fixed tensor Cayley transformation preserves products and
coordinate permutations, so the fact that a connected Lie group is
generated by its Lie-algebra exponentials applies equally in disk
coordinates. Exponentiating \eqref{glob:unit-algebra} and taking
finite products gives only scalar multiples of local tensor products.
A nontrivial coordinate permutation cannot have this form: varying one
input factor varies a different output factor, whereas an invertible
local tensor product varies only the corresponding output factor.
Equivalently, for a two-factor exchange, an identity
$Ax\otimes By=c\,y\otimes x$ with fixed nonzero $x$ and varying
independent $y$ would force the fixed vector $Ax$ to be proportional
to every $y$.
\end{proof}

\begin{proof}[The exact-product assertion in Theorem~\ref{glob:main}]
If $e^{A_m}\cdots e^{A_1}\in H_n$, with every $A_j\in W_n$,
Lemma~\ref{glob:unit-factors} makes every exponential factor a unit.
For $0\leq s\leq1$, both factors in
$e^{A_j}=e^{(1-s)A_j}e^{sA_j}$ preserve stability, so the same
lemma puts $e^{sA_j}$ in $H_n$. Each exponential, and hence their
product, therefore belongs to $H_n^0$. Conversely, $H_n^0$ is
generated by exponentials of its Lie algebra contained in $W_n$,
which gives $E_n\cap H_n=H_n^0$.
\end{proof}

For the closure assertion we use Neeb's classical connected-unit theorem:
if $G$ is a connected Lie group and $T\subseteq G$ is a closed
submonoid satisfying
\[
 T=\overline{\langle\exp L(T)\rangle}^{\,G},\qquad
 L(T)=\{A:\exp(tA)\in T\text{ for every }t\geq0\},
\]
then $T\cap T^{-1}$ is connected. This is
\cite[Definition~II.1(3) and Proposition~III.2, pp.~171, 178]{Neeb92};
see also \cite[Section~6.1]{SanMartinSantana02}.

\begin{proof}[The closure assertion in Theorem~\ref{glob:main}]
In the connected real Lie group $G=\operatorname{GL}_{2^n}(\C)$,
$T_n$ is a closed submonoid contained in $S_n$. Every $A\in W_n$
generates a one-parameter semigroup in $E_n$, whereas
$L(T_n)\subseteq L(S_n)=W_n$, so
\[
 L(T_n)=W_n,\qquad
 T_n=\overline{\langle\exp L(T_n)\rangle}^{\,G}.
\]
Neeb's theorem therefore makes $T_n\cap T_n^{-1}$ a connected
subgroup of $H_n$, so it is contained in $H_n^0$.
The reverse inclusion follows because the exponentials of both signs of
\eqref{glob:unit-algebra} belong to $T_n$ and generate $H_n^0$.

If a nontrivial coordinate permutation $\pi$ belonged to $T_n$,
its finite order would also put $\pi^{-1}$ in $T_n$, making it a
unit and hence a member of $H_n^0$. This contradicts
Lemma~\ref{glob:unit-factors} and establishes
\eqref{glob:permutation-obstruction}. Since relative closure in the
general linear group agrees with matrix-norm closure at the invertible
point $\pi$, the excluded neighborhood follows as well.
\end{proof}

\begin{corollary}\label{glob:local-divisibility}
Call $M\in S_n$ exactly locally divisible if, for every identity
neighborhood $U$ in $\operatorname{GL}(V_n)$, it is a finite product
of members of $S_n\cap U$.  The exactly locally divisible units of
$S_n$ are precisely $H_n^0$.  In addition, if a continuous path
$M:[0,1]\to S_n$ starts at $I$ and satisfies
$M(t)M(s)^{-1}\in S_n$ for $s\leq t$, then a unit endpoint must
belong to $H_n^0$.
\end{corollary}

\begin{proof}
The identity component of a Lie group is open, so choose $U$ with
$H_n\cap U\subset H_n^0$.  Any exact factorization of a unit into
$S_n\cap U$ consists entirely of units by
Lemma~\ref{glob:unit-factors}, hence has endpoint in $H_n^0$.
Conversely, subdividing a finite exponential representation of a member
of $H_n^0$ puts all its factors in any prescribed identity neighborhood.
For the path assertion, a unit endpoint in the factorization
$M(1)=[M(1)M(t)^{-1}]M(t)$ makes $M(t)$ a unit for every $t$,
so continuity places the whole path in $H_n^0$.
\end{proof}

\begin{remark}\label{glob:scope}
The identity for the closure in \eqref{glob:reversible-parts} is
$T_n\cap T_n^{-1}=H_n^0$, concerning elements whose inverses also
belong to $T_n$, as finite order ensures for permutations. Static
apolar maps can lie beyond these flows of the infinitesimal cone:
for $F(s,t)=s-t$,
$B_F=\pi-I$, so $I+B_F=\pi$, whereas
\[
 e^{tB_F}=\frac{1+e^{-2t}}2I+\frac{1-e^{-2t}}2\pi.
\]
Allowing such a static permutation as an additional primitive changes
the factorization question.
\end{remark}

\section{Zero-free radii and coefficient norms}
\label{spec:examples-section}

Although the gap bound in Theorem~\ref{spec:gap} is independent of
the undamped generator and the coordinate degrees, the strict radius
and the norm in Corollary~\ref{spec:resolvent} depend on the operator,
as the following examples show.

\begin{example}\label{spec:no-uniform-radius}
Fix \(\mu,t>0\).  On \(V_1\), in the coefficient basis \((1,z)\),
let \(X\) and \(Z\) be the usual Pauli matrices and put
\[
 A_{0,R}=RX,\qquad
 A_R=A_{0,R}-2\mu E=
 \begin{pmatrix}0&R\\R&-2\mu\end{pmatrix},\qquad R>0.
\]
With \(c=\cosh(Rs)\) and \(b=\sinh(Rs)\), the undamped flow
acts by the weighted substitution
\[
 (e^{sA_{0,R}}p)(z)=(c+bz)
 p\!\left(\frac{b+cz}{c+bz}\right).
\]
Since \(c^2-b^2=1\), the prefactor is nonzero on the disk and
the fraction is a disk automorphism, giving preservation of disk
stability even with arbitrary inert variables. Put \(\omega_R=\sqrt{R^2+\mu^2}\).  The identity
\((A_R+\mu I)^2=\omega_R^2I\) gives
\[
 e^{tA_R}=e^{-\mu t}\left[
 \cosh(\omega_Rt)I+
 \frac{\sinh(\omega_Rt)}{\omega_R}(A_R+\mu I)\right].
\]
Thus the image of the fixed input \(1\) has its unique zero at
\[
 z_R=-\frac{\omega_R\coth(\omega_Rt)+\mu}{R}.
\]
For every finite \(R\), \(|z_R|>1\), but \(|z_R|\to1\) as
\(R\to\infty\).  Since \(K_{e^{tA_R}}(z,0)=e^{tA_R}1(z)\),
the same zero obstructs a common kernel radius depending only on
\(\mu,t\), already for Hermitian generators on one binary variable.
The dominant eigenvector has its zero at modulus
\((\omega_R+\mu)/R\), which also tends to one.
\end{example}

\begin{example}\label{spec:nonnormal}
Fix \(\mu>0\), and on \(V_1\) put
\[
 G_R=R(X+iZ),\qquad
 B_R=G_R-2\mu E=
 \begin{pmatrix}iR&R\\R&-iR-2\mu\end{pmatrix},\qquad R>0.
\]
Here \(G_R^2=0\) and \(G_R^*Z+ZG_R=0\).  Its stability-preserving
action is explicit: with \(a=1+iRs\) and \(b=Rs\),
\[
 (e^{sG_R}p)(z)=(a+bz)
 p\!\left(\frac{b+\overline a z}{a+bz}\right),\qquad
 |a+bz|^2-|b+\overline a z|^2=1-|z|^2.
\]
The nonzero denominator and the disk-preserving fraction give
stability preservation, including in the presence of inert variables.

Choose the square root \(\omega_R=\sqrt{\mu^2+2i\mu R}\) with
\(x_R=\Re\omega_R>0\).  The two eigenvalues of \(B_R\) are
\(a_\pm=-\mu\pm\omega_R\), and
\[
 x_R=\sqrt{\frac{\mu\sqrt{\mu^2+4R^2}+\mu^2}{2}}\ge\mu.
\]
The dominant spectral projection and normalized exponential are
\[
 P_R=\frac12\left(I+\frac{B_R+\mu I}{\omega_R}\right),\qquad
 S_R(t)=e^{-ta_+}e^{tB_R}
       =P_R+e^{-2\omega_Rt}(I-P_R).
\]
Their off-diagonal entries give
\[
 \|P_R\|_2\ge\frac{R}{2|\omega_R|},\qquad
 |[S_R(t)]_{10}|
 =\frac{R}{2|\omega_R|}|1-e^{-2\omega_Rt}|
 \ge\frac{R}{2|\omega_R|}(1-e^{-2x_Rt}).
\]
Since \(|\omega_R|=(\mu^4+4\mu^2R^2)^{1/4}\), both quantities
diverge as \(R\to\infty\), the second for every fixed \(t>0\).

To keep the actual spectral gap fixed, pass to \(V_{(1,1)}\) and set
\[
 \mathcal A_R=G_R\otimes I-2\mu(E_1+E_2)
             =B_R\otimes I+I\otimes(-2\mu E).
\]
The displayed weighted substitution still makes the undamped part
a joint stability generator, and the four eigenvalues are
\(a_+,a_-,a_+-2\mu,a_--2\mu\). Thus \(a_+\) is dominant with
real-part gap \(2\mu\) and projector
\(\mathcal P_R=P_R\otimes P_0\).
For the fixed strictly stable input \(p(z_1,z_2)=1+z_2/2\), define
\(\mathcal R_R(t)=e^{-ta_+}e^{t\mathcal A_R}-\mathcal P_R\).
Its \(z_1z_2\) coefficient is
\[
 [z_1z_2]\mathcal R_R(t)p
 =\frac12e^{-2\mu t}\frac{R}{2\omega_R}
   (1-e^{-2\omega_Rt}).
\]
The projector contributes no \(z_2\) term, so the identity follows
from the tensor exponential, and its modulus diverges for every
fixed \(\mu,t>0\). Even after subtraction of the dominant
projector, a coefficient-norm mixing bound therefore cannot depend
only on damping, time, and the spectral gap: the obstruction occurs
in a fixed two-variable box on one fixed strictly stable input.
Equivalence of finite-dimensional norms gives the same obstruction
for every fixed coefficient norm, whereas the operator-dependent
norm of Corollary~\ref{spec:resolvent} retains the contraction bound.
\end{example}

\section{An exact counterexample to a stronger ground-state radius}\label{cx:section}

The qualitative strict radius in Theorem~\ref{eq:strict-field} cannot
in general be replaced by the proposed explicit radius for deformed
EPR ground states. We address the stronger clause of Conjecture~1 in
\cite[Section 5.4]{WBG26}, with its stated unweighted-graph hypothesis.

The computations in \cite[Numerical Study~1, equation~(70)]{WBG26}
test the signed diagonal specializations $z_i=(-1)^{y_i}z$,
$y\in\{0,1\}^n$, of the ground-state polynomial. The zero certified
below lies outside every such slice: its two fixed coordinate values
are neither equal nor opposite. The certificate therefore concerns
points outside the family tested in that numerical study.

\begin{theorem}\label{cx:epr}
Let $G$ have vertices $\{0,1,2,3,4,5\}$ and edges
\[
 \{01,02,03,14,15\}.
\]
Put
\[
 H=-\sum_{ij\in E(G)}
 (|00\rangle+\tfrac12|11\rangle)
 (\langle00|+\tfrac12\langle11|)_{ij}.
\]
The ground state $\psi$ is unique up to scalar. Its multiaffine
polynomial $f_\psi(z)=\sum_b\psi_bz^b$ has a zero satisfying
\[
 |z_i|<\frac{499}{500}\sqrt2\qquad(0\leq i\leq5).
\]
Thus the proposed radius $s^{-1/2}$ fails at $s=1/2$ on an
unweighted six-vertex tree.
\end{theorem}

\begin{proof}
The proof uses rational spectral and zero certificates, whose
inequalities are verified by the finite procedure supplied at the
end of this appendix using only integer and fraction arithmetic.

Write $K=-H$ and index its computational basis by
$b\in\{0,\ldots,63\}$, with bit $b_i$ assigned to vertex $i$.
For each edge $ij$ and each input with $b_i=b_j=a$, changing both
bits to $c\in\{0,1\}$ contributes $2^{-(a+c)}$ to the corresponding
entry of $K$, with every other contribution from that edge zero;
these rules specify the symmetric rational matrix completely.

For $M=37I-8K$, the conserved charge
\[
 Q(b)=b_0+b_4+b_5-b_1-b_2-b_3
\]
gives, in charge order $-3,-2,\ldots,3$, blocks of dimensions
\[
 1,\ 6,\ 15,\ 20,\ 15,\ 6,\ 1.
\]
Ordering each block by increasing integer $b$, rational symmetric
elimination gives the following pivot signs:
\[
\begin{array}{c|rrrrrrr}
 Q&-3&-2&-1&0&1&2&3\\ \hline
 \text{negative pivots}&0&0&0&1&0&0&0\\
 \text{positive pivots}&1&6&15&19&15&6&1.
\end{array}
\]
No pivot is zero, and the unique negative pivot is the first in the
charge-zero block; the recurrence at step $i$ is
\[
 p_i=M^{(i)}_{ii},\qquad
 M^{(i+1)}_{jk}
 =M^{(i)}_{jk}-M^{(i)}_{ji}M^{(i)}_{ik}/p_i
 \quad(j,k>i).
\]
Rational comparison determines the pivot signs, so Sylvester's law
of inertia puts one eigenvalue of $K$ above $37/8$ and all others
below it. Since $K_{00}=5$, this is the largest eigenvalue, which
is simple and at least $5$.

Define the rational vector $v$ by assigning to each bit string the key
\[
 (a,b,i,j)=(b_0,b_1,b_2+b_3,b_4+b_5).
\]
The amplitude for each individual bit string is the following
numerator divided by $10^{15}$; a missing key has amplitude zero.
\begin{center}
\begin{tabular}{cr}
\toprule
Key & Numerator\\
\midrule
$(0,0,0,0)$ & $1000000000000000$\\
$(0,0,1,1)$ & $8942179087130$\\
$(0,0,2,2)$ & $136748346620$\\
$(0,1,0,1)$ & $171232824560490$\\
$(0,1,1,2)$ & $2452435031309$\\
$(1,0,1,0)$ & $171232824560490$\\
$(1,0,2,1)$ & $2452435031309$\\
$(1,1,0,0)$ & $97255575326164$\\
$(1,1,1,1)$ & $37858301388951$\\
$(1,1,2,2)$ & $1200949534669$\\
\bottomrule
\end{tabular}
\end{center}
These are computational-basis amplitudes, not normalized orbit-basis
coordinates. Put
\[
 \widehat\lambda=\frac{539109343678406}{10^{14}}.
\]
Direct rational summation gives
\[
 \|(K-\widehat\lambda I)v\|_2^2<10^{-24},
 \qquad \widehat\lambda-37/8>1/2.
\]
Let $v_g$ be the orthogonal projection onto the largest eigenspace
of $K$. Expanding in an orthonormal eigenbasis and using the
preceding separation gives
\[
 \|v-v_g\|_2
 \leq2\|(K-\widehat\lambda I)v\|_2<2\cdot10^{-12}.
\]
Since $v_{000000}=1$, the vector $v_g$ is nonzero and hence a
scalar multiple of the ground state; both $v$ and $v_g$ have
charge zero and therefore even parity.

For an even-parity vector $w$ put
\[
 \widetilde f_w(x)=f_w(\sqrt2\,x)
        =\sum_b2^{|b|/2}w_bx^b.
\]
For the rational-coefficient polynomial $\widetilde f_v$, set
\[
 c=\frac{175105+981180\,\mathrm i}{10^6},\qquad
 \ell=\frac{-379834+921467\,\mathrm i}{10^6}.
\]
Squaring gives $|c|,|\ell|<499/500$, and fixing $x_0=x_1=c$
and $x_2=x_3=x_4=\ell$ reduces the polynomial by multiaffinity to
\[
 \widetilde f_v(c,c,\ell,\ell,\ell,x_5)=A_v+B_vx_5,
 \qquad
 |A_v|^2<(85/1000)^2,\quad |B_v|^2>(852/10000)^2 .
\]
The displayed inequalities are obtained by finite sums over the
amplitude table. Each coefficient $A$ or $B$ contains at most
$32$ terms, whose fixed-variable monomials have modulus at most
one and whose even tilt factors are at most eight, so
\[
 |A_{v_g}-A_v|,\ |B_{v_g}-B_v|
 <32\cdot8\cdot2\cdot10^{-12}<10^{-8}.
\]
In particular $B_{v_g}\neq0$, and
\[
 x_5=-A_{v_g}/B_{v_g},\qquad
 |x_5|<
 \frac{85/1000+10^{-8}}{852/10000-10^{-8}}<499/500.
\]
The final rational inequality places every coordinate of this zero
of $\widetilde f_{v_g}$ below $499/500$ in modulus; rescaling by
$\sqrt2$ gives the required zero of the ground-state polynomial.
\end{proof}

The residual bound transfers the zero certificate from the rational
vector to the ground-state polynomial, placing its radius strictly
below the proposed \(s^{-1/2}\) while leaving the qualitative
strict radius in Theorem~\ref{eq:strict-field} intact.

\subsection*{Reproducible rational verification}
The supplementary Python procedure
\path{epr_six_vertex_certificate.py} reproduces the inertia, residual,
and zero inequalities used above from the displayed rational data.
It uses only standard-library integer and fraction arithmetic;
Gaussian-rational operations are written explicitly.
The supplementary procedure \path{epr_independent_certificate.py}
provides a second exact verification. Its floating-point diagnostics
are separate from the rational assertions used in the proof.
Both procedures are supplied with execution instructions and are
available at the fixed repository version identified below.

\section*{Data and code availability}
All rational certificate data are included in Appendix~\ref{cx:section}.
The two verification procedures are supplied as supplementary material
and available in the \texttt{certificates} directory of a
\href{https://github.com/Isaac-Wei-2026/stability-preserving-semigroups/tree/85527775f5de8af8911941fc422dbe140a63d11a/certificates}{fixed GitHub version}
(certificate snapshot \texttt{8552777}).
The procedures use Python~3 and its standard library.

\section*{Disclosure of artificial intelligence use}
Most of the mathematical content, proofs, computational certificates,
and text in this article were generated with substantial assistance from
OpenAI GPT-6 Pro in ChatGPT chat mode and GPT-6 Astra in Codex. Claude Fable 5.1 in chat mode also
provided several rounds of revision suggestions.
The author directed the research and the selection of material, and
participated in restructuring and rewriting the manuscript; AI tools
were also used for literature searches and proof checking.
The author is responsible for its content.

\raggedbottom
\makeatletter\def\@biblabel#1{#1.}\makeatother
\providecommand{\bysame}{\leavevmode\hbox to3em{\hrulefill}\thinspace}
\providecommand{\MR}{\relax\ifhmode\unskip\space\fi MR }
\providecommand{\MRhref}[2]{%
  \href{http://www.ams.org/mathscinet-getitem?mr=#1}{#2}
}
\providecommand{\href}[2]{#2}

\end{document}